\documentclass[11pt,reqno,a4paper]{amsart}
\usepackage{amsmath,amsthm,verbatim,amscd,amssymb,setspace,enumitem,hyperref}
\usepackage{exscale,color}

\usepackage{enumitem}
\usepackage{mathrsfs}

\newtheorem{innercustomthm}{Theorem}
\newenvironment{customthm}[1]
  {\renewcommand\theinnercustomthm{#1}\innercustomthm}
  {\endinnercustomthm}

\renewcommand{\epsilon}{\varepsilon}
\newcommand{\N}{\mathbb{N}}

\newcommand{\R}{\mathbb{R}}
\newcommand{\C}{\mathbb{C}}

\renewcommand{\Re}{\operatorname{Re}}

\newcounter{mtheorem}
\newtheorem{mtheorem}[mtheorem]{Theorem}

\newtheorem{mcor}[mtheorem]{Corollary}

\newcommand{\supp}{\operatorname{supp}}

\newcommand{{\vol}}{\rm vol}
\newcommand{\p}{\partial}
\newcommand{\norm}[1]{\Vert #1 \Vert}

\newcommand{\Ric}{\operatorname{Ric}}

\newcommand{\Rm}{\operatorname{Rm}}

\providecommand{\norm}[1]{\lVert#1\rVert}

\def\tr{\operatorname{tr}}
\def\cit{\mathop{\rm \mathscr{C}}\nolimits}
\def\inj{\operatorname{inj}}
\def\eucl{\operatorname{eucl}}
\def\AVR{\operatorname{AVR}}
\def\Div{\operatorname{div}}
\def\Id{\operatorname{Id}}

\newtheoremstyle{fancy}{}{}{\itshape}{}{\textbf\bgroup}{.\egroup}{ }{}
\newtheoremstyle{fancy2}{}{}{\rm}{}{\textbf\bgroup}{.\egroup}{ }{}

\theoremstyle{fancy}
\newtheorem{theorem}{Theorem}[section]
\newtheorem{lemma}[theorem]{Lemma}
\newtheorem{corollary}[theorem]{Corollary}
\newtheorem{prop}[theorem]{Proposition}

\theoremstyle{fancy2}
\newtheorem{definition}[theorem]{Definition}
\newtheorem{example}[theorem]{Example}
\newtheorem{remark}[theorem]{Remark}
\newtheorem{convention}[theorem]{Convention}
\newtheorem{claim}[theorem]{Claim}

\setlist{leftmargin=*}

\numberwithin{equation}{section}

\begin{document}
\title{Expanding K\"ahler-Ricci solitons coming out of K\"ahler cones}
\date{\today}
\author{Ronan J.~Conlon}
\address{Department of Mathematics and Statistics, Florida International University, Miami, FL 33199, USA}
\email{rconlon@fiu.edu}
\author{Alix Deruelle}
\address{D\'epartement de math\'ematiques, B\^atiment 425, Facult\'e des sciences, Universit\'e
Paris-Sud 11, F-91405, Orsay, France}
\email{alix.deruelle@math.u-psud.fr}
\date{\today}

\begin{abstract}
We give necessary and sufficient conditions for a K\"ahler equivariant resolution of a K\"ahler cone, with the resolution satisfying one of a number of auxiliary conditions, to admit a unique asymptotically conical (AC) expanding gradient K\"ahler-Ricci soliton. In particular, it follows that for any $n\in\mathbb{N}_{0}$ and for any negative line bundle $L$ over a compact K\"ahler manifold $D$, the total space of the vector bundle $L^{\oplus (n+1)}$ admits a unique AC expanding gradient K\"ahler-Ricci soliton with soliton vector field a positive multiple of the Euler vector field if and only if $c_{1}(K_{D}\otimes(L^{*})^{\otimes (n+1)})>0$. This generalises the examples already known in the literature. We further prove a general uniqueness result and show that the space of certain AC expanding gradient K\"ahler-Ricci solitons on $\mathbb{C}^{n}$ with positive curvature operator on $(1,\,1)$-forms is path-connected.
\end{abstract}

\maketitle
\markboth{Ronan J.~Conlon and Alix Deruelle}{Expanding K\"ahler-Ricci solitons coming out of K\"ahler cones}

\section{Introduction}

\subsection{Overview} A \emph{Ricci soliton} is a triple $(M,\,g,\,X)$, where $M$ is a Riemannian manifold with a complete Riemannian metric $g$
and a complete vector field $X$ satisfying the equation
\begin{equation}\label{soliton1}
\Ric(g)-\frac{1}{2}\mathcal{L}_{X}g+\lambda g=0
\end{equation}
for some $\lambda\in\{-1,\,0,\,1\}$. We call $X$ the \emph{soliton vector field}. A soliton is said to be \emph{steady} if $\lambda=0$, \emph{expanding} if $\lambda=1$, and \emph{shrinking} if $\lambda=-1$. Moreover, if $X=\nabla^{g} f$ for some real-valued smooth function $f$ on $M$, then we say that $(M,\,g,\,X)$ is a \emph{gradient} soliton. In this case, the soliton equation \eqref{soliton1} reduces to $$\Ric(g)-\operatorname{Hess}(f)+\lambda g=0.$$ If $g$ is K\"ahler with K\"ahler form $\omega$, then we say that $(M,\,g,\,X)$ (or $(M,\,\omega,\,X)$) is a \emph{K\"ahler-Ricci soliton} if in addition to $g$ and $X$ satisfying \eqref{soliton1}, the vector field $X$ is real holomorphic. In this case, one can rewrite the soliton equation as
\begin{equation}\label{soliton22}
\rho_{\omega}-\frac{1}{2}\mathcal{L}_{X}\omega+\lambda\omega=0,
\end{equation}
where $\rho_{\omega}$ is the Ricci form of $\omega$. If $g$ is a K\"ahler-Ricci soliton and if $X=\nabla^{g} f$ for some real-valued smooth function $f$ on $M$, then we say that $(M,\,g,\,X)$ is a \emph{gradient K\"ahler-Ricci soliton}.

The study of Ricci solitons and their classification is important in the context of Riemannian geometry. For example, they provide a natural generalisation of Einstein manifolds. Also, to each soliton, one may associate a self-similar solution of the Ricci flow \cite[Lemma 2.4]{Chowchow} which are candidates for singularity models of the flow.

Given now an \emph{expanding} gradient Ricci soliton $(M,\,g,\,X)$ with quadratic Ricci curvature decay and appropriate decay on the derivatives, one may associate a unique tangent cone $(C_{0},\,g_{0})$ with a smooth link \cite{Che-Der, Der-Asy-Com-Egs, Siepmann} which is then an initial condition of the Ricci flow $g(t)$, $t\geq0$, associated to the soliton in the sense that $\lim_{t\to0^{+}}g(t)=g_{0}$ as a Gromov-Hausdorff limit \cite[Remark 1.5]{Che-Der}. The work of this paper is motivated by the converse statement, namely, the following problem.
\smallskip

\noindent {\bf Problem.} \emph{For which metric cones $C_{0}$ is it possible to find an expanding (gradient) Ricci soliton with tangent cone $C_{0}$? Or more generally, given a metric cone $(C_{0},\,g_{0})$, when is it possible to find a Ricci flow $g(t)$, $t\geq0$, such that $\lim_{t\to0^{+}}g(t)=g_{0}$ in the Gromov-Hausdorff sense?}
\smallskip

A metric cone with its radial vector field satisfies \eqref{soliton1} with $\lambda=1$ up to terms of order $O(r^{-2})$ with $r$ the distance from the apex of the cone, hence it defines an ``approximate'' expanding gradient Ricci soliton. So heuristically, the question here is whether or not one can perturb the cone metric to define an actual expanding gradient Ricci soliton.

The second author has shown that one can always solve the above problem when the link of the cone $C_{0}$ is a sphere with positive curvature operator bounded below by $\Id$ \cite{Der-Asy-Com-Egs, Der-Smo-Pos-Cur-Con}, and a recent result due to Lott and Wilson \cite{Lott-Wilson} shows that the above question has a positive answer at the level of formal expansions. When the cone $C_{0}$ is K\"ahler, Siepmann \cite{Siepmann} has shown that the above question always has an affirmative answer when $C_{0}$ is furthermore Ricci-flat and admits an equivariant resolution satisfying certain topological conditions.

In this paper, we also consider this problem in the K\"ahler category. More precisely, we extend the aforementioned result of Siepmann to K\"ahler cones in general (without the need for Ricci-flatness). This leads to new examples of expanding gradient K\"ahler-Ricci solitons generalising those complete expanding gradient K\"ahler-Ricci soliton examples found in \cite{cao2, wang, FIK, futaki-wang, Siepmann}. Moreover, we show uniqueness of expanding gradient K\"ahler-Ricci solitons asymptotic to a given cone with a fixed soliton vector field and
prove that the space of certain solitons of this type on $\mathbb{C}^{n}$ with positive curvature operator on $(1,\,1)$-forms is path-connected.

\subsection{Main results}

Our first main result is the following theorem, which is a generalisation of \cite[Theorem 5.3.1]{Siepmann} where the K\"ahler cone $C_0$ is assumed to be Ricci-flat.
\begin{mtheorem}[Existence and uniqueness]\label{theorem-A}
Let $C_0$ be a K\"ahler cone with complex structure $J_{0}$, K\"ahler cone metric $g_{0}$, Ricci curvature $\operatorname{Ric}(g_{0})$, and radial function $r$. Let $\pi:M\to C_0$ be a K\"ahler resolution of $C_0$ with complex structure $J$ and exceptional set $E$ such that
\begin{enumerate}[label=(\alph*)]
\item  the real torus action on $C_0$ generated by $J_{0}r\partial_{r}$ extends to $M$ so that $X=\pi^{*}(r\partial_{r})$ lifts to $M$;
\item $H^{1}(M)=0$ or $H^{0,\,1}(M)=0$ or $X|_{A}=0$ for $A\subset E$ for which $H_{1}(A)\to H_{1}(E)$ is surjective.
\end{enumerate}
Then for each $c>0$, there exists a unique expanding gradient K\"ahler-Ricci soliton $g_{c}$ on $M$ with soliton vector field $X=\pi^{*}(r\partial_{r})$, the lift of the vector field $r\partial_{r}$ on $C_0$, and with $\mathcal{L}_{JX}g_{c}=0$, such that
\begin{equation}\label{e:main_rate_0}
|(\nabla^{g_{0}})^k(\pi_{*}g_{c}-cg_{0}-\operatorname{Ric}(g_{0}))|_{g_0} \leq C(k)r^{-4-k}\quad\textrm{for all $k\in\mathbb{N}_{0}$}
\end{equation}
if and only if
\begin{equation}\label{condition}
\int_{V}(i\Theta)^{k}\wedge\omega^{\dim_{\mathbb{C}}V-k}>0
\end{equation}
for all positive-dimensional irreducible analytic subvarieties $V\subset E$ and for all $1\leq k\leq \dim_{\C}V$ for some K\"ahler form $\omega$ on $M$ and for some curvature form $\Theta$ of a hermitian metric on $K_{M}$.
\end{mtheorem}
We call a resolution of $C_0$ satisfying condition (a) here an \emph{equivariant resolution}. Such a resolution of a complex cone always exists; see \cite[Proposition 3.9.1]{kollar}. What is not clear a priori is if this resolution satisfies condition \eqref{condition}. From \eqref{soliton22}, one can see that \eqref{condition} is in fact a necessary condition on $M$ in Theorem \ref{theorem-A} to admit an expanding K\"ahler-Ricci soliton. Furthermore, as remarked in \cite{Lebrun-hein}, an asymptotically conical (AC) K\"ahler manifold of complex dimension $n\geq2$ can only have one end, hence in these dimensions having one end is also a necessary condition on $M$ in Theorem A to admit asymptotically conical K\"ahler-Ricci solitons.


Regarding the regularity of the soliton metrics of Theorem A, we note that there exist examples of expanding gradient Ricci solitons of non-K\"ahler type asymptotic to a cone up to only finitely many derivatives \cite{Der-Smo-Pos-Cur-Con}. We further remark that each of the conditions of hypothesis (b), together with the fact that $JX$ is Killing with $J$ the complex structure on $M$, forces our solitons to be gradient; see Lemma \ref{hellothere} and the remarks thereafter, as well as Corollary \ref{kingpin}. On the other hand, these assumptions do allow us to reformulate the problem as a complex Monge-Amp\`ere equation which we can then solve. Our proof of the existence of a solution to this equation follows closely the work of the second author that developed the analysis in the case that the asymptotic cone is not Ricci-flat \cite{Der-Asy-Com-Egs, Der-Smo-Pos-Cur-Con} and the work of Siepmann \cite{Siepmann}.

As an application of Theorem \ref{theorem-A}, we obtain the following generalisation of \cite[Theorem 4.20(ii)]{wang}.
\begin{mcor}[Examples]\label{coro-B}
Let $n\in\mathbb{N}_{0}$ and let $L$ be a negative line bundle over a compact K\"ahler manifold $D$. Moreover, let $\tilde{L}$ denote the line bundle $p^{*}_{1}\mathcal{O}_{\mathbb{P}^{n}}(-1)\otimes p^{*}_{2}L\to\mathbb{P}^{n}\times D$ over $\mathbb{P}^{n}\times D$ with $p_{1}:\mathbb{P}^{n}\times D\to\mathbb{P}^{n}$ and $p_{2}:\mathbb{P}^{n}\times D\to D$ the projections, let $\tilde{L}^{\times}$ denote the blowdown of the zero section of $\tilde{L}$, blowdown $\mathbb{P}^{n}\subset\tilde{L}$ to obtain $L^{\oplus (n+1)}$ and let $\pi:L^{\oplus (n+1)}\to\tilde{L}^{\times}$ denote the further blowdown of the zero section of $L^{\oplus(n+1)}$. Finally, let $g_{0}$ be a K\"ahler cone metric on $\tilde{L}^{\times}$ with Ricci curvature $\operatorname{Ric}(g_{0})$ and with radial function $r$ such that $\frac{1}{a}\cdot r\partial r$ is the Euler vector field\footnote{By the Euler vector field on a vector bundle $E$, we mean the infinitesimal generator of the homotheties of $E$.} on $\tilde{L}\setminus\{0\}$ for some $a>0$.

Then for all $c>0$, there exists a unique expanding gradient K\"ahler-Ricci soliton $g_{c}$ on the total space of $L^{\oplus(n+1)}$ with soliton vector field $X=\pi^{*}(r\partial_{r})$ a scaling of the Euler field on $L^{\oplus(n+1)}$ by $a$, such that
\begin{equation}\label{asymptote}
|(\nabla^{g_0})^k(\pi_{*}g_{c}-cg_{0}-\operatorname{Ric}(g_{0}))|_{g_0 } \leq C(k)r^{-4-k}\quad\textrm{for all $k\in\mathbb{N}_{0}$}
\end{equation}
if and only if $c_{1}(K_{D}\otimes(L^{*})^{\otimes(n+1)})>0$.
\end{mcor}
Here we are able to take any regular K\"ahler cone metric on $\tilde{L}^{\times}$, which amounts to choosing an arbitrary K\"ahler metric on $\mathbb{P}^{n}\times D$, in contrast to \cite[Theorem 4.20(ii)]{wang} where the metric on $\mathbb{P}^{n}\times D$ was required to be the Fubini-Study metric on $\mathbb{P}^{n}$ times a K\"ahler-Einstein metric on $D$.

Notice that when $n=0$, Corollary B asserts that the total space of $L^{\otimes p}$ admits an expanding gradient K\"ahler-Ricci soliton asymptotic to a cone at infinity for any negative line bundle $L$ over a projective manifold $D$ and for any $p$ such that $c_{1}(K_{D}\otimes(L^{*})^{\otimes p})>0$. Corollary B follows from Theorem A after applying the adjunction formula and noting that $\pi:L^{\oplus(n+1)}\to\tilde{L}^{\times}$ is a K\"ahler equivariant resolution of $\tilde{L}^{\times}$ with respect to any positive scaling of the standard $S^{1}$-action on these bundles and that $X$ restricted to the zero set of $L^{\oplus(n+1)}$, that is, the exceptional set of the resolution $\pi$, vanishes, so that the final condition of hypothesis (b) of Theorem A is satisfied with $A=E$.

We next state a result concerning the uniqueness of expanding K\"ahler-Ricci solitons with a fixed holomorphic vector field.
\begin{mtheorem}[General uniqueness]\label{Uniqueness-Theorem}
Let $(M,\,\omega_{i},\,X)_{i\,=\,1,\,2}$ be two complete expanding gradient K\"ahler-Ricci solitons on a non-compact K\"ahler manifold $M$ with the same holomorphic vector field $X$, i.e., $\omega_{1}$ and $\omega_{2}$ are two K\"ahler forms on $M$ that satisfy
$$\rho_{\omega_{i}}-\frac{1}{2}\mathcal{L}_{X}\omega_{i}+\omega_{i}=0\qquad\textrm{for $i=1,\,2,$}$$
where $\rho_{\omega_{i}}$ is the Ricci form of $\omega_{i}$. Suppose in addition that $|\omega_{1}-\omega_{2}|_{g_{1}}(x)=O(d_{g_{1}}(x,\,x_{0})^{\lambda})$ for some $\lambda<0$, where $d_{g_{1}}(\,\cdot\,,\,x_{0})$ denotes the distance to a fixed point $x_{0}\in M$ measured with respect to the K\"ahler metric $g_{1}$ associated to $\omega_{1}$.
\begin{enumerate}[label=\textnormal{(\roman{*})}, ref=(\roman{*})]
\item If $\lambda\in(-2,\,-1)$ and $\operatorname{\rho}_{\omega_{i}}\geq0$ for $i=1,\,2$, then $\omega_{1}=\omega_{2}$.
\item If $\lambda<-2$, $|\pi_{1}(M)|<\infty$, and the difference between the scalar curvatures $|s_{\omega_{1}}-s_{\omega_{2}}|=o(1)$ at infinity, then $\omega_{1}=\omega_{2}$.
\end{enumerate}
\end{mtheorem}
Note that, using elementary Morse theory, one can show that the hypotheses of (i) imply that $M$ is diffeomorphic to $\mathbb{R}^{2n}$ (see Proposition \ref{pot-fct-est}). Bryant \cite{Bry-Kah-Sol} in fact proved more; he showed that expanding gradient K\"ahler-Ricci solitons with non-negative Ricci curvature can exist only on $\mathbb{C}^{n}$.

If two expanding gradient Ricci solitons $g_{1},\,g_{2},$ are asymptotic to the same cone with respect to the same diffeomorphism, then in fact $|g_{1}-g_{2}|_{g_{1}}=O(d_{g_{1}}(x,\,x_{0})^{-k})$ for any $k\geq 0$; see \cite{Der-Uni-Con-Ric-Exp} for details. Thus, this observation, together with Theorem \ref{Uniqueness-Theorem}, imply the uniqueness statement of Theorem A. In general, Theorem \ref{Uniqueness-Theorem} does not require the existence of an asymptotic cone in order to be applied.

The uniqueness issue on compact gradient shrinking K\"ahler-Ricci solitons has already been solved by Tian and Zhu; in \cite{Tian-Zhu-I}, they treat the case where the vector field is fixed and in \cite{Tian-Zhu-II} they remove this constraint and solve the general case. Chodosh and Fong \cite{Cho-Fon} proved uniqueness for expanding gradient K\"ahler-Ricci solitons with positive bisectional curvature asymptotic to the cone $(\mathbb{C}^n,2\operatorname{Re}(\partial\bar{\partial}(|\cdot|^{2a})))$ for some $a\in(0,1)$, where $|\cdot|$ is the Euclidean norm on $\mathbb{C}^{n}$.

For expanding gradient Ricci solitons which are not necessarily K\"ahler and possess no sign assumption on the curvature, but are asymptotic to a cone at infinity, the uniqueness question has been treated by the second author \cite{Der-Uni-Con-Ric-Exp}. The result there can be seen as a ``unique continuation'' statement at infinity; in particular, the methods are very different from those used by Tian and Zhu. In \cite{Der-Uni-Con-Ric-Exp}, the cone at infinity is assumed to be Ricci-flat in contrast to Theorem \ref{Uniqueness-Theorem}. Let us also mention that the uniqueness statement of Siepmann \cite[Theorem 5.4.5]{Siepmann} assumes that the K\"ahler potential, together with its derivatives, decay exponentially to zero at infinity, in which case the uniqueness issue reduces to a uniqueness problem for solutions to a complex Monge-Amp\`ere equation. However, this assumption is too strong in general.

In the proof of Theorem \ref{theorem-A}, we reduce the analysis to a complex Monge-Amp\`ere equation with a transport term given by the action of the lift $X$ of the radial vector field $r\partial_r$ on K\"ahler potentials.

\begin{mtheorem}[Existence, PDE version]\label{t:existence}
Let $\pi:M\to C_0$ be a resolution of a K\"ahler cone $C_0$ of complex dimension $n$ satisfying hypotheses (a) and (b) of Theorem \ref{theorem-A}. Then, with notation as in Theorem A, let $\omega$ be a K\"ahler form on $M$ satisfying $\mathcal{L}_{JX}\omega=0$ with
\begin{equation*}
|(\nabla^{g_0})^k(\pi_{*}\omega-\omega_{0})|_{g_0 } \leq C(k)r^{-2-k}\qquad\textrm{for all $k\in\mathbb{N}_{0}$,}
\end{equation*}
where $\omega_{0}$ denotes the K\"ahler form of the K\"ahler cone metric $g_{0}$ on $C_0$. Then for any $F=\textit{O}(r^{-2})$ together with its derivatives, there exists a unique function $\varphi=\textit{O}(r^{-2})$ together with its derivatives, such that
\begin{equation}
\left\{
\begin{array}{rl}
&\omega_{\varphi}:=\omega+i\partial\bar{\partial}\varphi>0,\\
&\\
&-\varphi+\log\frac{(\omega+i\partial\bar{\partial}\varphi)^{n}}{\omega^{n}}+\frac{1}{2}X\cdot\varphi=F\label{MA-Intro}.
\end{array} \right.
\end{equation}
\end{mtheorem}
We refer the reader to the end of Section 4 for a precise statement of Theorem D.

 Notice that the vector field $X$ in this theorem is unbounded and thus the analysis involved in its proof is not straightforward. Indeed, the linearized operator $\Delta_{\omega_0}+X/2$ is unitarily conjugate to a harmonic oscillator of the asymptotic form $\Delta_{\omega_0}-cr^2$ for some positive constant $c$. On one hand, such a linearized operator is naturally symmetric on a certain $L^2$-space endowed with a measure that is asymptotic to $e^{r^2/4}\omega_0^n$. In particular, it can be shown to have a pure discrete spectrum; c.f. \cite{Der-Sta-Egs}. On the other hand, the fact that the convergence rate to the asymptotic cone is only quadratic, hence polynomial, forces us to consider the same Schauder spaces as in the Riemannian setting \cite{Der-Smo-Pos-Cur-Con}.

On proving a family version of Theorem A, namely Theorem \ref{theo-family-version}, we are able to deduce from the results of Perelman proved in \cite{Tia-Zhu-Con-Kah-Ric} (see also \cite[Chapter 6]{Bou-Eys-Gue}) that the space of certain expanding gradient K\"ahler-Ricci solitons on $\mathbb{C}^{n}$ with positive curvature operator on $(1,\,1)$-forms is path-connected. More precisely, we prove the following.
\begin{mtheorem}[Path-connectedness]\label{coro-F}
 Let $g_{0}$ be a K\"ahler cone metric on $\mathbb{C}^{n}$ with radial function $r$ such that $a\cdot r\partial_{r}$ is the Euler vector field on $\mathbb{C}^{n}$ for some $0<a<1$ and such that the corresponding K\"ahler metric induced on $\mathbb{P}^{n-1}$ by $g_{0}$ has positive curvature operator on $(1,\,1)$-forms. Then for all $c>0$, the unique expanding gradient K\"ahler-Ricci soliton $g_{c}$ satisfying \eqref{asymptote} (with $\pi=\Id$) on $\mathbb{C}^{n}$ with soliton vector field $r\partial_{r}$ has positive curvature operator on $(1,\,1)$-forms and is connected to Cao's expanding gradient K\"ahler-Ricci soliton \cite{cao2} with tangent cone $(\mathbb{C}^{n},\,c\operatorname{Re}(\partial\bar{\partial}|\cdot|^{2a}),\,r\partial_{r})$ by a smooth one-parameter family of expanding gradient K\"ahler-Ricci solitons, each with positive curvature operator on $(1,\,1)$-forms and asymptotic to a K\"ahler cone at infinity.
\end{mtheorem}
\noindent Note that Cao's expanding gradient K\"ahler-Ricci solitons do indeed have positive curvature operator on $(1,\,1)$-forms; see \cite{Che-Zhu-Pos-Cur} for the relevant computations.

Theorem \ref{coro-F} echoes the results of the second author \cite{Der-Smo-Pos-Cur-Con} where the path-connectedness of the space of asymptotically conical expanding gradient Ricci solitons with positive curvature operator is proved. Note that the positivity of the curvature operator on $(1,1)$-forms is not an open condition. Consequently, as in \cite{Der-Smo-Pos-Cur-Con}, one must use certain soliton identities (see Appendix A) to show that its positivity is preserved along the path provided by the family version of Theorem A.

\subsection{Outline of paper}

We begin in Section 2 by recalling the basics of expanding K\"ahler-Ricci solitons and introducing the main analytic tools needed in the setting of asymptotically conical (AC) manifolds. We then construct a background AC K\"ahler metric and set up the complex Monge-Amp\`ere equation in Section \ref{Section-App-Met}. Our background metric will be conical at infinity, hence it will be an approximate expanding K\"ahler-Ricci soliton. We then prove Theorem \ref{t:existence} in Sections \ref{section-fct-spa} -- \ref{section-theorem-existence} which gives us a solution to the complex Monge-Amp\`ere equation, allowing us to perturb our background metric to an expanding gradient K\"ahler-Ricci soliton.

The proof of Theorem \ref{t:existence} follows closely the work of Siepmann \cite{Siepmann}. We implement the continuity method as in the seminal work of Aubin \cite{Aub-Equ-MA} and Yau \cite{Calabiconj} on the existence of K\"ahler-Einstein metrics. The relevant function spaces comprising functions invariant under the extension of the torus action generated by the cone are introduced in Section 4. The openness part of the continuity method is then proved in Section \ref{section-small-def}, followed by a proof of the a priori estimates for the closedness part in Section \ref{section-a-priori-est}. The toric invariance of the function spaces is crucial here in the proof of the a priori $C^0$-estimate on the radial derivative $X\cdot\varphi$; see Proposition \ref{prop-c^0-rad-der}. Section \ref{section-bootstrapping} is then devoted to proving a bootstrapping phenomenon for (\ref{MA-Intro}). As noted previously, the presence of the unbounded vector field $X$ makes the analysis more difficult; for instance, the so-called weighted $C^0$-estimates for the radial derivative $X\cdot \varphi$, where $\varphi$ solves (\ref{MA-Intro}), have to be proved \emph{before} the $C^2$-estimates in order to avoid a circular argument. Section \ref{section-theorem-existence} then completes the proof of Theorem \ref{t:existence}.

In Section \ref{section-proof-uniqueness}, we prove the uniqueness statement, namely Theorem \ref{Uniqueness-Theorem}, before completing the proof of Theorem A in Section \ref{Proof of Theorem A}. We then prove Theorem E in Section \ref{Proof of coroF}, that is, that the space of certain expanding gradient K\"ahler-Ricci solitons on $\mathbb{C}^{n}$ with positive curvature operator on $(1,\,1)$-forms is path-connected. A family version of Theorem A is required for the proof of this. We include it in Section 11 as Theorem \ref{theo-family-version}. Finally, Appendix A gathers together some background technical results.

\subsection{Acknowledgments} The authors wish to thank Hans-Joachim Hein, Song Sun, and Jeff Viaclovsky for many useful discussions and for comments on a preliminary version of this paper. Moreover, the authors wish to thank Hans-Joachim for explaining to them how to remove the assumption \linebreak $H^{1}(D)=0$ from Corollary B in the first version of this paper.

This work was carried out while the authors were supported by the National Science Foundation under Grant No.~DMS-1440140 while in residence at the Mathematical Sciences Research Institute (MSRI) in Berkeley, California, during the Spring 2016 semester. They wish to thank MSRI for their excellent working conditions and hospitality during this time.

The second author wishes to acknowledge funding from the European Research Council (ERC) under the European Union's Seventh Framework Program (FP7/2007-2013)/ERC grant agreement No.~291060 which supported the research that lead to the results contained within this paper.

\newpage

\section{Preliminaries}

\subsection{Riemannian cones} For us, {the definition of a Riemannian cone will take the following form}.

\begin{definition}\label{cone}
Let $(S, g_{S})$
be a compact connected Riemannian manifold. The \emph{Riemannian cone} $C_{0}$ with \emph{link} $S$ is defined to be $\R^+ \times S$ with metric $g_0 = dr^2 \oplus r^2g_{S}$ up to isometry. The radius function $r$ is then characterized intrinsically as the distance from the apex in the metric completion.
\end{definition}

Suppose that we are given a Riemannian cone $(C_0,g_{0})$ as above. Let $(r,x)$ be polar coordinates on $C_{0}$, where $x\in S$, and for $t>0$, define a map
$$\nu_{t}: S\times[1,2] \ni (r,x) \mapsto (tr,x) \in S \times [t,2t].$$ One checks that $\nu_{t}^{*}(g_{0})=t^{2}g_{0}$ and $\nu^{*}_{t}\circ\nabla^{g_0}=\nabla^{g_0}\circ\nu_{t}^{*}$, where $\nabla^{g_0}$ is the  Levi-Civita connection of $g_{0}$.

\begin{lemma}\label{simple321}
Suppose that $\alpha\in\Gamma((TC_0)^{\otimes p}\otimes (T^{*}C_0)^{\otimes q})$ satisfies $\nu_{t}^{*}(\alpha)=t^{k}\alpha$ for every $t>0$ for some $k\in\R$. Then $|(\nabla^{g_0})^{l}\alpha|_{g_{0}}=O(r^{k+p-q-l})$ for all $l\in\N_0$.
\end{lemma}

We shall say that ``$\alpha=O(r^{\lambda})$ with $g_{0}$-derivatives'' whenever $|(\nabla^{g_0})^{k}\alpha|_{g_{0}}=O(r^{\lambda-k})$ for every $k \in \N_0$.
We will then also say that $\alpha$ has ``rate at most $\lambda$'', or sometimes, for simplicity, ``rate $\lambda$'', although it should be understood that (at least when $\alpha$ is purely polynomially behaved and does not contain any $\log$ terms) the rate of $\alpha$ is really the infimum of all $\lambda$ for which this holds.

\subsection{K{\"a}hler cones} Boyer-Galicki \cite{book:Boyer} is a comprehensive reference here.

\begin{definition}A \emph{K{\"a}hler cone} is a Riemannian cone $(C_0,g_0)$ such that $g_0$ is K{\"a}hler, together with a choice of $g_0$-parallel complex structure $J_0$. This will in fact often be unique up to sign. We then have a K{\"a}hler form $\omega_0(X,Y) = g_0(J_0X,Y)$, and $\omega_0 = \frac{i}{2}\p\bar{\p} r^2$ with respect to $J_0$.
\end{definition}

The vector field $r\partial_{r}$ on a K\"ahler cone is real holomorphic and $J_{0}r\partial_r$ is real holomorphic and Killing. This latter vector field is known as the \emph{Reeb field}. The closure of its flow in the isometry group of the link of the cone generates the holomorphic isometric action of a real torus on $C_{0}$ that fixes the apex of the cone. We call a K{\"a}hler cone ``quasiregular'' if this action is an
$S^1$-action (and, in particular, ``regular'' if this $S^1$-action is free), and ``irregular'' if the action generated is that of a real torus of rank $>1$.

Given a K\"ahler cone $(C_{0},\,\omega_{0}=\frac{i}{2}\partial\bar{\partial}r^{2})$ with radius function $r$, it is true that
\begin{equation}\label{conemetric}
\omega_{0}=rdr\wedge\eta +\frac{1}{2}r^{2}d\eta
\end{equation}
where $\eta=i(\bar{\partial}-\partial)\log r$.

Clearly, any K\"ahler cone metric on $L^{\times}$, the contraction of the zero section of a negative line bundle $L$ over a projective manifold, with some positive multiple of the radial vector field equal to the Euler vector field on $L\setminus\{0\}$, is regular. In fact, as the following theorem states, this property characterises all regular K\"ahler cones.
\begin{theorem}[{\cite[Theorem 7.5.1]{book:Boyer}}]\label{regulars}
Let $(C_{0},\,\omega_{0})$ be a regular K\"ahler cone with K\"ahler cone metric $\omega_{0}=\frac{i}{2}\partial\bar{\partial}r^{2}$, radial function $r$, and radial vector field $r\partial_{r}$. Then:
\begin{enumerate}[label=\textnormal{(\roman{*})}, ref=(\roman{*})]
\item $C_{0}$ is biholomorphic to the blowdown $L^\times$ of the zero section of a negative line bundle $L$ over a projective manifold $D$, with $a\cdot r\partial_{r}$ equal to the Euler field on $L\setminus\{0\}$ for some $a>0$.
\item Let $p:L\to D$ denote the projection. Then, writing $\omega_{0}$ as in \eqref{conemetric}, we have that $\frac{1}{2}d\eta=p^{*}\sigma$ for some K\"ahler metric $\sigma$ on $D$ with $[\sigma]=2\pi a\cdot c_{1}(L^{*})$.
\end{enumerate}
\end{theorem}

Conversely, as the next example shows, one can always endow $L^\times$ with the structure of a regular K\"ahler cone metric.
\begin{example}[{\cite[Theorem 7.5.2]{book:Boyer}}]\label{regular-eg}
Let $L$ be a negative line bundle over a projective manifold $D$ and let $p:L\to D$ denote the projection. Then $L$ has a hermitian metric $h$ with $i\partial\bar{\partial}\log h$ a K\"ahler form on $D$. Set $r^{2}=(h|z|^{2})^{a}$ for any $a>0$. Then $\frac{r^{2}}{2}$ defines the K\"ahler potential of a K\"ahler cone metric on $L^{\times}$, the contraction of the zero section of $L$, with K\"ahler form $\omega_{0}=\frac{i}{2}\partial\bar{\partial}r^{2}$ and radial vector field $r\partial_{r}$ a scaling of the Euler vector field on $L\setminus\{0\}$ by $\frac{1}{a}$. Finally, writing $\omega_{0}$ as in \eqref{conemetric}, we have that $\frac{1}{2}d\eta=p^{*}\sigma$, where $\sigma=a\cdot i\partial\bar{\partial}\log h$ is a K\"ahler form on $D$ with $[\sigma]=2\pi a\cdot c_{1}(L^{*})$.
\end{example}
We call the K\"ahler metric $\sigma$ on $D$ from Theorem \ref{regulars} and Example \ref{regular-eg} the \emph{transverse K\"ahler form} of
$\omega_{0}$ on $D$.

\subsection{Type II deformations of K\"ahler cones}
One may deform a K\"ahler cone $(C_{0},\,\omega_{0}=\frac{i}{2}\partial\bar{\partial}r^{2})$ as follows. Let $J_{0}$ denote the complex structure on $C_{0}$. Then take any smooth real-valued function $\varphi$ on $C_{0}$ with \nolinebreak $\mathcal{L}_{r\partial_{r}}\varphi=\mathcal{L}_{J_{0}r\partial_{r}}\varphi=0$ such that $\tilde{\omega}_{0}=\frac{i}{2}\partial\bar{\partial}(r^{2}e^{2\varphi})>0$. The form
$\tilde{\omega}_{0}$ will define a new K\"ahler cone metric on $C_{0}$ with radius function $\tilde{r}:=re^{\varphi}$ and radial vector field $\tilde{r}\partial_{\tilde{r}}=r\partial_{r}$. Let $\tilde{\eta}=i(\bar{\partial}-\partial)\log\tilde{r}$. Then, by \eqref{conemetric}, $\tilde{\omega}_{0}$ may be written as
\begin{equation*}
\begin{split}
\tilde{\omega}_{0}&=\frac{i}{2}\partial\bar{\partial}(r^{2}e^{2\varphi})=\tilde{r}d\tilde{r}\wedge\tilde{\eta} +\frac{1}{2}\tilde{r}^{2}d\tilde{\eta}=\tilde{r}d\tilde{r}\wedge\tilde{\eta} +\frac{1}{2}\tilde{r}^{2}(d\eta+i\partial\bar{\partial}\varphi).
\end{split}
\end{equation*}
A deformation of this type is called a ``deformation of type II''; see for example \cite[Section 7.5.1]{book:Boyer} or \cite[Proposition 4.2]{futaki} for more details.

\begin{example}\label{egtypeii}
Consider a K\"ahler cone $(L^\times,\,\omega_{0}=\frac{i}{2}\partial\bar{\partial}r^{2})$, where $L$ is a negative line bundle over a projective manifold $D$, with radial vector field $r\partial_{r}$ equal to a positive scaling of the Euler vector field on $L\setminus\{0\}$. This K\"ahler cone is regular. Denote by $\sigma$ the transverse K\"ahler form of $\omega_{0}$ on $D$ and by $p:L\to D$ the projection. Then, for any smooth function $\varphi:D\to\mathbb{R}$ such that $\sigma+i\partial\bar{\partial}\varphi>0$, the K\"ahler cone metric $\tilde{\omega}_{0}=\frac{i}{2}\partial\bar{\partial}(r^{2}e^{2p^{*}\varphi})$ defines a deformation of $\omega_{0}$ of type II. In this case, the transverse K\"ahler form $\tilde{\sigma}$ of $\tilde{\omega}_{0}$ on $D$ is precisely $\sigma+i\partial\bar{\partial}\varphi>0$ so that $[\tilde{\sigma}]=[\sigma]$ in $H^{2}(D)$. Moreover, the radial vector field of $\tilde{\omega}_{0}$ is equal to the radial vector field $r\partial_{r}$ of $\omega_{0}$.
\end{example}

\begin{example}\label{example-Kahler-Ricci-Type-II}
Another example of a deformation of type II is to deform a K\"ahler cone metric along the \emph{Sasaki-Ricci flow}. Here we discuss a special case.

Let $D$ be a Fano manifold and consider a K\"ahler cone $(L^\times,\,\omega_{0}=\frac{i}{2}\partial\bar{\partial}r^{2})$, where $L=K^{\frac{1}{q}}_{D}$ for some $q$ that divides the Fano index\footnote{The \emph{Fano index} of a Fano manifold $D$ is the divisibility of $K_{D}$ in $\operatorname{Pic}(D)$.} of $D$ and where $a\cdot r\partial_{r}$ is equal to the Euler vector field on $L\setminus\{0\}$ for some $a>0$. This is also a regular K\"ahler cone and by Theorem \ref{regulars}, the transverse K\"ahler form $\sigma$ of $\omega_{0}$ on $D$ satisfies $[\sigma]\in 2\pi a\cdot c_{1}(-K_{D}^{\frac{1}{q}})=2\pi\frac{a}{q}\cdot c_{1}(D)$. One evolves the rescaled K\"ahler metric $\hat{\sigma}:=\frac{q}{a}\sigma\in 2\pi c_{1}(D)$ on $D$ along the normalised K\"ahler-Ricci flow to obtain a one-parameter family of smooth real-valued functions $\varphi(t)\in C^{\infty}(D)$ satisfying
\begin{equation}\label{krf}
\left\{
\begin{array}{rl}
&\frac{\partial\varphi}{\partial t}=\log\left(\frac{(\hat{\sigma}+i\partial\bar{\partial}\varphi(t))^{n-1}}{\hat{\sigma}^{n-1}}\right)+\varphi(t)-h\\
&\\
&\varphi(0)=0,
\end{array} \right.
\end{equation}
where $h\in C^{\infty}(D)$ is such that $\rho_{\hat{\sigma}}-\hat{\sigma}=i\partial\bar{\partial}h$, here $\rho_{\hat{\sigma}}$ denoting the Ricci form of $\hat{\sigma}$. The induced evolution $\sigma(t)$ of $\sigma$ is then via $\sigma(t)=\sigma+i\partial\bar{\partial}(\frac{a}{q}\cdot\varphi(t))$
with the corresponding evolution $\omega_{0}(t)$ of the cone metric $\omega_{0}$ (see Example \ref{egtypeii}) given by $$\omega_{0}(t)=\frac{i}{2}\partial\bar{\partial}(r^{2}e^{2\frac{a}{q}p^{*}\varphi(t)}),$$
where $p:K^{l}_{D}\to D$ denotes the projection. Note that, by construction, $\sigma(t)$ is the transverse K\"ahler form of $\omega_{0}(t)$ on $D$, and by \cite{Cao}, the family $\omega_{0}(t)$ exists for all $t\geq0$. We also point out that the radial vector field of $\omega_{0}(t)$ remains fixed equal to $r\partial_{r}$ for all $t$.
\end{example}

\subsection{Asymptotically conical Riemannian manifolds}

\begin{definition}\label{d:AC}
Let $(M,g)$ be a complete Riemannian manifold and let $(C_0,g_0)$ be a Riemannian cone. We call $M$ \emph{asymptotically conical} (AC) with tangent cone $C_{0}$ if there exists a diffeomorphism $\Phi: C_0\setminus K \to M \setminus K'$ with $K,K'$ compact, such that $\Phi^*g - g_0 = O(r^{-\epsilon})$ with $g_0$-derivatives for some $\epsilon > 0$. A \emph{radius function} is a smooth function $\rho: M \to [1,\infty)$ with $\Phi^*\rho = r$ away from $K'$.
\end{definition}

\subsection{Asymptotically conical K\"ahler manifolds}

\begin{definition}\label{d:ACK}
Let $(M,g)$ be a complete K\"ahler manifold with complex structure $J$ and let $(C_0,g_0)$ be a K\"ahler cone with a choice of $g_0$-parallel complex structure $J_0$. We call $M$ \emph{asymptotically conical} (AC) \emph{K\"ahler} with tangent cone $C_{0}$ if there exists a diffeomorphism $\Phi: C_0\setminus K \to M \setminus K'$ with $K,K'$ compact, such that $\Phi^*g - g_0 = O(r^{-\epsilon})$ with $g_0$-derivatives
and $\Phi^*J - J_0 = O(r^{-\epsilon})$ with $g_0$-derivatives for some $\epsilon > 0$. In particular, $(M,\,g)$ is AC with tangent cone $C_{0}$.
\end{definition}

We implicitly only allow for one end in Definitions \ref{d:AC} and \ref{d:ACK}. This is simply to fix ideas. Furthermore, for our applications, the map $\Phi$ in Definition \ref{d:ACK} will always be a biholomorphism.

\subsection{K\"ahler-Ricci solitons}

The metrics we are interested in are the following.
\begin{definition}
A \emph{K\"ahler-Ricci soliton} is a triple $(M,\,g,\,X)$, where $M$ is a K\"ahler manifold, $X$ is a holomorphic vector field on $M$, and $g$ is a complete K\"ahler metric on $M$ whose K\"ahler form $\omega$ satisfies
\begin{equation}\label{soliton}
\rho_{\omega}-\frac{1}{2}\mathcal{L}_{X}\omega+\lambda\omega=0
\end{equation}
for some $\lambda\in\{-1,\,0,\,1\}$, here $\rho_{\omega}$ denoting the Ricci form of $\omega$. A K\"ahler-Ricci soliton is said to be \emph{steady} if $\lambda=0$, \emph{expanding} if $\lambda=1$, and \emph{shrinking} if $\lambda=-1$. We call $X$ the \emph{soliton vector field}. If, in addition, $X=\nabla^{g}f$ for some real-valued smooth function $f$ on $M$, then we say that $(M,\,g,\,X)$ is a \emph{gradient K\"ahler-Ricci soliton}. In this case, we call $f$ the \emph{potential function} of the soliton.
\end{definition}
Here we are concerned with expanding K\"ahler-Ricci solitons. This includes the class of K\"ahler-Einstein metrics with negative scalar curvature (with $X=0$) and Ricci-flat K\"ahler cone metrics (with $X=r\frac{\partial}{\partial r}$). Moreover, a K\"ahler cone metric (with $X=r\frac{\partial}{\partial r}$) satisfies \eqref{soliton} with $\lambda=1$ up to terms of order $O(r^{-2})$.

We next note some important properties of K\"ahler-Ricci solitons.
\begin{lemma}\label{hellothere}
Let $(M,\,g,\,X)$ be a K\"ahler-Ricci soliton with complex structure $J$. If $H^{1}(M)=0$ or $H^{0,\,1}(M)=0$ and $JX$ is Killing for $g$, then $(M,\,g,\,X)$ is gradient. Conversely, if $(M,\,g,\,X)$ is gradient, then $JX$ is Killing for $g$.
\end{lemma}
This lemma follows from Corollary \ref{howdy}.

By construction, the vector field $JX$ is Killing for the solitons of Theorem \ref{theorem-A}. Thus, when $M$ in Theorem A satisfies either one of the vanishing conditions $H^{1}(M)=0$ or $H^{0,\,1}(M)=0$, we see that the resulting solitons there are gradient. Since $M$ is homotopy equivalent to $E$ in Theorem A, by Corollary A.9, the same conclusion also holds true if $M$ in Theorem A satisfies the third condition of hypothesis (b) of that theorem. As for the solitons $(M,\,g,\,X)$ of \cite{Siepmann}, the triviality of the canonical bundle of the cone model implies that $H^{0,\,1}(M)=0$ so that his solitons are also gradient. In general, the vanishing $H^{0,\,1}(M)=0$ holds on a resolution $M$ of a complex cone whose apex is a rational singularity; see Proposition \ref{vanishinging}.

We also have a necessary condition for the existence of an expanding K\"ahler-Ricci soliton.
\begin{lemma}\label{necessaryy}
Let $(M,\,g,\,X)$ be an expanding K\"ahler-Ricci soliton. Then $K_{M}|_{V}$ is ample for any compact smooth irreducible subvariety $V$ of $M$.
\end{lemma}
\noindent One can see this directly from the defining equation of an expanding K\"ahler-Ricci soliton. In particular, if $\Gamma$ is a non-cyclic finite subgroup of $U(2)$ containing no complex reflections, then, since the minimal resolution of $\mathbb{C}^{2}/\Gamma$ always contains a $(-2)$-curve \cite[Theorem 4.1]{Jeff}, no resolution of such a singularity can admit expanding K\"ahler-Ricci solitons. If $\Gamma$ is a cyclic subgroup of $U(2)$, then, as shown by Siepmann \cite[Theorem 5.6.3]{Siepmann}, the minimal resolution $M$ of $\mathbb{C}^{2}/\Gamma$ admits expanding K\"ahler-Ricci solitons if and only if $K_{M}|_{C}$ is ample for every curve $C$ in the exceptional set of the resolution. There is a numerical criterion on $\Gamma$ to determine when this is the case.

\subsection{Function spaces on AC manifolds}\label{s:linear_analysis}
We require a definition of weighted H\"older spaces.

\begin{definition}
Let $(M, g)$ be AC with tangent cone $(C_0,g_{0})$, and let $\rho$ be a radius function.

(i) For $\beta\in\R$ and $k$ a non-negative integer, define $C_{\beta}^{k}(M)$ to be the space of continuous functions $u$ on $M$ with $k$ continuous derivatives such that $$\norm{u}_{C^{k}_{\beta}} :=\sum_{j=0}^{k}\sup_{M}|\rho^{j-\beta}(\nabla^{g})^ju| < \infty.$$
Define $C_{\beta}^{\infty}(M)$ to be the intersection of the $C_{\beta}^{k}(M)$ over all $k\in \N_0$.

(ii) Let $\delta(g)$ be the convexity radius of $g$, and write $d(x,y)$ for the distance between two points $x$ and $y$ in $M$. For $T$ a tensor field on $M$ and $\alpha,\gamma\in\R$, define
$$[T]_{C^{0,\alpha}_\gamma} :=\sup
_{\substack{x\,\neq\,y\,\in\,M \\
d(x,y)\,<\,\delta(g)}}\left[\min(\rho(x),\rho(y))^{-\gamma}\frac{|T(x)-T(y)|}{d(x,y)^{\alpha}} \right] ,$$
where $|T(x)-T(y)|$ is defined via parallel transport along the minimal geodesic from $x$ to $y$.

(iii) For $\beta\in\R$, $k$ a non-negative integer, and $\alpha\in(0,1)$, define the weighted H\"older space $C_{\beta}^{k,\alpha}(M)$ to be the set of $u\in C_{\beta}^{k}(M)$ for which the norm $$\norm{u}_{C_{\beta}^{k,\alpha}}:=\norm{u}_{C^{k}_{\beta}} +[(\nabla^{g})^ku]_{C^{0,\alpha}_{\beta-k-\alpha}} < \infty.$$
\end{definition}

Whether one decides to measure the asymptotics of a function $u\in C_{\beta}^{k}(M)$ in terms of the metric $g$ or $g_{0}$ actually makes no difference.

\section{Constructing a background metric and the equation set-up}\label{Section-App-Met}\label{subsection-existence-sol}

\subsection{Construction of an approximate soliton}
In this section we consider a K\"ahler cone $C_{0}$ of complex dimension $n$ with complex structure $J_{0}$ and radius function $r$ and an equivariant resolution $\pi:M\to C_0$ of $C_{0}$ with exceptional set $E$ so that the torus action induced by the flow of the vector field $J_{0}r\partial_{r}$ on $C_{0}$ extends to $M$. We denote by $J$ the complex structure on $M$ and we write $X$ for the lift of the vector field $r\partial_{r}$ on $C_{0}$ to $M$. We claim the following.
\begin{prop}\label{background-metric}
Suppose that
\begin{equation}\label{hypoth}
\int_{V}(i\Theta_{h})^{k}\wedge\omega^{\dim_{\mathbb{C}}V-k}>0
\end{equation}
for all positive-dimensional irreducible analytic subvarieties $V\subset E$ and for all $1\leq k\leq \dim_{\C}V$ for some K\"ahler form $\omega$ on $M$ and for some hermitian metric $h$ on $K_{M}$ with curvature form $\Theta_{h}$. Denote by $\widetilde{\Theta_{h}}$ the average of $\Theta_{h}$ over the torus action on $M$ induced by the flow of the vector field $J_{0}r\partial_{r}$ on $C_{0}$. Then for each $c>0$, there exists a real-valued smooth function $u_{c}\in C^{\infty}(M)$ with $\mathcal{L}_{JX}u_{c}=0$ such that $\omega_{c}:=i\widetilde{\Theta_{h}}+i\partial\bar{\partial}u_{c}$ is a K\"ahler form satisfying
$$\omega_{c}=\pi^{*}(c\omega_{0}-\rho_{\omega_{0}})$$
outside a compact subset of $M$. Here, $\omega_{0}$ denotes the K\"ahler form of the K\"ahler cone metric on $C_{0}$ and $\rho_{\omega_{0}}$ denotes the corresponding Ricci form. In particular, $\mathcal{L}_{JX}\omega_{c}=0$.
\end{prop}

\begin{proof}
In what follows, we identify $M\setminus E$ and $C_{0}$ via $\pi$.

Since $\rho_{\omega_{0}}$ is the curvature form of the hermitian metric on $-K_{C_{0}}$ induced by $\omega_{0}$, there exists a smooth function $u$ on $M\setminus E$ such that $$-i\Theta_{h}=\rho_{\omega_{0}}+i\partial\bar{\partial}u.$$
Moreover, since \eqref{hypoth} holds by assumption, \cite[Theorem 1.1]{paunthm} implies that there exists $\epsilon>0$ such that $i\Theta_{h}+i\partial\bar{\partial}\varphi>0$ on $E\cup\{r<4\epsilon\}$ for some smooth real-valued function $\varphi$ on this set.

We proceed as in the proof of \cite[Lemma 2.15]{Conlon}. Let $\alpha>0$ and let $\psi_{\alpha}:\R^+\to\R^+$ be smooth with $\psi_{\alpha}',\psi_{\alpha}''\geq 0$ and
$$\psi_{\alpha}(t) = \begin{cases}
\left(\frac{\epsilon}{3}\right)^{2\alpha} & \textrm{if}\;\,t<\left(\frac{\epsilon}{2}\right)^{2\alpha},\\
t & \textrm{if}\;\,t>\epsilon^{2\alpha}.
\end{cases}
$$
Then $\Psi_{\alpha}:=\psi_{\alpha}\circ r^{2\alpha}: M \to \R^+$ satisfies
$$
i\partial\bar{\partial}\Psi_{\alpha} =
\begin{cases}
0 &\textrm{on}\;E \cup \{0 < r <\frac{\epsilon}{2}\},\\
\psi_{\alpha}'' i\p r^{2\alpha}\wedge\bar{\p}r^{2\alpha}+\psi_{\alpha}'i\p\bar{\p}r^{2\alpha} &\textrm{on}\;\{r > \frac{\epsilon}{4}\}.
\end{cases}
$$
Clearly $i\partial\bar{\partial}\Psi_{\alpha}\geq 0$ everywhere on $M$ and $i\partial\bar{\partial}\Psi_{\alpha}=i\partial\bar{\partial}r^{2\alpha}>0$ on $\{r>\epsilon\}$. {Also}, fix a cutoff function $\zeta:\mathbb{R}^{+}\to\R^{+}$ with
$$
\zeta(t) = \begin{cases}
1 & \textrm{if}\;\,t<2,\\
0 & \textrm{if}\;\,t>3,
\end{cases}
$$
and for $R>4\epsilon$, define $\zeta_{R}:M\to\mathbb{R}$ by $\zeta_{R}:=\zeta\circ(r/R)$. Given $c > 0$, we construct
$$\hat{\omega}_{c}:=i\Theta_{h}+i\partial\bar{\partial}(\zeta_{\epsilon}\varphi)+i\p\bar{\p}((1-\zeta_{\epsilon})u) + Ci\p\bar{\p}(\zeta_{R}\Psi_{\alpha})+ci\p\bar{\p}\Psi_{1}$$
with $C$ and $R$ to be determined and with $\alpha\in(0,\,1)$ fixed. {Note that} $$\textrm{$\hat{\omega}_{c} = i\Theta_{h}+i\p\bar{\p}\varphi+Ci\p\bar{\p}(\zeta_{R}\Psi_{\alpha})+ci\p\bar{\p}\Psi_{1}\geq i\Theta_{h}+i\p\bar{\p}\varphi>0$ on $E\cup\{0< r < 2\epsilon\}$}$$ because $\Psi_{\alpha}$ and $\Psi_{1}$ are plurisubharmonic; $\hat{\omega}_{c} =-\rho_{\omega_{0}} + ci\p\bar{\p}r^{2} > 0$ on $\{2R < r \}$, after increasing $R$ if necessary, because $|\rho_{\omega_{0}}|_{i\partial\bar{\partial}r^{2}} = O(r^{-2})$; $\hat{\omega}_{c}>0$ on $\{2\epsilon\leq r\leq 3\epsilon\}$ by compactness if $C$ is made large enough; $\hat{\omega}_{c}=-\rho_{\omega_{0}}+Ci\partial\bar{\partial}r^{2\alpha}+ci\partial\bar{\partial}r^{2}>0$ on $\{3\epsilon<r<R\}$ after further increasing $C$ independently of $R$, which one can do since $|\rho_{\omega_{0}}|_{i\partial\bar{\partial}r^{2\alpha}}=O(r^{-2\alpha})$; and finally, $\hat{\omega}_{c} > 0$ on $\{2R \leq r \leq 3R\}$ after further increasing $R$ if necessary, since $\Psi_{\alpha}$ is of lower order compared to $\Psi_{1}$. In conclusion, $\hat{\omega}_{c}$ is a genuine K\"ahler form {on $M$} for suitable choices of $C$ and $R$ with $\hat{\omega}_{c} = c\omega_{0}-\rho_{\omega_{0}}$ on $\{r>2R\}$.

We next average $\hat{\omega}_{c}$ over the action of the torus $T^{k}$ on $M$ induced by the flow of the vector field $J_{0}r\partial_{r}$ on $C_{0}$ by setting
$$\omega_{c}:=\frac{1}{|T^{k}|}\int_{T^{k}}\psi_{g}^{*}\hat{\omega}_c\,d\mu(g)=i\widetilde{\Theta_{h}}+i\partial\bar{\partial}u_{c},$$
where $\psi_{g}:M\to M$ is the automorphism of $M$ induced by $g\in T^{k}$ and where $u_{c}$ is defined implicitly. Since there is a path in $T^{k}$ connecting $g$ to the identity, we have that $\psi_{g}^{*}[\hat{\omega}_{c}]=[\psi_{g}^{*}\hat{\omega}_{c}]=[\hat{\omega}_{c}]$, from which it follows that $[\omega_{c}]=[\hat{\omega}_{c}]$. Moreover, it is clear that $\mathcal{L}_{JX}u_{c}=0$ and $\mathcal{L}_{JX}\omega_{c}=0$. Finally, since $T^{k}$ acts by holomorphic isometries on the slices of the cone $C_{0}$, we have that $\psi_{g}^{*}\rho_{\omega_{0}}=\rho_{\omega_{0}}$ and $\psi_{g}^{*}\omega_{0}=\omega_{0}$ for every $g\in T^{k}$. Hence $\omega_{c}=c\omega_{0}-\rho_{\omega_{0}}$ on $\{r>2R\}$ also.
\end{proof}

\newpage
\subsection{Set-up of the complex Monge-Amp\`ere equation}

We next set up the complex Monge-Amp\`ere equation that we will solve in order to construct our expanding gradient K\"ahler-Ricci solitons. In what follows, we drop the subscript $c$ from the metric $\omega_{c}$ and the function $u_{c}$ of Proposition \ref{background-metric} for clarity.
\begin{prop}\label{equationsetup}
Let $\omega$ denote the K\"ahler form of Proposition \ref{background-metric} and suppose that the resolution $\pi:M\to C_{0}$ satisfies hypothesis (b) of Theorem A, i.e., $H^{1}(M)=0$ or $H^{0,\,1}(M)=0$ or \linebreak $X|_{A}=0$ for $A\subset E$ for which $H_{1}(A)\to H_{1}(E)$ is surjective. Let $\varphi\in C_{-\epsilon}^{\infty}(M)$ for some $\epsilon>0$ with \linebreak $\omega_{\varphi}:=\omega+i\partial\bar{\partial}\varphi>0$. Then
\begin{equation}\label{sexysoliton}
\rho_{\omega_{\varphi}}+\omega_{\varphi}-\frac{1}{2}\mathcal{L}_{X}\omega_{\varphi}=0
\end{equation}
if and only if
\begin{equation}\label{e:soliton}
-\varphi+\log\frac{(\omega+i\partial\bar{\partial}\varphi)^{n}}{\omega^{n}}+\frac{1}{2}X\cdot\varphi=F,
\end{equation}
where $F\in C^{\infty}_{-2}(M)$ satisfies
\begin{equation*}
\left\{
\begin{array}{rl}
&\rho_{\omega}+\omega-\frac{1}{2}\mathcal{L}_{X}\omega=i\partial\bar{\partial}F\\
&\\
&\mathcal{L}_{JX}F=0
\end{array} \right.
\end{equation*}
and has leading order term $\frac{s_{\omega_{0}}}{2}+O(r^{-4})$. Here, $\rho_{\omega_{\varphi}}$,\,$\rho_{\omega_{0}}$, denote the Ricci forms of $\omega_{\varphi}$ and the K\"ahler cone metric $\omega_{0}$ respectively, and $s_{\omega_{0}}$ denotes the scalar curvature of $\omega_{0}$.
\end{prop}

\begin{proof}
If $\varphi$ satisfies \eqref{e:soliton}, then by taking $i\partial\bar{\partial}$ of this equation, we see that
$\omega_{\varphi}$ satisfies \eqref{sexysoliton}.

Conversely, suppose that $\omega_{\varphi}$ satisfies \eqref{sexysoliton}. Then
\begin{equation}\label{computation}
\begin{split}
0&=\rho_{\omega_{\varphi}}+\omega_{\varphi}-\frac{1}{2}\mathcal{L}_{X}
\omega_{\varphi}\\
&=\rho_{\omega_{\varphi}}-\rho_{\omega}+\rho_{\omega}+\omega_{\varphi}-\frac{1}{2}\mathcal{L}_{X}\omega_{\varphi}\\
&=-i\partial\bar{\p}\log\frac{(\omega+i\partial\bar{\partial}\varphi)^{n}}{\omega^{n}}+
\rho_{\omega}+\omega_{\varphi}-\frac{1}{2}\mathcal{L}_{X}\omega_{\varphi}\\
&=-i\partial\bar{\p}\log\frac{(\omega+i\partial\bar{\partial}\varphi)^{n}}{\omega^{n}}+
i\partial\bar{\partial}\varphi-\frac{1}{2}i\partial\bar{\partial}\left(X\cdot\varphi\right)
+(\rho_{\omega}+\omega-\frac{1}{2}\mathcal{L}_{X}\omega),\\
\end{split}
\end{equation}
so that
\begin{equation*}
\begin{split}
i\partial\bar{\p}\left(-\varphi+\log\frac{(\omega+i\partial\bar{\partial}\varphi)^{n}}{\omega^{n}}+\frac{1}{2}X\cdot\varphi\right)
&=\rho_{\omega}+\omega-\frac{1}{2}\mathcal{L}_{X}\omega.\\
\end{split}
\end{equation*}
Now, since $\mathcal{L}_{JX}\omega=0$, $JX$ is Killing, and so by Lemma \ref{lemmma}, the $g$-dual one-form $\eta_{X}:=g(X,\,\cdot)$ of $X$ is closed,
where $g$ is the K\"ahler metric associated to $\omega$. In the case that $H^{1}(M)=0$ or $H^{0,\,1}(M)=0$, one can then find a smooth real-valued function $\theta_{X}$ such that $\eta_{X}=d\theta_{X}$, so that $X=\nabla^{g}\theta_{X}$. In the case that $H_{1}(A)\to H_{1}(E)$ is surjective, the existence of such a function $\theta_{X}$ follows from Corollary \ref{kingpin} since $E$ is homotopy equivalent to $M$. It then follows that in all cases, $\omega\lrcorner X=d\theta_{X}\circ J$, so that we can write $$\mathcal{L}_{X}\omega=d(\omega\lrcorner X)=i\partial\bar{\partial}\theta_{X}.$$ Since $\mathcal{L}_{JX}\mathcal{L}_{X}\omega=\omega([JX,\,X])=0$, by averaging over the action of the torus on $M$ induced by that on $C_{0}$, we may assume that $\mathcal{L}_{JX}\theta_{X}=0.$

Moreover, we have that $$\rho_{\omega}+i\Theta_{h}=i\partial\bar{\partial}v$$ for some $v\in C^{\infty}(M)$, where $\Theta_{h}$ is as in Proposition \ref{background-metric}. Averaging this equation over the action of the torus on $M$ induced by that on $C_{0}$, we obtain
$$\rho_{\omega}+i\widetilde{\Theta_{h}}=i\partial\bar{\partial}\tilde{v}$$
for some $\tilde{v}\in C^{\infty}(M)$ satisfying $\mathcal{L}_{JX}\tilde{v}=0$, where again our notation follows Proposition \ref{background-metric}. Here we have used the fact that $JX$ is holomorphic and Killing so that $\mathcal{L}_{JX}\rho_{\omega}=0$.

Now let $u\in C^{\infty}(M)$ be as in Proposition \ref{background-metric}. Then we can write:
\begin{equation}\label{sexier}
\begin{split}
\rho_{\omega}+\omega-\frac{1}{2}\mathcal{L}_{X}\omega&=\rho_{\omega}+i\widetilde{\Theta_{h}}
+i\partial\bar{\partial}u-\frac{1}{2}\mathcal{L}_{X}\omega\\
&=i\partial\bar{\partial}\tilde{v}+i\partial\bar{\partial}u-i\partial\bar{\partial}\theta_{X}\\
&=i\partial\bar{\partial}F
\end{split}
\end{equation}
for $F:=\tilde{v}+u-\theta_{X}\in C^{\infty}(M)$. In particular, notice that $\mathcal{L}_{JX}F=0$.

Next observe that at infinity we have
\begin{equation}\label{sexiest}
\begin{split}
\rho_{\omega}+\omega-\frac{1}{2}\mathcal{L}_{X}\omega&=\rho_{\omega}+(\omega_{0}-\rho_{\omega_{0}})-
\frac{1}{2}\left(\mathcal{L}_{r\p_{r}}\omega_{0}-\frac{1}{2}\underbrace{\mathcal{L}_{r\p_{r}}\rho_{\omega_{0}}}_{=\,0}\right)\\
&=-i\partial\bar{\partial}\log\frac{(\omega_{0}-\rho_{\omega_{0}})^{n}}{\omega^{n}_{0}}
+\underbrace{\omega_{0}-\frac{1}{2}\mathcal{L}_{r\p_{r}}\omega_{0}}_{=\,0}\\
&=-i\partial\bar{\partial}\log\frac{(\omega_{0}-\rho_{\omega_{0}})^{n}}{\omega^{n}_{0}}=i\partial\bar{\partial}G
\end{split}
\end{equation}
for
\begin{equation*}
\begin{split}
G=G(\omega_{0})&:=-\log\frac{(\omega_{0}-\rho_{\omega_{0}})^{n}}{\omega^{n}_{0}}=-\log\Bigg(1-\underbrace{n\frac{\omega_{0}^{n-1}\wedge \rho_{\omega_{0}}}{\omega^{n}_{0}}}_{=\,O(r^{-2})}+O(r^{-4})\Bigg)\\
&=-\left(-n\frac{\omega_{0}^{n-1}\wedge \rho_{\omega_{0}}}{\omega^{n}_{0}}+O(r^{-4})\right)=\frac{s_{\omega_{0}}}{2}+O(r^{-4})\in C^{\infty}_{-2}(M).
\end{split}
\end{equation*}
Notice that $\mathcal{L}_{JX}G=0$. On subtracting \eqref{sexier} from \eqref{sexiest}, we see that at infinity
\begin{equation}\label{beyonce}
i\partial\bar{\partial}(F-G)=0.
\end{equation}
Since $\mathcal{L}_{JX}(F-G)=0$, it then follows that $\frac{X}{2}\cdot(F-G)$ is holomorphic. But since $\frac{X}{2}\cdot(F-G)$ is real-valued, at infinity $\frac{X}{2}\cdot (F-G)=c_{0}$ for some constant $c_{0}$. Hence, $$F-G=c_{0}\log r+c_{1}(x),$$ where $c_{1}(x)$ is a function that depends on the link $(S,\,g_{S})$ of the cone. But by \eqref{beyonce}, we also have that $\Delta_{g_{0}}(F-G)=0$, which implies that
$$\frac{(2n - 2)}{r^{2}}c_{0}+\frac{1}{r^{2}}\Delta_{g_{S}}c_{1}(x)=0.$$ Integrating this equation over the link of the cone shows that $c_{0}=0$ so that $F-G=C$ at infinity for some constant $C$. Therefore, by subtracting a constant from $F$ in \eqref{sexier} if necessary, we arrive at
\begin{equation}\label{gorilla}
i\partial\bar{\p}\left(-\varphi+\log\frac{(\omega+i\partial\bar{\partial}\varphi)^{n}}{\omega^{n}}+\frac{1}{2}X\cdot\varphi\right)=i\partial\bar{\partial}F,\\
\end{equation}
where $F\in C^{\infty}_{-2}(M)$ is equal to $\frac{s_{\omega_{0}}}{2}+O(r^{-4})$ at infinity.

We now contract \eqref{gorilla} with $\omega$ to find that
\begin{equation*}
\Delta_{\omega}\left(-\varphi+\log\frac{(\omega+i\partial\bar{\partial}\varphi)^{n}}{\omega^{n}}
+\frac{1}{2}X\cdot\varphi-F\right)=0.
\end{equation*}
Applying the maximum principle then yields \eqref{e:soliton}.
\end{proof}

\section{Main setting and function spaces}\label{section-fct-spa}

Let $(M,\,g)$ be a complete Riemannian manifold. Motivated by the work of Siepmann \cite{Siepmann}, we define the following weighted H\"older spaces for $(M,\,g)$ that have already been introduced in \cite{Der-Smo-Pos-Cur-Con} and differ slightly from those function spaces introduced in Section \ref{s:linear_analysis}.

Let $E$ be a tensor bundle over $M$, i.e., $E=(\otimes^rT^*M)\otimes(\otimes^sTM)$ or a subbundle thereof, such as the exterior bundles $\Lambda^rT^*M$ or the symmetric bundles $S^rT^*M$. (We will mainly be concerned with the cases where $E$ is the trivial line bundle over $M$, the bundle of symmetric $2$-tensors $S^2T^*M$, or the bundle of $(1,\,1)$-forms $\Lambda^{(1,1)}T^*M$ when $M$ is a complex manifold.) Let $\alpha\in(0,1)$ and let $k$ be a non-negative integer. We omit the reference to $\alpha$ when $\alpha=0$; the same convention will apply when we deal with functions.
We begin with some definitions.
\begin{itemize}
\item Let $(M,\,g,\,J)$ be a K\"ahler manifold with associated K\"ahler form $\omega$ and let $X$ be a smooth real vector field on $M$. The \textit{weighted Laplacian} (\emph{with respect to $X$}) is defined as
$$\Delta_{\omega,X}T:=\Delta_{\omega} T+\nabla^g_{X}T,$$
where $T$ is a tensor on $M$ and $\nabla^g$ is the complex linear extension of the Levi-Civita connection of $g$. Also, here $\Delta_{\omega}$ denotes the Laplacian associated to $\nabla^g$. In normal coordinates this may be written as  $$\Delta_{\omega}:=\frac{1}{2}\left(\nabla^g_{i}\nabla^g_{\bar{\imath}}+\nabla^g_{\bar{\imath}}\nabla^g_i\right).$$
Recall that the Laplacian acting on functions takes the form
$$\Delta_{\omega}f=g^{i\bar{\jmath}}\partial_i\partial_{\bar{\jmath}}f=\tr_{\omega}\left(\frac{i}{2}\partial\bar{\partial}f\right)$$
for $f\in C^{\infty}(M)$ a smooth real-valued function on $M$. Here, the trace operator $\tr_{\omega}$ on $(1,\,1)$-forms is defined by
$$\tr_{\omega}(\alpha):=\frac{n\omega^{n-1}\wedge\alpha}{\omega^n}=g^{i\bar{\jmath}}\alpha_{i\bar{\jmath}},$$
where $\alpha=\frac{i}{2}\alpha_{j\bar{k}}dz^j\wedge dz^{\bar{k}}$ is a $(1,\,1)$-form on $M$.

In order to keep the notation as light as possible, we will omit the reference to the background K\"ahler metric $g$ or the associated K\"ahler form $\omega$ when there is no possibility for confusion.\\
\item Let $(M,\,g,\,J)$ be a K\"ahler manifold and let $w:M\rightarrow\mathbb{R}$ be a smooth real-valued function on $M$. The \textit{weighted Laplacian} (\emph{with respect to $w$}) is defined as
$$\Delta_{\omega,w}T:=\Delta_{\omega,\nabla w} T,$$
where $T$ is a tensor on $M$.\\

\item We define
\begin{eqnarray*}
&&|\nabla\varphi|^2:=g^{i\bar{\jmath}}\partial_i\varphi\partial_{\bar{\jmath}}\varphi\quad\textrm{and}\quad |\alpha|^2:=g^{i\bar{l}}g^{k\bar{\jmath}}\alpha_{i\bar{\jmath}}\alpha_{k\bar{l}},
\end{eqnarray*}
where $\varphi$ is a smooth real-valued function and $\alpha=\frac{i}{2}\alpha_{j\bar{k}}dz^j\wedge dz^{\bar{k}}$ is a $(1,\,1)$-form on $M$. Notice that $|\nabla\varphi|^2$ is half of the usual Riemannian norm of $\nabla\varphi$ with respect to $g$. One can also define in a similar fashion the norm of a tensor of any type on $M$.\\

\item Let $h\in C^{k,\alpha}_{loc}(M,E)$ and define
$$\left[\nabla^kh\right]_{\alpha}:=\sup_{x\in M}\sup_{y\in B(x,\delta)\setminus\{x\}}\frac{\arrowvert\nabla^kh(x)-P_{x,y}^*\nabla^kh(y)\arrowvert}{d(x,y)^{\alpha}},$$ where $\delta$ is a fixed positive constant depending on the injectivity radius of $(M,g)$ and $P_{x,y}$ denotes the parallel transport along the unique minimizing geodesic from $x$ to $y$. Then we set $$C^{k,\alpha}(M,E):=\{h\in C^{k,\alpha}_{loc}(M,E):\|h\|_{C^{k,\alpha}(M,E)}<+\infty\},$$ where
$$\|h\|_{C^{k,\alpha}(M,E)}:=\sum_{i=0}^k\sup_M\arrowvert\nabla^ih\arrowvert+\left[\nabla^kh\right]_{\alpha}.$$

\item We define $$C_{con}^{k,\alpha}(M,E):=\{h\in C^{k,\alpha}_{loc}(M,E):\|h\|_{C_{con}^{k,\alpha}(M,E)}<+\infty\},$$ where
$$\|h\|_{C_{con}^{k,\alpha}(M,E)}:=\sum_{i=0}^k\|(r_p^2+1)^{i/2}\nabla^ih\|_{C^{0,\alpha}(M,E)}.$$
Here, $r_p$ denotes the distance function to a fixed point $p\in M$ with respect to $g$ and ``$con$'' is an abbreviation of ``conical''.\\

\item Let $X$ be a smooth vector field on $M$. We define
$$D^{k+2}_{X}(M,E):=\{h\in \cap_{p\geq 1}W^{k+2,p}_{loc}(M,E): h\in C_{con}^{k}(M,E)\quad\textrm{and}\quad \Delta_{X} h\in C_{con}^{k}(M,E)\}$$ which we equip with the norm
$$\|h\|_{D^{k+2}_{X}(M,E)}:=\|h\|_{C_{con}^{k}(M,E)}+\|\Delta_{X}h\|_{C^{k}_{con}(M,E)}.$$
We also define
$$D^{k+2,\alpha}_{X}(M,E):=\{h\in C^{k+2,\alpha}_{loc}(M,E): h\in C_{con}^{k,\alpha}(M,E)\quad\textrm{and}\quad\Delta_{X} h\in C_{con}^{k,\alpha}(M,E)\}$$ which we equip with the norm $$\|h\|_{D^{k+2,\alpha}_{X}(M,E)}:=\|h\|_{C_{con}^{k,\alpha}(M,E)}+\|\Delta_{X} h\|_{C^{k,\alpha}_{con}(M,E)}.$$\\

\item Let $w:M\rightarrow\mathbb{R}_+$ be a smooth function on $M$. Then we define
\begin{equation*}
C_{con,w}^{k,\alpha}(M,E):=w^{-1}C_{con}^{k,\alpha}(M,E).
\end{equation*}
We endow this space with the norm $$\|h\|_{C_{con,w}^{k,\alpha}(M,E)}:=\|wh\|_{C_{con}^{k,\alpha}(M,E)}.$$
Similarly, we define
\begin{eqnarray*}
D^{k+2,\alpha}_{w,X}(M,E):=w^{-1}\cdot D^{k+2,\alpha}_{X}(M,E)
\end{eqnarray*}
which we endow with the norm $$\|h\|_{D^{k+2,\alpha}_{w,X}(M,E)}:=\|wh\|_{D^{k+2,\alpha}_{X}(M,E)}.$$\\

\item Finally, we define the important spaces
\begin{equation*}
\begin{split}
\mathcal{D}^{k+2,\alpha}_{w,X}(M)&:=D^{k+2,\alpha}_{w,X}(M)\cap  \{\varphi\in C^{1}_{loc}(M): \mathcal{L}_{JX}(\varphi)=0\},\\
\mathcal{D}^{\infty}_{w,X}(M)&:=\bigcap_{k\geq 0}\mathcal{D}^{k,\alpha}_{w,X}(M),\\
\mathcal{C}^{k,\alpha}_{w,X}(M)&:=C^{k,\alpha}_{con,w}(M)\cap\{\varphi\in C^1_{loc}(M): \mathcal{L}_{JX}(\varphi)=0\},\\
\mathcal{K}^{k+2,\alpha}_{w,X}&:=\{\varphi\in C^2_{loc}(M):\omega+i\partial\bar{\partial}\varphi>0\}\cap \mathcal{D}^{k+2,\alpha}_{w,X}(M).\\
\end{split}
\end{equation*}
\end{itemize}

In our setting, the weight $w$ will be chosen to be the anticipated potential function $f$ of the soliton, hence it will be quadratic in the distance from a fixed point of $M$ (cf.~Lemma \ref{id-EGS}). Furthermore, this choice stems from the fact that if the asymptotic cone is not Ricci-flat, then the convergence rate of the Ricci soliton at infinity is polynomial of order precisely two \cite{Der-Asy-Com-Egs}, hence our solution of the complex Monge-Amp\`ere equation must lie in a function space such that this is the case. If the cone at infinity is Ricci-flat, then the convergence rate to the asymptotic cone is exponential. This is the case considered by Siepmann \cite{Siepmann} and he sets $w=e^{f}$ to reflect this fact. However, this choice doesn't yield the optimal rate of convergence of his solitons. For this, one may also consult \cite{Der-Asy-Com-Egs}.

\begin{remark}
We remark that the spaces $C_{con}^{k,\alpha}(M,E)$ are not equal to the interpolation spaces $(C^k_{con}(M,E),C^{k+1}_{con}(M,E))_{\alpha,\infty}$. Indeed, one can make the identification
\begin{equation*}
(C^k_{con}(M,E),C^{k+1}_{con}(M,E))_{\alpha,\infty}=\left\{h\in C^k_{con}(M,E) : \left[(r_p^2+1)^{k/2}\nabla^kh\right]_{con,\alpha}<+\infty\right\},
\end{equation*}
where
\begin{equation*}
\left[H\right]_{con,\alpha}:=\sup_{x\in M}\sup_{y\in B(x,\delta r_p(x))\setminus\{x\}}\min\left\{r_p(x)^{\alpha},r_p(y)^{\alpha}\right\}\frac{\arrowvert H(x)-P_{x,y}^*H(y)\arrowvert}{d(x,y)^{\alpha}}
\end{equation*}
for a tensor $H$ on $M$. Here, $\inj(x,g)$ denotes the injectivity radius at $x\in M$ of the metric $g$ and $\delta>0$ is a fixed positive constant depending only on a lower bound of $\inf_{x\in M}\inj(x,g)/r_p(x)$.

Later we will need to use the fact that $C^{k,\alpha}_0(M)\subset C^{k,\alpha}_{con}(M)$ for any $k\geq 0$ and $\alpha\in(0,1)$, where $C^{k,\alpha}_0(M)$ is defined as in Definition \ref{s:linear_analysis}.
\end{remark}

We now turn our attention to Theorem D. Let $C_{0}$ be a K\"ahler cone of complex dimension $n$ with complex structure $J_{0}$, K\"ahler cone metric $g_{0}$, and radial function $r$, and let $\pi:M\to C_0$ be an equivariant resolution of $C_{0}$ with $X=\pi^{*}r\partial_{r}$ denoting the lift of the vector field $r\partial_{r}$ on $C_{0}$ to $M$. Denote by $J$ the complex structure on $M$ and define the function $f:M\rightarrow\mathbb{R}$ to be any positive extension of the pull-back by $\pi$ of the radial function $r^2/2$. Furthermore, let $g$ be an AC K\"ahler metric on $M$ asymptotic to $g_{0}$ with associated K\"ahler form $\omega$. We state an easy but fundamental lemma concerning the asymptotics of the derivatives of $f$ with respect to $g$.
\begin{lemma}\label{prop-basic-est-pot-fct}
In the above setting, we have
\begin{eqnarray*}
\nabla^gf=(1+\textit{O}(r^{-2}))X\quad\textrm{and}\quad\Delta_{\omega}f=n+\textit{O}(r^{-2})\quad\textrm{with $g$-derivatives}.
\end{eqnarray*}
In other words,
\begin{equation*}
\nabla^gf-X\in C^{\infty}_{con,f^{\frac{1}{2}}}(M)\quad\textrm{and}\quad\Delta_{\omega}f-n\in C^{\infty}_{con,f}(M).
\end{equation*}
\end{lemma}

We now wish to prove Theorem D, a precise version of which may now be written as follows.
\begin{customthm}{D}[Existence, PDE precise version]
In the above situation, let $F\in \mathcal{C}^{\infty}_{f,X}(M)$, i.e., $F$ decays quadratically at infinity together with its derivatives.
Then there exists a unique K\"ahler potential $\varphi:M\rightarrow\mathbb{R}$ in $\mathcal{D}^{\infty}_{f,X}(M)$ satisfying the complex Monge-Amp\`ere equation
\begin{eqnarray}
\log\left(\frac{\omega_{\varphi}^n}{\omega^n}\right)=\varphi-\frac{X}{2}\cdot\varphi+F,\label{MA-bis}
\end{eqnarray}
or equivalently,
\begin{eqnarray}
\omega_{\varphi}^n=e^{\varphi-\frac{X}{2}\cdot\varphi+F}\omega^n,\label{MA}
\end{eqnarray}
where $$\omega_{\varphi}:=\omega+i\partial\bar{\partial}\varphi.$$
\end{customthm}
We prove this theorem in the subsequent Sections $5-8$.

\newpage
\section{Existence of small deformations}\label{section-small-def}

\subsection{Preliminaries and Fredholm properties of the linearized operator }
Define the following map as in \cite{Siepmann}:
\begin{eqnarray*}
MA:\varphi\in\{\phi\in C^2_{loc}(M):\omega_{\phi}:=\omega+i\partial\bar{\partial}\phi>0\}\mapsto\log\left(\frac{\omega_{\varphi}^n}{\omega^n}\right)+\frac{X}{2}\cdot\varphi-\varphi\in\mathbb{R}.
\end{eqnarray*}
Brute force computations then show that
\begin{eqnarray}
&&MA(0)=0,\nonumber\\
 &&D_{\varphi}MA(\psi)=\Delta_{\omega_{\varphi}}\psi+\frac{X}{2}\cdot\psi-\psi,\nonumber\\
 &&D^2_{\varphi}MA(\psi,\psi)=-\arrowvert i \partial\bar{\partial}\psi\arrowvert^2_{g_{\varphi}},\label{equ:sec-der}\\
 &&MA(\varphi)=\Delta_{\omega}\varphi+\frac{X}{2}\cdot\varphi-\varphi-\int_0^1\int_0^{\tau}\arrowvert i\partial\bar{\partial}\varphi\arrowvert^2_{g_{\sigma\varphi}}d\sigma d\tau,\label{equ:taylor-exp}
\end{eqnarray}
for any $\psi\in C_{loc}^2(M)$, where $g_{\varphi}$ (respectively $g_{\sigma\varphi}$) denotes the K\"ahler metric associated to the K\"ahler form $\omega_{\varphi}$ (resp.~$\omega_{\sigma\varphi}$ for any $\sigma\in[0,\,1]$) for $\varphi$ as above.

We summarise from \cite{Der-Smo-Pos-Cur-Con} the main result we need. A straightforward inspection of the proof of \cite[Theorem 2.15]{Der-Smo-Pos-Cur-Con} yields the following theorem.
\begin{theorem}\label{iso-sch-Laplacian}
In the above setting, let $\alpha\in[0,1)$ and $k\in\mathbb{N}$. Then
\begin{eqnarray*}
\Delta_{\omega}+\frac{X}{2}\cdot-\operatorname{Id}:\mathcal{D}^{k+2,\alpha}_{f,X}(M)\rightarrow \mathcal{C}^{k,\alpha}_{f,X}(M)
\end{eqnarray*}
is an isomorphism of Banach spaces. Moreover, the following holds.
\begin{itemize}
\item For $\alpha\in(0,1)$, the space
\begin{eqnarray*}
\cit_{f,X}^{k;1,\alpha}(M)&:=&\left\{\psi\in C^{k+1,\alpha}_{loc}(M): f^{i/2}\nabla^{g,i}(f\psi)\in C^{1,\alpha}(M)\quad \forall i=0,...,k,\quad\textrm{and}\quad\mathcal{L}_{JX}(\psi)=0\right\}
\end{eqnarray*}
embeds continuously in $\mathcal{D}^{k+2}_{f,X}(M)$.\\

\item There exists a positive constant $C$ such that, for $\alpha\in(0,1)$,
\begin{eqnarray*}
\|\varphi\|_{\cit^{k;2,\alpha}_{f,X}(M)}\leq C\left\|\Delta_{\omega}\varphi+\frac{X}{2}\cdot\varphi-\varphi\right\|_{\mathcal{C}^{k,\alpha}_{f,X}(M)},
\end{eqnarray*}
where
\begin{eqnarray*}
\cit^{k;2,\alpha}_{f,X}(M)&:=&\left\{\psi\in C^{k+2,\alpha}_{loc}(M): f^{i/2}\nabla^{g,i}(f\psi)\in C^{2,\alpha}(M)\quad \forall i=0,...,k,\quad\textrm{and}\quad\mathcal{L}_{JX}(\psi)=0\right\}.
\end{eqnarray*}
\end{itemize}
\end{theorem}

\begin{proof}
Since our notation is slightly different to that in \cite{Der-Smo-Pos-Cur-Con}, we recall the main steps.
\begin{itemize}
\item $\Delta_{\omega}+\frac{X}{2}\cdot-\Id: \mathcal{D}^{k+2,\alpha}_{f,X}(M)\rightarrow \mathcal{C}^{k,\alpha}_{f,X}(M)$, $k\geq 0$, $\alpha\in(0,1)$, is injective.\\

Actually, we can do better than this. The argument we implement here shall be used throughout the proof of Theorem D.
\begin{claim}\label{claim-injectivity-max-ppe}
Let $\varphi\in C^2_{loc}(M)$ be a function on $M$ such that
\begin{eqnarray*}
\varphi=\textit{O}(f^{a})\quad\textrm{and}\quad\Delta_{\omega}\varphi+\frac{X}{2}\cdot\varphi-b\varphi=0
\end{eqnarray*}
for some non-negative real numbers $a,\,b,$ with $\min\{1,b\}>a$. Then $\varphi\equiv 0$.
\end{claim}
\begin{proof}[Proof of Claim \ref{claim-injectivity-max-ppe}]
Consider the function $\varphi-K^{-1} f^{a+\epsilon}$, where $K>0$ and $\epsilon\in[0,\,1]$ are such that $a+\epsilon\leq \min\{1,b\}$. Then
\begin{eqnarray*}
\left(\Delta_{\omega}+\frac{X}{2}\cdot\right)\left(\varphi-K^{-1} f^{a+\epsilon}\right)&\geq& b\varphi-K^{-1}\left(\Delta_{\omega}+\frac{X}{2}\cdot\right)\left(f^{a+\epsilon}\right)\\
&\geq&b\left(\varphi-K^{-1} f^{a+\epsilon}\right)+K^{-1}\left(bf^{a+\epsilon}-\left(\Delta_{\omega}+\frac{X}{2}\cdot\right)f^{a+\epsilon}\right).
\end{eqnarray*}
Now, by Lemma \ref{prop-basic-est-pot-fct}, we have that
\begin{eqnarray*}
\left|\left(\Delta_{\omega}+\frac{X}{2}\cdot\right)f^{a+\epsilon}-(a+\epsilon)f^{a+\epsilon}\right|\leq Cf^{a+\epsilon-1}\leq C'
\end{eqnarray*}
for some positive constants $C$ and $C'$ depending only on $\epsilon\in[0,1]$. Indeed,
\begin{eqnarray*}
\frac{X}{2}\cdot f^{a+\epsilon}&=&\frac{a+\epsilon}{2}f^{a+\epsilon-1}|\nabla^g_{\mathbb{R}}f|^2_g(1+\textit{O}(f^{-1}))\\
&=&(a+\epsilon)f^{a+\epsilon-1}|\nabla^gf|^2(1+\textit{O}(f^{-1}))\\
&=&(a+\epsilon)f^{a+\epsilon}(1+\textit{O}(f^{-1}))
\end{eqnarray*}
so that
\begin{eqnarray*}
\left(\Delta_{\omega}+\frac{X}{2}\cdot\right)\left(\varphi-K^{-1} f^{a+\epsilon}\right)&\geq&b\left(\varphi-K^{-1} f^{a+\epsilon}\right)-C'K^{-1}.
\end{eqnarray*}
By the growth assumption on $\varphi$, we know that $\lim_{x\to+\infty}(\varphi-K^{-1}f^{a+\epsilon})(x)=-\infty$. Hence \linebreak $\varphi-K^{-1}f^{a+\epsilon}$ attains its maximum on $M$ at some point $x_{K,\epsilon}\in M$ say. The maximum principle then implies that
\begin{eqnarray*}
\max_M(\varphi-K^{-1}f^{a+\epsilon})\leq \frac{C'}{bK}
\end{eqnarray*}
for some positive constant $C'$ depending only on $\epsilon$. Consequently, $\sup_M\varphi\leq 0$ by letting $K$ tend to $+\infty$. By considering $-\varphi$ also, one arrives at the desired result.
\end{proof}

On setting $a=0$ and $b=1$ in the previous claim, we obtain the desired injectivity of $\Delta_{\omega}+\frac{X}{2}\cdot-\Id$ on bounded functions.\\

 \item For any $k\geq 0$ and $\alpha\in(0,1)$, $\Delta_{\omega}+\frac{X}{2}\cdot-\Id: D^{2+k,\alpha}_{f,X}(M)\rightarrow C^{k,\alpha}_{f,X}(M)$ is a Fredholm operator of index $0$ with the corresponding embedding results concerning the spaces $\cit_{f,X}^{k;1,\alpha}(M)$ and $\cit_{f,X}^{k;2,\alpha}(M)$ (the so-called ``rescaled'' Schauder estimates).\\

The main idea here involves conjugating the operator $\Delta_{\omega}+\frac{X}{2}\cdot-\Id$ with the weight $f^{-1}$. This reduces the analysis to the study of a translation of this operator by a negative constant acting on functions (or tensors) in $C_{con}^{k,\alpha}(M)$ up to a compact perturbation. Indeed, we have
\begin{eqnarray*}
&&f\left(\Delta_{\omega}+\frac{X}{2}\cdot-\Id\right)(f^{-1}\psi)=\left(\Delta_{\omega}+\frac{X}{2}\cdot-2\Id\right)\psi+K\psi,\quad \psi\in C^2_{loc}(M),\\
\end{eqnarray*}
where, by Lemma \ref{prop-basic-est-pot-fct}, $K\in C^{\infty}_{con,f}(M)$ acts by multiplication. It is not difficult to check that $K: D_X^{2+k,\alpha}(M)\rightarrow C_{con}^{k,\alpha}(M)$ is a compact operator for any $k\geq 0$ and $\alpha\in(0,1)$. The remaining thing that needs to be checked is that the operator $\Delta_{\omega}+\frac{X}{2}\cdot-2\Id: D_X^{2+k,\alpha}(M)\rightarrow C_{con}^{k,\alpha}(M)$ is an isomorphism of Banach spaces with the corresponding rescaled Schauder estimates. The proof of this fact is a combination of the proofs of Theorems $2.2$ and $2.15$ in \cite{Der-Smo-Pos-Cur-Con}. \\

 \item The final thing that needs to be checked is that the operator $\Delta_{\omega}+\frac{X}{2}\cdot-\Id$ stays surjective when restricted to the set of $JX$-invariant functions. This essentially follows from the uniqueness established in Claim \ref{claim-injectivity-max-ppe}. Indeed, let $F\in \mathcal{C}^{k,\alpha}_{f,X}(M)$, let $\varphi\in D^{k+2,\alpha}_{f,X}(M)$ be a solution to
     \begin{equation}\label{ele}
     (\Delta_{\omega}+\frac{X}{2}\cdot-\Id)\varphi=F,
     \end{equation}
     and let $(\psi_t)_t$ be the flow generated by $JX$. Then, since $F$, $X$, and $JX$, are $JX$-invariant, the function $\varphi_t:=\psi_t^*\varphi$ also satisfies \eqref{ele}. Consequently, $\varphi-\varphi_t$ lies in the kernel of $\Delta_{\omega}+\frac{X}{2}\cdot-\Id$. Since $\varphi-\varphi_t$ is clearly bounded, we see from Claim \ref{claim-injectivity-max-ppe} that $\varphi_t=\varphi$ for every $t\in\mathbb{R}$. In other words, $\varphi$ is $JX$-invariant, as claimed.
\end{itemize}
\end{proof}

\subsection{Implicit deformations of expanding K\"ahler-Ricci solitons}
In this section, we will prove the following theorem.
\begin{theorem}\label{Imp-Def-Kah-Exp}
Let $F_0\in \mathcal{C}^{1,\alpha}_{f,X}(M)$ for some $\alpha\in(0,1)$ and let $\varphi_0\in\mathcal{D}^{3,\alpha}_{f,X}(M)$ be a solution to the complex Monge-Amp\`ere equation
\begin{eqnarray*}
\log\left(\frac{\omega^n_{\varphi_0}}{\omega^n}\right)=-\frac{X}{2}\cdot\varphi_0+\varphi_0+F_0.
\end{eqnarray*}
Then there exists a neighborhood $U_{F_0}\subset\mathcal{C}^{1,\alpha}_{f,X}(M)$ of $F_{0}$ in $\mathcal{C}^{1,\alpha}_{f,X}(M)$ such that for all $F\in U_{F_0}$, there exists a unique solution $\varphi\in\mathcal{D}^{3,\alpha}_{f,X}(M)$ such that
\begin{eqnarray*}
\log\left(\frac{\omega^n_{\varphi}}{\omega^n}\right)=-\frac{X}{2}\cdot\varphi+\varphi+F.
\end{eqnarray*}
\end{theorem}

\begin{proof}
In order to apply the implicit function theorem for Banach spaces, we must re-interpret the statement of Theorem \ref{Imp-Def-Kah-Exp} in terms of the map $MA$ introduced formally at the beginning of this section. Hence consider the mapping
\begin{eqnarray*}
\widetilde{MA}:(\varphi,F)\in \mathcal{K}^{2+1,\alpha}_{f,X}(M)\times \mathcal{C}^{1,\alpha}_{f,X}(M)\mapsto \log\left(\frac{\omega_{\varphi}^n}{\omega^n}\right)+\frac{X}{2}\cdot \varphi-\varphi-F\in \mathcal{C}^{1,\alpha}_{f,X}(M),\quad \alpha\in(0,1).
\end{eqnarray*}
First note that $\widetilde{MA}$ is well-defined. Indeed, by (\ref{equ:taylor-exp}),
\begin{eqnarray*}
\widetilde{MA}(\varphi,F)=\Delta_{\omega}\varphi+\frac{X}{2}\cdot\varphi-\varphi-F-\int_0^1\int_0^{\tau}\arrowvert i\partial\bar{\partial}\varphi\arrowvert^2_{g_{\sigma\varphi}}d\sigma d\tau.
\end{eqnarray*}
Now, by the very definition of $\mathcal{D}^{2+1,\alpha}_{f,X}(M)$, the first three terms $\Delta_{\omega}\varphi+\frac{X}{2}\cdot\varphi$, $\varphi$, and $F$ are in $\mathcal{C}^{1,\alpha}_{f,X}(M)$. By the rescaled Schauder estimates in Theorem \ref{iso-sch-Laplacian}, $\varphi\in \cit^{1;2,\alpha}_{f,X}(M)$, i.e., $f^{\frac{i}{2}}\nabla^{g,i}(f\varphi)\in C^{2,\alpha}(M)$ for $i=0,1$. In particular, this implies that $f\nabla^{g,2}\varphi\ast \nabla^{g,2}\varphi\in C^{0,\alpha}(M)$, where $\ast$ denotes any linear combination of contractions of tensors (with respect to the metric $g_{\varphi}$ here), and that
\begin{eqnarray*}
f^{\frac{3}{2}}\nabla^g(\nabla^{g,2}\varphi\ast \nabla^{g,2}\varphi)=f^{\frac{3}{2}}(\nabla^{g,3}\varphi\ast \nabla^{g,2}\varphi)\in C^{0,\alpha}(M).
\end{eqnarray*}
That is, $|\partial\bar{\partial}\varphi|^2_{g_{\sigma\varphi}}\in\mathcal{C}^{1,\alpha}_{f,X}(M)$, where the $JX$-invariance is straightforward.

By definition, $\widetilde{MA}(\varphi,F)=0$ if and only if $\varphi$ is a solution to $(\ref{MA})$ with data $F$. By (\ref{equ:sec-der}), $$D_{\varphi_0}\widetilde{MA}(\psi)=\Delta_{\omega_{\varphi_0}}\psi+\frac{X}{2}\cdot\psi-\psi\quad\textrm{for $\psi\in \mathcal{D}^{2+1,\alpha}_{f,X}(M)$}.$$ Hence, by Theorem \ref{iso-sch-Laplacian}, $D_{\varphi_0}\widetilde{MA}$ is an isomorphism of Banach spaces. The result now follows by applying the implicit function theorem to the map $\widetilde{MA}$ in a neighborhood of $(\varphi_0,F_0)$.
\end{proof}

\newpage
\section{A priori estimates}\label{section-a-priori-est}

In this section, we establish crucial a priori weighted $\mathcal{D}^{3,\alpha}_{f,X}(M)$-estimates for a smooth solution to (\ref{MA}) with data $F\in \mathcal{C}^{1,\alpha}_{f,X}(M)$. This is the content of Theorem \ref{theo-C^3-wei-est} below. We fix $F\in \mathcal{C}^{1,\alpha}_{f,X}(M)$ and we consider a K\" ahler potential $\varphi$ that is a solution in $\mathcal{D}^{3,\alpha}_{f,X}(M)$ to (\ref{MA}). To see that one can make sense of the higher derivatives of a solution $\varphi\in \mathcal{K}^{3,\alpha}_{f,X}(M)$ when $F\in C^{k,\alpha}_{loc}(M)$, $k\geq 1$, we refer the reader to Proposition \ref{prop-loc-reg} where the desired local regularity of such solutions is established.

\subsection{$C^0$ a priori estimates}

\subsubsection{$C^0$-estimate on the potential $\varphi$}
\begin{prop}\label{prop-C^0-est}
We have the following $C^0$ a priori estimate:
\begin{eqnarray*}
\|\varphi\|_{C^0}\leq \| F\|_{C^0}.
\end{eqnarray*}
\end{prop}
\begin{proof}
The proof is standard and only requires the maximum principle.

Assume that $\sup_M\varphi >0$; otherwise, there is nothing to prove. Then, since $\varphi$ tends to zero at infinity, it attains its maximum at some point $x_0\in M$. At this point, $\omega_{\varphi}(x_0)\leq \omega(x_0)$, which implies by (\ref{MA}) that
\begin{eqnarray*}
\max_M\varphi=\varphi(x_0)\leq -F(x_0)\leq \| F\|_{C^0}.
\end{eqnarray*}
The same argument applied to $-\varphi$ establishes the result.
\end{proof}

\subsubsection{$C^0$-weighted estimate on the potential $\varphi$}

\begin{prop}\label{prop-a-prio-wei-C0-est}
We have the following  weighted $C^0$ a priori estimate:
\begin{eqnarray*}
\|\varphi\|_{C^0_{con,f}(M)}\leq C\left(\| F\|_{C^0_{con,f}(M)}\right),
\end{eqnarray*}
where $C\left(\| F\|_{C^0_{con,f}(M)}\right)$ is bounded by a constant $C(\Lambda)$ depending only on an upper bound $\Lambda$ of $\| F\|_{C^0_{con,f}(M)}$.
\end{prop}

\begin{proof}
We begin with an upper bound for $\|\varphi\|_{C^0_{con,f}(M)}$. First notice that, by (\ref{equ:taylor-exp}), $\varphi$ satisfies the following differential inequality:
\begin{eqnarray*}
F&=&\log\left(\frac{\omega_{\varphi}^n}{\omega}\right)+\frac{X}{2}\cdot\varphi-\varphi\\
&\leq& \Delta_{\omega}\varphi+\frac{X}{2}\cdot \varphi-\varphi.
\end{eqnarray*}
Hence, outside a compact set $K$ independent of $\varphi$, we have that
\begin{eqnarray*}
\Delta_{\omega}(f\varphi)+\frac{X}{2}\cdot(f\varphi)-(f\varphi)&=&\left(\Delta_{\omega}f+\frac{X}{2}\cdot f\right)\varphi+2\nabla^gf\cdot\varphi+f\left(\Delta_{\omega}\varphi+\frac{X}{2}\cdot\varphi-\varphi\right)\\
&\geq& fF+2\nabla^g\ln f\cdot(f\varphi)+\left(\frac{\Delta_{\omega}f}{f}+\frac{X}{2}\cdot \ln f-2\nabla^g\ln f\cdot \ln f\right)(f\varphi).
\end{eqnarray*}
Now, on one hand, by Lemma \ref{prop-basic-est-pot-fct}, we know that
\begin{eqnarray*}
\frac{\Delta_{\omega}f}{f}+\frac{X}{2}\cdot \ln f-2\nabla^g\ln f\cdot \ln f=1+\textit{O}(f^{-1}),
\end{eqnarray*}
whereas on the other hand, we have that
 \begin{eqnarray*}
\left(\Delta_{\omega}+\left(\frac{X}{2}-2\nabla^g\ln f\right)\right)(\ln f)=1+\textit{O}(f^{-1}).
\end{eqnarray*}
So consider the function $f\varphi-\epsilon\ln f$ which tends to $-\infty$ at infinity since $f\varphi$ is bounded by assumption. At a point where $f\varphi-\epsilon\ln f$ achieves its maximum (outside of the compact set $K$; otherwise, there is nothing to prove), one has, by the maximum principle, the following inequality:
\begin{eqnarray*}
0&\geq& \left(\Delta_{\omega}+\left(\frac{X}{2}-2\nabla^g\ln f\right)\right)(f\varphi-\epsilon \ln f)\geq (1+\textit{O}(f^{-1}))(f\varphi)-\epsilon (1+\textit{O}(f^{-1}))-\|fF\|_{C^0}\\
&\geq& C\max_M(f\varphi-\epsilon\ln f)-\|fF\|_{C^0}-\epsilon (1+\textit{O}(f^{-1}))
\end{eqnarray*}
for some positive constant $C$ independent of $\epsilon$ and depending only on the compact set $K$. From this, we deduce that
\begin{eqnarray*}
f\varphi\leq C^{-1}(\|f F\|_{C^0}+c\epsilon)+\epsilon\ln f
\end{eqnarray*}
for some positive constant $c$ independent of $\epsilon$. This yields the desired upper bound if we now let $\epsilon$ tend to zero.

As for the lower bound, consider the function $\chi_{\kappa}:=-\kappa f^{-1}$ for some $\kappa>0$ to be defined later. One has the following estimate:
 \begin{eqnarray*}
\omega_{\chi_{\kappa}}&=&\omega+\frac{\kappa}{f}\left(\frac{i\partial\bar{\partial}f}{f}-2i\frac{\partial f\wedge\bar{\partial}f}{f^2}\right)\\
&\geq&(1-c\kappa f^{-2})\omega\\
\end{eqnarray*}
on $\{f^2\geq 2c\kappa\}$, where $c$ is a universal constant. Hence,
\begin{eqnarray*}
\log\left(\frac{\omega_{\chi_{\kappa}}^n}{\omega^n}\right)+\frac{X}{2}\cdot
\chi_{\kappa}-\chi_{\kappa}&\geq& \frac{\kappa}{f}\left(2-\textit{O}(f^{-1})\right)+n\log(1-c\kappa f^{-2})\\
&\geq&\frac{\kappa}{f}\left(2-\textit{O}(f^{-1})\right)-C(n)\kappa f^{-2}\\
&\geq&\frac{\kappa}{f}
\end{eqnarray*}
on $\{f^2\geq c\max\{\kappa,C(n)\}\}$, where $c$ is now a positive constant independent of $\kappa$ that can vary from line to line.

Next observe that on $\{f^2\geq c\max\{\kappa,C(n)\}\}$,
\begin{eqnarray*}
\log\left(\frac{(\omega_{\varphi}+i\partial\bar{\partial}(\chi_{\kappa}-\varphi))^n}{\omega_{\varphi}^n}\right)+\frac{X}{2}\cdot(\chi_{\kappa}-\varphi)-(\chi_{\kappa}-\varphi)&=&\log\left(\frac{\omega_{\chi_{\kappa}}^n}{\omega_{\varphi}^n}\right)+\frac{X}{2}\cdot(\chi_{\kappa}-\varphi)-(\chi_{\kappa}-\varphi)\\
&\geq&\frac{\kappa}{f}-F\\
&>&0
\end{eqnarray*}
if $\kappa\geq 2\|fF\|_{C^0}$. Thus, if $\kappa\geq 2\|fF\|_{C^0}$, then one has the following bound for any height $R$ with $R\geq c\sqrt{\kappa}$:
\begin{eqnarray*}
\max_{\{f\geq R\}}(\chi_{\kappa}-\varphi)=\max\left\{\max_{\{f=R\}}(\chi_{\kappa}-\varphi),0\right\}.
\end{eqnarray*}
Now, by Proposition \ref{prop-C^0-est},
\begin{eqnarray*}
\max_{\{f=R\}}(\chi_{\kappa}-\varphi)\leq \|\varphi\|_{C^0}-\kappa R^{-1}\leq \|F\|_{C^0}-\kappa R^{-1}\leq 0
\end{eqnarray*}
once $\kappa \geq \|F\|_{C^0}R$. We therefore define $\kappa:=\max\{c^2\| F\|^2_{C^0},\|fF\|_{C^0}\}$ and $R:=c\sqrt{\kappa}$ so that on $\{f\geq c\max\{c\| F\|_{C^0},\|fF\|^{1/2}_{C^0}\}\}$,
\begin{eqnarray*}
-f\varphi\leq \max\{c^2\| F\|^2_{C^0},\|fF\|_{C^0}\}.
\end{eqnarray*}
This yields the desired lower bound.
\end{proof}

\newpage
\subsubsection{$C^0$-estimate on the radial derivative $X\cdot \varphi$}
\begin{prop}\label{prop-c^0-rad-der}
We have the following $C^{0}$ a priori estimate on the radial derivative $X\cdot \varphi$:
\begin{eqnarray*}
\|X\cdot\varphi\|_{C^0(M)}\leq C\left(\| F\|_{C^0_{con,f}(M)}+1\right)
\end{eqnarray*}
for some positive constant $C$ independent of $\varphi$.
\end{prop}

\begin{remark}
The quadratic decay at infinity and $JX$-invariance of $\varphi$ are both essential for the proof of Proposition \ref{prop-c^0-rad-der}. The proof as written would fail without either one of these two assumptions.
\end{remark}

\begin{proof}
The proof is due to Siepmann in the case that the convergence rate to the asymptotic cone is exponential; see \cite[Section 5.4.14]{Siepmann}. We adapt his proof here to the case where the convergence rate is only polynomial.

The proof shall comprise two parts. The first part will concern an upper bound for $X\cdot\varphi$ and the latter part will concern a lower bound for $X\cdot\varphi$. Before proceeding with the first part though, we make the following claim.
\begin{claim}\label{est-sec-der-vec-fiel}
One has
$$X^{1,0}\cdot (X^{1,0}\cdot \varphi)=i\partial\bar{\partial}\varphi(\Re(X^{1,0}),J\Re(X^{1,0}))\geq -\arrowvert \Re(X^{1,0})\arrowvert_{g}^2.$$
\end{claim}
\begin{proof}[Proof of Claim \ref{est-sec-der-vec-fiel}]
Since $\varphi$ is invariant under the flow of $JX$, we know that $$JX\cdot (X\cdot \varphi)=0.$$
In particular, we have that $X^{1,0}\cdot(X^{1,0}\cdot \varphi)=\Re(X^{1,0})\cdot(\Re(X^{1,0})\cdot\varphi)=\overline{X^{1,0}}\cdot(X^{1,0}\cdot\varphi)$. A straightforward computation then shows that
\begin{eqnarray*}
\overline{X^{1,0}}\cdot(X^{1,0}\cdot \varphi)=\partial\bar{\partial}\varphi(X^{1,0},\overline{X^{1,0}})=i\partial\bar{\partial}\varphi(\Re(X^{1,0}),J\Re(X^{1,0})).
\end{eqnarray*}
The result now follows from the fact that $\omega_{\varphi}> 0$ and
\begin{eqnarray*}
i\partial\bar{\partial}\varphi(\Re(X^{1,0}),J\Re(X^{1,0}))=\omega_{\varphi}(\Re(X^{1,0}),J\Re(X^{1,0}))-\arrowvert \Re(X^{1,0})\arrowvert_{g}^2\geq -\arrowvert \Re(X^{1,0})\arrowvert_{g}^2.
\end{eqnarray*}
\end{proof}
To get an upper bound for $X\cdot \varphi$, we introduce the flow $(\psi_t)_{t\in\mathbb{R}}$ generated by the vector field $X/2$. This flow is complete since $X$ grows linearly at infinity. By Lemma \ref{prop-basic-est-pot-fct}, we have the following estimates at $(x,\,t)\in M\times\mathbb{R}$:
\begin{equation*}
\partial_t f(\psi_t(x))=\left(\frac{X}{2}\cdot f\right)(\psi_t(x))=f(\psi_t(x))+\textit{O}(1)
\end{equation*}
and
\begin{equation*}
\partial_t\arrowvert X\arrowvert^2_{\omega}(\psi_t(x))\leq\arrowvert\nabla^gX\arrowvert_{\omega}(\psi_t(x))\arrowvert X\arrowvert^2_{\omega}(\psi_t(x))\leq C \arrowvert X\arrowvert^2_{\omega}(\psi_t(x))
\end{equation*}
for some uniform positive constant $C$ independent of $(x,\,t)\in M\times\mathbb{R}$, where, in the latter inequality, we have used the fact that $\nabla^gX$ is uniformly bounded on $M$. Gronwall's inequality then implies that
\begin{eqnarray}
f(\psi_t(x))+C&\leq& e^t(f(x)+C), \label{Gronwall-inequ-1}\\
f(\psi_t(x))-C&\geq& e^t(f(x)-C),\nonumber\\
\arrowvert X\arrowvert^2_{\omega}(\psi_t(x))&\leq& e^{Ct}\arrowvert X\arrowvert^2_{\omega}(x),\label{Gronwall-inequ-2}
\end{eqnarray}
for any space-time point $(x,t)\in M\times\mathbb{R}.$

Next, define $\varphi_x(t):=\varphi(\psi_t(x))$ for $(x,t)\in M\times\mathbb{R}.$ Then for any cut-off function $\eta:\mathbb{R}_+\rightarrow[0,1]$ such that $\eta(0)=1$, $\eta'(0)=0$, we have that
\begin{eqnarray*}
\int_0^{+\infty}\eta''(t)\varphi_x(t)dt&=&-\int_0^{+\infty}\eta'(t)\varphi_x'(t)dt\\
&=&\varphi_x'(0)+\int_0^{+\infty}\eta(t)\varphi_x''(t)dt.
\end{eqnarray*}
It then follows from (\ref{Gronwall-inequ-2}) and Claim \ref{est-sec-der-vec-fiel} that
\begin{eqnarray*}
\frac{X}{2}\cdot\varphi(x)&=&\varphi_x'(0)\leq -\int_{\supp(\eta)}\frac{X}{2}\cdot \left(\frac{X}{2}\cdot \varphi\right)(\psi_t(x))dt+\sup_{t\in\supp(\eta'')}\arrowvert\varphi_x(t)\arrowvert\int_{\supp(\eta'')}\arrowvert\eta''(t)\arrowvert dt\\
&\leq&\frac{1}{4}\int_{\supp(\eta)}\arrowvert X\arrowvert^2_{g}(\psi_t(x))dt+\sup_{t\in\supp(\eta'')}\arrowvert\varphi_x(t)\arrowvert\int_{\supp(\eta'')}\arrowvert\eta''(t)\arrowvert dt\\
&\leq& \frac{1}{4}\arrowvert X\arrowvert^2_{g}(x)\int_{\supp(\eta)}e^{Ct}dt+\sup_{t\in\supp(\eta'')}\arrowvert\varphi_x(t)\arrowvert\int_{\supp(\eta'')}\arrowvert\eta''(t)\arrowvert dt.
\end{eqnarray*}
Choose $\eta$ such that $\supp(\eta)\subset [0,\epsilon]$ for some $\epsilon>0$ to be chosen later and such that $\arrowvert\eta''\arrowvert\leq C/\epsilon^2$ for some uniform positive constant $C$. Then
\begin{eqnarray*}
\frac{X}{2}\cdot\varphi(x)\leq (4C)^{-1}(e^{C\epsilon}-1)\arrowvert X\arrowvert^2_{g}(x)+C\sup_{t\in[0,\epsilon]}\arrowvert\varphi_x(t)\arrowvert\epsilon^{-1}.
\end{eqnarray*}
Since $X$ grows linearly with respect to the distance from a fixed point, this dictates the following choice of $\epsilon$:
\begin{eqnarray*}
\epsilon(x):=C^{-1}\log\left(1+\frac{1}{1+\arrowvert X\arrowvert^2_{g}(x)}\right).
\end{eqnarray*}
It remains to bound $\sup_{t\in[0,\epsilon]}\arrowvert\varphi_x(t)\arrowvert$ in terms of $\epsilon$. Since $\varphi$ decays quadratically at infinity, one has that
\begin{eqnarray*}
\varphi_x(t)=\varphi(\psi_x(t))\leq \frac{\|f\varphi\|_{C^0(M)}}{f(\psi_x(t))}
\end{eqnarray*}
for any $(x,t)\in M\times \mathbb{R}$. Now, thanks to (\ref{Gronwall-inequ-1}),
\begin{eqnarray*}
\varphi_x(t)\leq \frac{\|f\varphi\|_{C^0(M)}}{C+e^t(f(x)-C)}\leq \frac{\|f\varphi\|_{C^0(M)}}{f(x)}
\end{eqnarray*}
for $t\geq 0$ and $f(x)\geq C$ with $C$ a uniform positive constant. Since $\epsilon(x)$ decays precisely as $C^{-1}f(x)^{-1}$ as $x$ tends to $+\infty$, we get the desired upper bound on the radial derivative $\frac{X}{2}\cdot \varphi$. Indeed, if $f(x)\leq C$, then the upper bound follows directly from the boundedness of $X$ in terms of $C$. The lower bound can be proved in a similar way by arguing on an interval $[-\epsilon,0]$, where $\epsilon$ has the same behaviour at infinity as above. The desired estimate now follows from Proposition \ref{prop-a-prio-wei-C0-est}.
\end{proof}

\subsection{$C^2$-estimate}
\begin{prop}\label{prop-C^2-est}
We have the following $C^2$ a priori estimate:
\begin{eqnarray*}
\|\partial\bar{\partial}\varphi\|_{C^0(M)}\leq C\left(n,\omega,\|F\|_{C^2(M)},\|F\|_{C^0_{con,f}(M)}\right).
\end{eqnarray*}
\end{prop}

\begin{remark}
The proof of Proposition \ref{prop-C^2-est} makes use of the equivalence of the metrics $g$ and $g_{\varphi}$ in order to build a suitable barrier function with which to apply the maximum principle. A consequence of this estimate will be that $g$ and $g_{\varphi}$ are in fact uniformly equivalent.
\end{remark}

\begin{proof}
The proof uses standard computations performed in Yau's seminal paper \cite{Calabiconj}. Only the presence of the vector field $X$ has to be taken into account, hence we will only outline the main steps.

Since $\varphi$ satisfies
 \begin{eqnarray*}
\log\left(\frac{\omega_{\varphi}^n}{\omega^n}\right)=F-\frac{X}{2}\cdot\varphi+\varphi=:F(\varphi),
\end{eqnarray*}
one can, as in \cite{Calabiconj}, compute the Laplacian of $F(\varphi)$ with respect to $\omega$ in holomorphic coordinates around a point $x\in M$ such that at $x$, the metrics $g$ and $g_{\varphi}$ associated to $\omega$ and $\omega_{\varphi}$ take the form $g_{i\bar{\jmath}}(x)=\delta_{i\bar{\jmath}}$ and $g_{\varphi}(x)=(1+\varphi_{i\bar{\imath}}(x))\delta_{i\bar{\jmath}}$ respectively; see Lemma \ref{hol-coor}. After a lengthy computation, one arrives at
\begin{eqnarray}
\Delta_{\omega}F(\varphi)=\Delta_{\omega_{\varphi}}(\tr_{\omega}(\omega_{\varphi}))-\frac{\varphi_{i\bar{\jmath}k}\varphi_{\bar{\imath}j\bar{k}}}{(1+\varphi_{i\bar{\imath}})(1+\varphi_{k\bar{k}})}+\Rm(g)_{i\bar{\imath}k\bar{k}}\left(1-\frac{1}{1+\varphi_{i\bar{\imath}}}-\frac{\varphi_{i\bar{\imath}}}{1+\varphi_{k\bar{k}}}\right).\label{eq-sec-der-yau}
\end{eqnarray}
Now,
\begin{eqnarray*}
\Rm(g)_{i\bar{\imath}k\bar{k}}\left(1-\frac{1}{1+\varphi_{i\bar{\imath}}}-\frac{\varphi_{i\bar{\imath}}}{1+\varphi_{k\bar{k}}}\right)&=&\frac{1}{2}\Rm(g)_{i\bar{\imath}k\bar{k}}\left(\frac{(\varphi_{i\bar{\imath}}-\varphi_{k\bar{k}})^2}{(1+\varphi_{i\bar{\imath}})(1+\varphi_{k\bar{k}})}\right)\\
&\geq& \frac{\inf_M\Rm(g)}{2}\left(\frac{1+\varphi_{i\bar{\imath}}}{1+\varphi_{k\bar{k}}}-1\right)\\
&\geq&\inf_M \Rm(g)\left(\tr_g(g_{\varphi}^{-1})(n+\Delta_{\omega}\varphi)-n^2\right)\\
&=&\inf_M \Rm(g)\left(\tr_g(g_{\varphi}^{-1})\tr_{\omega}(\omega_{\varphi})-n^2\right),
\end{eqnarray*}
where $\Rm(g)$ is the complex-linear extension of the curvature operator of the metric $g$ and where $\inf_M\Rm(g):=\inf_{i\neq k}\Rm(g)_{i\bar{\imath}k\bar{k}}$.

Next we study the term $\Delta_{\omega}(X\cdot\varphi)$. Since $X$ is holomorphic, $\nabla^gX$ is bounded, and since $\omega$ and $\varphi$ are $JX$-invariant, we have that
\begin{eqnarray*}
\Delta_{\omega}\left(\frac{X}{2}\cdot\varphi\right)&=&\Delta_{\omega}\left(X^{1,0}\cdot\varphi\right)\\
&=&\nabla^g_i(X^{1,0})^k\varphi_{\bar{\imath}k}+\frac{X^{1,0}}{2}\cdot\Delta_{\omega}\varphi\\
&=&\nabla^gX^{1,0}\ast\partial\bar{\partial}\varphi+\frac{X}{2}\cdot\Delta_{\omega}\varphi\\
&\leq&C\tr_{\omega}(\omega_{\varphi})+C(n)\|\nabla^gX\|_{C^0(M)}+\frac{X}{2}\cdot\tr_{\omega}(\omega_{\varphi}),
\end{eqnarray*}
where we have used the fact that $0<\omega_{\varphi}\leq (n+\Delta_{\omega}\varphi)\omega.$ To summarise, we obtain the following first crucial estimate:
\begin{equation}\label{crucial-est-tr-C2}
\begin{split}
\Delta_{\omega_{\varphi}}\tr_{\omega}(\omega_{\varphi})+\frac{X}{2}\cdot \tr_{\omega}(\omega_{\varphi})\geq& \frac{\varphi_{i\bar{\jmath}k}\varphi_{\bar{\imath}j\bar{k}}}{(1+\varphi_{i\bar{\imath}})(1+\varphi_{k\bar{k}})}\\
&+\Delta_{\omega}F-C\tr_{\omega}(\omega_{\varphi})(1+\inf_M \Rm(g)\tr_g(g_{\varphi}^{-1}))-C(n,g).
\end{split}
\end{equation}

Now, if $u:=e^{-\alpha\varphi}\tr_{\omega}(\omega_{\varphi})$, where $\alpha\in\mathbb{R}$ will be defined later, then, as in the proof of \cite[Lemma 5.4.16]{Siepmann}, one estimates the Laplacian of $u$ with respect to $\omega_{\varphi}$ as follows:
 \begin{eqnarray*}
\Delta_{\omega_{\varphi}}u\geq e^{-\alpha \varphi}\left(\Delta_{\omega}F(\varphi)-\inf_M\Rm(g)\tr_{g}(g_{\varphi}^{-1})\tr_{\omega}(\omega_{\varphi})-C-\alpha \Delta_{\omega_{\varphi}}\varphi\tr_{\omega}(\omega_{\varphi})\right)
\end{eqnarray*}
for some positive constant $C$ independent of $\varphi$. Thus, for some positive constant $C$ independent of $\varphi$,
\begin{eqnarray*}
\Delta_{\omega_{\varphi}}u+\frac{X}{2}\cdot u&\geq& e^{-\alpha\varphi}\left(\Delta_{\omega}F-\inf_M\Rm(g)\tr_g(g_{\varphi}^{-1})\tr_{\omega}(\omega_{\varphi})\right)\\
&&-Ce^{-\alpha \varphi}-\alpha \frac{X}{2}\cdot\varphi u-Cu-\alpha(n-\tr_{g}(g_{\varphi}^{-1}))
u\\
&\geq&-C(\|\varphi\|_{C^0(M)},\|F\|_{C^2(M)})-C(n,\|X\cdot\varphi\|_{C^0(M)})u+\tr_g(g_{\varphi}^{-1})u,
\end{eqnarray*}
where we set $\alpha:=\max\{1+\inf_M\Rm(g),1\}.$ A final estimate using the following geometric inequality
\begin{eqnarray*}
\sum_i\frac{1}{1+\varphi_{i\bar{\imath}}}\geq\left(\frac{\sum_i(1+\varphi_{i\bar{\imath}})}{\Pi_i(1+\varphi_{i\bar{\imath}})}\right)^{\frac{1}{n-1}},
\end{eqnarray*} or equivalently,
\begin{eqnarray*}
\tr_g(g_{\varphi}^{-1})\geq \left(\frac{\tr_g(g_{\varphi})}{\det_g(g_{\varphi})}\right)^{\frac{1}{n-1}},
\end{eqnarray*}
then leads to the following differential inequality satisfied by $u$:
\begin{eqnarray*}
\Delta_{\omega_{\varphi}}u+\frac{X}{2}\cdot u\geq -C(1+u)+C'u^{\frac{n}{n-1}}
\end{eqnarray*}
for some positive constants $C$ and $C'$ depending only on $n$, $\omega$, $\|\nabla^gX\|_{C^0(M)}$, $\|\varphi\|_{C^0(M)}$, $\|X\cdot\varphi\|_{C^0(M)}$, and $\|F\|_{C^2(M)}$. Since $u$ is non-negative and bounded, the maximum principle yields the desired upper bound for $n+\Delta_{\omega}\varphi$.

Indeed, consider the function $u-\epsilon \ln f$ for some positive $\epsilon$. Since $\lim_{x\to+\infty}(u-\epsilon \ln f)(x)=-\infty$, this function attains its maximum on $M$ at some point $x_0$. We can then apply the maximum principle to this function to obtain
\begin{eqnarray*}
(\max_M(u-\epsilon\ln f)+\epsilon\ln f(x_0))^{\frac{n}{n-1}}\leq C+C\epsilon\left\arrowvert\Delta_{\omega_{\varphi}}\ln f+\frac{X}{2}\cdot\ln f\right\arrowvert(x_0)
\end{eqnarray*}
for some uniform positive constant $C$. Now, since $g$ and $g_{\varphi}$ are equivalent, this implies that $\Delta_{\omega_{\varphi}}\ln f $ is bounded on $M$. Moreover, thanks to the asymptotics of $X$, $X\cdot \ln f$ is also bounded on $M$. Consequently,
\begin{eqnarray*}
u-\epsilon\ln f&\leq& (C+C'\epsilon)^{\frac{n-1}{n}}-\epsilon \ln f(x_0)\\
&\leq&(C+C'\epsilon)^{\frac{n-1}{n}}+C\epsilon,\\
\end{eqnarray*}
where $C$ is a uniform positive constant and $C'$ is a constant that may depend on the equivalence class of the metrics $g$ and $g_{\varphi}$. Since this estimate holds for any $\epsilon>0$, one obtains the desired a priori estimate independent of the equivalence class of the metrics $g$ and $g_{\varphi}$. Hence we obtain an upper bound on $\partial\bar{\partial}\varphi$ since $\varphi$, being a K\"ahler potential, implies that $\|\omega_{\varphi}\|_{C^0(M)}\leq \|n+\Delta_{\omega}\varphi\|_{C^0(M)}$.
\end{proof}

\begin{corollary}\label{coro-equiv-metrics-0}
The tensors $g^{-1}g_{\varphi}$ and $g_{\varphi}^{-1}g$ satisfy the following uniform estimate:
\begin{eqnarray*}
\|g^{-1}g_{\varphi}\|_{C^{0}(M)}+\|g_{\varphi}^{-1}g\|_{C^{0}(M)}\leq \Lambda\left(n,\alpha,\|F\|_{C^2(M)},\|F\|_{C^0_{con,f}(M)}\right).
\end{eqnarray*}
In particular, the metrics $g$ and $g_{\varphi}$ are uniformly equivalent.
\end{corollary}

\begin{proof}
By Proposition \ref{prop-C^2-est}, we know that $$\|g^{-1}g_{\varphi}\|_{C^{0}(M)}\leq \Lambda\left(n,\alpha,\|F\|_{C^2(M)},\|F\|_{C^0_{con,f}(M)}\right).$$
Moreover, by Propositions \ref{prop-C^0-est} and \ref{prop-c^0-rad-der}, $g^{-1}g_{\varphi}$ satisfies
\begin{eqnarray*}
\det(g^{-1}g_{\varphi})=e^{F+\varphi-\frac{X}{2}\cdot\varphi}\geq e^{-C}
\end{eqnarray*}
for some uniform positive constant $C$. Finally, each eigenvalue of $g_{\varphi}$ is uniformly bounded from above. Thus, $$\|g_{\varphi}^{-1}g\|_{C^0(M)}\leq\Lambda\left(n,\alpha,\|F\|_{C^2(M)},\|F\|_{C^0_{con,f}(M)}\right).$$
\end{proof}

\subsection{$C^3$-estimate}

\begin{prop}\label{prop-C^3-est}
We have the following $C^3$ a priori estimate:
\begin{eqnarray*}
\|\nabla^g\partial\bar{\partial}\varphi\|_{C^0(M)}\leq C\left(n,\omega,\|F\|_{C^3(M)},\|F\|_{C^0_{con,f}(M)}\right).
\end{eqnarray*}
\end{prop}

\begin{remark}
As observed in \cite{Siepmann}, it is more convenient to adapt the computations of Yau \cite{Calabiconj} in the presence of an unbounded vector field $X$ than to use the machinery developed by Evans, Krylov and Safonov to avoid such a $C^3$-estimate. However, since this computation is tedious and not very enlightening, we will follow the alternative route of \cite{Pho-Ses-Stu} in order to establish a $C^3$ a priori estimate.
\end{remark}

\begin{proof}
Again, this is precisely \cite[Lemma 5.4.20]{Siepmann}. We will give a different proof with a flow flavor that follows closely the arguments of \cite[Chapter $3$]{Bou-Eys-Gue}.

Define
$$S(g_{\varphi},g):=\arrowvert\nabla^gg_{\varphi}\arrowvert^2_{g_{\varphi}}.$$
Then, by the very definition of $S$, we have that
\begin{eqnarray*}
S(g_{\varphi},g)&=&g_{\varphi}^{i\bar{\jmath}}g_{\varphi}^{k\bar{l}}g_{\varphi}^{p\bar{q}}\nabla^{g}_i{g_{\varphi}}_{k\bar{q}}\overline{\nabla^{g}_{j}{g_{\varphi}}_{l\bar{q}}}\\
&=&|\Psi|_{g_{\varphi}}^2,
\end{eqnarray*}
where $\Psi_{ij}^k(g_{\varphi},g):=\Gamma(g_{\varphi})_{ij}^k-\Gamma(g)_{ij}^k$. Now, since $\varphi$ solves (\ref{MA}), $(M,g_{\varphi},X)$ is an ``approximate'' expanding gradient K\"ahler-Ricci soliton in the following precise sense: if $g_{\varphi}(\tau):=(1+\tau)\varphi_{\tau}^*g_{\varphi}$ and $g(\tau):=(1+\tau)\varphi_{\tau}^*g$, where $(\varphi_{\tau})_{\tau>-1}$ is the one-parameter family of diffeomorphisms generated by $-X/(2(1+\tau))$, then, by Section \ref{subsection-existence-sol}, $(g_{\varphi}(\tau))_{\tau>-1}$ is a solution to the following perturbed K\"ahler-Ricci flow with initial condition $g_{\varphi}$:
\begin{eqnarray*}
\partial_{\tau}g_{\varphi}(\tau)&=&-\Ric(g_{\varphi}(\tau))+\varphi_{\tau}^*\left(-\mathcal{L}_{\frac{X}{2}}(g)+g+\Ric(g)+\partial\bar{\partial}F\right)\\
&=&-\Ric(g_{\varphi}(\tau))+\varphi_{\tau}^*\partial\bar{\partial}(F-F(g)),\quad \tau>-1,\\
 g_{\varphi}(0)&=&g_{\varphi}.
\end{eqnarray*}
In particular, $\partial_{\tau}g_{\varphi}=-\Ric(g_{\varphi})+\varphi_{\tau}^*\eta$, where $\eta:=\partial\bar{\partial}(F-F(g))$ has uniformly controlled $C^1$-norm by our assumptions and by construction of $g$.

Define $S(\tau):=S(g_{\varphi}(\tau),g(\tau))$ and correspondingly, $\Psi(\tau):=\Psi(g_{\varphi}(\tau),g(\tau))$. We adapt \cite[Proposition 3.2.8]{Bou-Eys-Gue} to our setting. By a brute force computation, we have that
\begin{eqnarray*}
\Delta_{\omega_{\varphi}}S&=&2\Re\left(g_{\varphi}^{i\bar{\jmath}}g_{\varphi}^{p\bar{q}}{g_{\varphi}}_{k\bar{l}}\left(\Delta_{\omega_{\varphi},1/2}\Psi_{ip}^k\right)\overline{\Psi_{jq}^l}\right)+|\nabla^{g_{\varphi}} \Psi|^2_{g_{\varphi}}+|\overline{\nabla}^{g_{\varphi}}\Psi|_{g_{\varphi}}^2\\
&&+\Ric(g_{\varphi})^{i\bar{\jmath}}g_{\varphi}^{p\bar{q}}{g_{\varphi}}_{k\bar{l}}\Psi_{ip}^k\overline{\Psi_{jq}^l}+g_{\varphi}^{i\bar{\jmath}}\Ric(g_{\varphi})^{p\bar{q}}{g_{\varphi}}_{k\bar{l}}\Psi_{ip}^k\overline{\Psi_{jq}^l}-g_{\varphi}^{i\bar{\jmath}}g_{\varphi}^{p\bar{q}}\Ric(g_{\varphi})_{k\bar{l}}\Psi_{ip}^k\overline{\Psi_{jq}^l},
\end{eqnarray*}
where
\begin{eqnarray*}
&&\Delta_{\omega_{\varphi},1/2}=g_{\varphi}^{i\bar{\jmath}}\nabla^{g_{\varphi}}_i\nabla^{g_{\varphi}}_{\bar{\jmath}},\\
&&T^{i\bar{\jmath}}:=g_{\varphi}^{i\bar{k}}g_{\varphi}^{l\bar{\jmath}}T_{k\bar{l}},
\end{eqnarray*}
for $T_{k\bar{l}}\in\Lambda^{1,\,0}M\otimes\Lambda^{0,\,1}M$. We also have that
\begin{eqnarray*}
{\partial_{\tau}\Psi(\tau)_{ip}^k|_{\tau=0}}&=&\partial_{\tau}|_{\tau=0}(\Gamma(g_{\varphi}(\tau))-\Gamma(g(\tau)))_{ip}^k\\
&=&\nabla^{g_{\varphi}}_i(-\Ric(g_{\varphi})_p^k+\eta_p^k)-\nabla^{g}_i(-\Ric(g)_p^k+\partial\bar{\partial}F(g)_p^k),\\
\partial_{\tau}g_{\varphi}^{i\bar{\jmath}}&=&\Ric(g_{\varphi})^{i\bar{\jmath}}-\eta^{i\bar{\jmath}}.
\end{eqnarray*}
Finally, by using the second Bianchi identity, we see that
\begin{eqnarray*}
\Delta_{g_{\varphi},1/2}\Psi_{ip}^k=g_{\varphi}^{a\bar{b}}\nabla_a^{g_{\varphi}}\Rm(g)_{i\bar{b}p}^k-\nabla^{g_{\varphi}}_i\Ric(g_{\varphi})_p^k,
\end{eqnarray*}
which implies that the following evolution equation is satisfied by $\Psi$:
\begin{eqnarray*}
{\partial_{\tau}\Psi_{ip}^k(\tau)|_{\tau=0}}&=&\Delta_{g_{\varphi},1/2}\Psi_{ip}^k+T_{ip}^k,
\end{eqnarray*}
where $T$ is a tensor such that
\begin{eqnarray*}
T&=&g_{\varphi}^{-1}\ast\nabla^{g_{\varphi}}\Rm(g)+\nabla^{g_{\varphi}}\eta+\nabla^g(\Ric(g)-\partial\bar{\partial}F(g))\\
&=&g_{\varphi}^{-1}\ast\nabla^g\Rm(g)+g_{\varphi}^{-1}\ast g_{\varphi}^{-1}\ast\Rm(g)\ast\Psi+g_{\varphi}^{-1}\ast\Psi\ast\eta+\nabla^g(\eta+\Ric(g)-\partial\bar{\partial} F(g)).
\end{eqnarray*}
Since this flow is only evolving by homotheties and diffeomorphisms, we have that
\begin{eqnarray*}
S(\tau)&=&(1+\tau)^{-1}\varphi_{\tau}^*S(g_{\varphi},g),\\
\partial_{\tau}S|_{\tau=0}&=&-S(g_{\varphi},g)-\frac{X}{2}\cdot S(g_{\varphi},g).
\end{eqnarray*}

To summarise, by Young's inequality, the boundedness of $\|g_{\varphi}^{-1}g\|_{C^0(M)}$, $\|g_{\varphi}g^{-1}\|_{C^0(M)}$, and the boundedness of the covariant derivatives of the tensors $\Rm(g)$, $\eta$, and $F(g)$, we have that
\begin{eqnarray*}
\Delta_{g_{\varphi}}S+\frac{X}{2}\cdot S\geq -C(S+1)
\end{eqnarray*}
for some positive uniform constant $C$.

We use as a barrier function the trace $\tr_{\omega}(\omega_{\varphi})$ which, by (\ref{crucial-est-tr-C2}) and the uniform equivalence of the metrics $g$ and $g_{\varphi}$, satisfies
\begin{eqnarray*}
\Delta_{\omega_{\varphi}}\tr_{\omega}(\omega_{\varphi})+\frac{X}{2}\cdot \tr_{\omega}(\omega_{\varphi})&\geq& C^{-1}S-C,
\end{eqnarray*}
where $C$ is a uniform positive constant that may vary from line to line. By applying the maximum principle to $\epsilon S+\tr_{\omega}(\omega_{\varphi})$ for some sufficiently small $\epsilon>0$, one arrives at the desired a priori estimate.
\end{proof}

\begin{corollary}\label{coro-equiv-metrics}
 The tensors $g^{-1}g_{\varphi}$ and $g_{\varphi}^{-1}g$ satisfy the following uniform estimate:
 \begin{eqnarray*}
\|g^{-1}g_{\varphi}\|_{C^{0,\alpha}(M)}+\|g_{\varphi}^{-1}g\|_{C^{0,\alpha}(M)}\leq \Lambda\left(n,\alpha,\|F\|_{C^3(M)},\|F\|_{C^0_{con,f}(M)}\right)
\end{eqnarray*}
for any $\alpha\in(0,1)$.
 \end{corollary}

 \begin{proof}
By Proposition \ref{prop-C^2-est}, together with Proposition \ref{prop-C^3-est}, we know that $$\|g^{-1}g_{\varphi}\|_{C^{0,\alpha}}\leq \Lambda\left(n,\alpha,\|F\|_{C^3(M)},\|F\|_{C^0_{con,f}(M)}\right).$$ We also know from Propositions \ref{prop-C^0-est} and \ref{prop-c^0-rad-der} that $g^{-1}g_{\varphi}$ satisfies
 \begin{eqnarray*}
\det(g^{-1}g_{\varphi})=e^{F+\varphi-\frac{X}{2}\cdot\varphi}\geq e^{-C}
\end{eqnarray*}
for some uniform positive constant $C$. Moreover, each eigenvalue of $g_{\varphi}$ is uniformly bounded from above. Thus, $$\|g_{\varphi}^{-1}g\|_{C^0(M)}\leq\Lambda\left(n,\alpha,\|F\|_{C^3(M)},\|F\|_{C^0_{con,f}(M)}\right).$$ Finally, if $u$ is a positive function on $M$ in $C^{\alpha}(M)$ uniformly bounded from below by a positive constant, then $[u^{-1}]_{\alpha}\leq \|u\|_{C^{\alpha}(M)}(\inf_Mu)^{-2}$. This last remark implies that $$\|g_{\varphi}^{-1}g\|_{C^{0,\alpha}(M)}\leq\Lambda\left(n,\alpha,\|F\|_{C^3(M)},\|F\|_{C^0_{con,f}(M)}\right)$$ as well.
\end{proof}

\newpage

\subsection{Weighted estimates on higher derivatives of the potential $\varphi$}

\begin{prop}\label{prop-weighted-rad-der-est}
Let $\varphi$ be a solution to (\ref{MA}). Then
\begin{eqnarray*}
\|X\cdot \varphi\|_{C^0_{con,f}(M)}\leq C\left(n,\omega,\|F\|_{C^3_{con,f}}\right).
\end{eqnarray*}
\end{prop}

\begin{proof}
As in \cite{Siepmann}, we compute the evolution equation of the quantity $\varphi-\frac{X}{2}\cdot \varphi$.

Since $\varphi$ satisfies (\ref{MA}) and is $JX$-invariant, and since $X^{1,0}$ is holomorphic, we see that
\begin{eqnarray*}
\frac{X}{2}\cdot\left(\varphi-\frac{X}{2}\cdot\varphi\right)+\frac{X}{2}\cdot F&=&X^{1,0}\cdot\left(\varphi-X^{1,0}\cdot\varphi\right)+X^{1,0}\cdot F\\
&=&\nabla^g_{X^{1,0}}\log\left(\frac{\omega_{\varphi}^n}{\omega^n}\right)\\
&=&\tr_{g_{\varphi}}(\nabla^g_{X^{1,0}}\nabla^g_{\cdot}\nabla^g_{\bar{\cdot}}\varphi)\\
&=&\Delta_{\omega_{\varphi}}\left(\frac{X}{2}\cdot\varphi\right)-g_{\varphi}^{i\bar{\jmath}}\nabla^g_{i}{X^{1,0}}^k\nabla^g_{\bar{\jmath}}\nabla^g_k\varphi\\
&=&\Delta_{\omega_{\varphi}}\left(\frac{X}{2}\cdot\varphi\right)-g_{\varphi}^{i\bar{\jmath}}g^{k\bar{l}}\nabla^g_i\nabla^g_{\bar{l}}f\nabla^g_{\bar{\jmath}}\nabla^g_k\varphi+\textit{O}(f^{-1})\ast\partial\bar{\partial}\varphi\\
&=&\Delta_{\omega_{\varphi}}\left(\frac{X}{2}\cdot\varphi-\varphi\right)+\textit{O}(f^{-1})\ast\partial\bar{\partial}\varphi,\\
\end{eqnarray*}
where we have made use of Lemma \ref{prop-basic-est-pot-fct} in the final two equalities. The $C^2$ a priori estimate provided by Proposition \ref{prop-C^2-est}, together with the asymptotics of $F$, then imply that
\begin{eqnarray*}
\left\arrowvert\Delta_{\omega_{\varphi}}\left(\frac{X}{2}\cdot\varphi-\varphi\right)+\frac{X}{2}\cdot\left(\frac{X}{2}\cdot\varphi-\varphi\right)\right\arrowvert\leq Cf^{-1}
\end{eqnarray*}
for some uniform positive constant $C$.

Now, it turns out that $f^{-1}$ is a good barrier tensor outside a compact set. Indeed, thanks to Lemma \ref{prop-basic-est-pot-fct} and Corollary \ref{coro-equiv-metrics},
\begin{eqnarray*}
\left(\Delta_{\omega_{\varphi}}+\frac{X}{2}\cdot\right)(f^{-1})\leq -Cf^{-1}
\end{eqnarray*}
for some uniform positive constant $C$. Hence, for any positive constant $A\geq C$, the following holds outside a compact set $\{f\leq t_0\}$ independent of $A$:
\begin{eqnarray*}
\Delta_{\omega_{\varphi}}\left(\frac{X}{2}\cdot\varphi-\varphi-Af^{-1}\right)+\frac{X}{2}\cdot\left(\frac{X}{2}\cdot\varphi-\varphi-Af^{-1}\right)\leq 0.
\end{eqnarray*}
Now, we can find a constant $A$ depending on $\|\frac{X}{2}\cdot\varphi-\varphi\|_{C^0(M)}$, hence, by Propositions \ref{prop-C^0-est} and \ref{prop-c^0-rad-der}, on the data $F$, such that $\max_{\{f=t_0\}}(\frac{X}{2}\cdot\varphi-\varphi-Af^{-1})\leq 0$. Since $\lim_{x\to+\infty}\left(\frac{X}{2}\cdot\varphi-\varphi\right)(x)=0$ by assumption, the maximum principle applied to $\frac{X}{2}\cdot\varphi-\varphi-Af^{-1}$ yields the desired estimate
\begin{eqnarray*}
\frac{X}{2}\cdot\varphi-\varphi\leq Af^{-1}.
\end{eqnarray*}
The same argument applies to obtain a uniform weighted lower bound.
\end{proof}

\subsection{$C^2$-weighted estimates}
\begin{prop}\label{prop-C^2-wei-est}
Let $F\in \mathcal{C}^{3,\alpha}_{f,X}(M)$ for some $\alpha\in(0,1)$ and let $\varphi$ be a solution to (\ref{MA}) in $\mathcal{D}^{2,\alpha}_{f,X}(M)$. Then
\begin{eqnarray*}
\|\varphi\|_{\mathcal{D}^{2,\alpha}_{f,X}(M)}\leq C\left(n,\alpha,w,\|F\|_{\mathcal{C}^{3,\alpha}_{f,X}(M)}\right).
\end{eqnarray*}
\end{prop}

\begin{proof}
Since $\varphi$ is a solution to (\ref{MA}), we have by (\ref{equ:taylor-exp}),
\begin{eqnarray}
F=MA(\varphi)=\Delta_{\omega}\varphi+\frac{X}{2}\cdot\varphi-\varphi-\int_0^1\int_0^{\tau}\arrowvert \partial\bar{\partial}\varphi\arrowvert^2_{g_{\sigma\varphi}}d\sigma d\tau.\label{MA-lin-form}
\end{eqnarray}
In order to apply Theorem \ref{iso-sch-Laplacian} to obtain the desired a priori weighted $C^2$-bound, it suffices to prove the following claim which gives a rough estimate on the second derivatives of $\varphi$.
\begin{claim}\label{claim-int-est-a priori}
For any $\alpha\in(0,1)$,
 \begin{equation*}
\left\|\int_0^1\int_0^{\tau}\arrowvert \partial\bar{\partial}\varphi\arrowvert^2_{g_{\sigma\varphi}}d\sigma d\tau\right\|_{C_{con,f}^{0,\alpha}(M)}\leq C\left(n,\alpha,w,\|F\|_{\mathcal{C}^{3,\alpha}_{f,X}(M)}\right).
\end{equation*}
 \end{claim}
 \begin{proof}[Proof of Claim \ref{claim-int-est-a priori}]
We proceed by using the local (Morrey)-Schauder estimates as in \cite[Section 5.4.1.8]{Siepmann}.

Let $x\in M$ and choose normal holomorphic coordinates in a ball $B_{g}(x,\delta)$ for some $\delta>0$ uniform in $x\in M$ (cf.~Lemma \ref{hol-coor}). Then we have that
 \begin{eqnarray*}
F&=&\log\left(\frac{\omega_{\varphi}^n}{\omega^n}\right)+\frac{X}{2}\cdot\varphi-\varphi\\
&=&\left(\int_0^1 g_{\tau\varphi}^{i\bar{\jmath}}d\tau\right)\partial_i\partial_{\bar{\jmath}}\varphi+\frac{X}{2}\cdot\varphi-\varphi\\
&=&a^{i\bar{\jmath}}\partial_i\partial_{\bar{\jmath}}\varphi+\frac{X}{2}\cdot\varphi-\varphi.
\end{eqnarray*}
Now, by Corollary \ref{coro-equiv-metrics}, $\|a^{i\bar{\jmath}}\|_{C^{0,\alpha}(B_{g}(x,\delta))}$ is uniformly bounded from above and $a^{i\bar{\jmath}}\geq \Lambda^{-1}\delta^{i\bar{\jmath}}$ on $B_{g}(x,\delta)$ for some uniform constant $\Lambda>0$. Therefore, by considering $\frac{X}{2}\cdot \varphi-\varphi$ as a source term, the local Morrey-Schauder estimates \cite[Chapter 3]{Lun-Sch-Est} yield
\begin{eqnarray*}
\|\varphi\|_{C^{1,\alpha}(B_g(x,\delta/2))}\leq C\left(\left\|\frac{X}{2}\cdot\varphi-\varphi\right\|_{C^{0}(B_g(x,\delta))}+\|F\|_{C^0(B_g(x,\delta))}+\|\varphi\|_{C^0(B_g(x,\delta))}\right)
\end{eqnarray*}
for some uniform positive constant $C=C\left(n,\alpha,\omega,\|F\|_{\mathcal{C}^{3,\alpha}_{f,X}(M)}\right)$. Moreover, Propositions \ref{prop-a-prio-wei-C0-est} and \ref{prop-weighted-rad-der-est} imply that
\begin{eqnarray*}
\sup_{x\in M}f(x)\|\varphi\|_{C^{1,\alpha}(B_g(x,\delta/2))}\leq C\left(n,\alpha,\omega,\|F\|_{\mathcal{C}^{3,\alpha}_{f,X}(M)}\right),
\end{eqnarray*}
which, in turn, gives rise to the following rough a priori decay:
\begin{eqnarray*}
\sup_{x\in M}f^{\frac{1}{2}}(x)\|X\cdot\varphi\|_{C^{0,\alpha}(B_g(x,\delta))}\leq C\left(n,\alpha,\omega,\|F\|_{\mathcal{C}^{3,\alpha}_{f,X}(M)}\right).
\end{eqnarray*}
Thus, the Schauder estimates imply that
\begin{eqnarray*}
\|\varphi\|_{C^{2,\alpha}(B_g(x,\delta/2))}&\leq& C\left(\|F\|_{C^{0,\alpha}(B_g(x,\delta))}+\left\|\frac{X}{2}\cdot\varphi-\varphi\right\|_{C^{0,\alpha}(B_g(x,\delta))}+\|\varphi\|_{C^{0,\alpha}(B_g(x,\delta))}\right)\\
&\leq& Cf(x)^{-\frac{1}{2}}
\end{eqnarray*}
for some uniform positive constant $C=C\left(n,\alpha,\omega,\|F\|_{\mathcal{C}^{3,\alpha}_{f,X}(M)}\right)$. The desired rough a priori estimate on $\partial\bar{\partial}\varphi$ now follows.
\end{proof}
\end{proof}

\subsection{Weighted $C^3$-estimates}
\begin{theorem}\label{theo-C^3-wei-est}
Let $F\in \mathcal{C}^{3,\alpha}_{f,X}(M)$ for some $\alpha\in(0,1)$ and let $\varphi$ be a solution to (\ref{MA}) in $\mathcal{D}^{2+1,\alpha}_{f,X}(M)$. Then
 \begin{eqnarray*}
\|\varphi\|_{\mathcal{D}^{2+1,\alpha}_{f,X}(M)}\leq C\left(n,\alpha,w,\|F\|_{\mathcal{C}^{3,\alpha}_{f,X}(M)}\right).
\end{eqnarray*}
\end{theorem}

\begin{proof}
The proof is almost identical to the proof of Proposition \ref{prop-C^2-wei-est}. Recall that if $x\in M$ and if we choose normal holomorphic coordinates on $B_{g}(x,\delta)$ for some constant $\delta>0$ uniform in $x\in M$ (cf.~Lemma \ref{hol-coor}), then
 \begin{eqnarray}
F&=&a^{i\bar{\jmath}}\partial_i\partial_{\bar{\jmath}}\varphi+\frac{X}{2}\cdot\varphi-\varphi.\label{MA-equ-hol-coord}
\end{eqnarray}
In order to apply the Schauder estimates, we need an a priori $C^{1,\alpha}$-bound on the coefficients $(a^{i\bar{\jmath}})_{i\bar{\jmath}}$. Arguing as in the proof of Corollary \ref{coro-equiv-metrics}, it suffices to prove the following claim.
\begin{claim}\label{claim-C^3-alpha-est}
There exists a uniform bound on the $C^{3,\alpha}(M)$ norm of $\varphi$, i.e.,
\begin{eqnarray*}
\|\varphi\|_{C^{3,\alpha}}\leq C\left(n,\alpha,\omega,\|F\|_{\mathcal{C}_{f,X}^{3,\alpha}}\right).
\end{eqnarray*}
\begin{proof}[Proof of Claim \ref{claim-C^3-alpha-est}]
From the proof of the $C^2$-estimate (cf.~Proposition \ref{prop-C^2-est} and equation (\ref{eq-sec-der-yau})), one has
\begin{equation}\label{weloveele}
\begin{split}
\Delta_{\omega_{\varphi}}\left(\Delta_{\omega}\varphi+\frac{X}{2}\cdot\varphi-\varphi\right) =& \Delta_{\omega_{\varphi}}F+g_{\varphi}^{-1}\ast g^{-1}\ast\Rm(g)\\
&+g^{-1}\ast g^{-1}\ast\Rm(g)+g^{-1}\ast g_{\varphi}^{-1}\ast g_{\varphi}^{-1}\ast \bar{\nabla}\nabla\bar{\nabla}\varphi\ast\nabla\bar{\nabla}\nabla \varphi,
\end{split}
\end{equation}
where $\ast$ denotes the ordinary contraction of two tensors. By Propositions \ref{prop-C^2-est} and \ref{prop-C^3-est}, the $C^0(M)$-norm of the right-hand side of \eqref{weloveele} is uniformly bounded and, thanks to Corollary \ref{coro-equiv-metrics}, so also are the coefficients of $\Delta_{\omega_{\varphi}}$ in the $C^{0,\alpha}$ sense. Consequently, by applying the Morrey-Schauder $C^{1,\alpha}$-estimates, we see that
\begin{eqnarray*}
\left\|\Delta_{\omega}\varphi+\frac{X}{2}\cdot\varphi-\varphi\right\|_{C^{1,\alpha}(M)}\leq C\left(n,\alpha,\omega,\|F\|_{C^{3,\alpha}_{f,X}(M)}\right).
\end{eqnarray*}
Applying the Schauder estimates once again with respect to $\Delta_{\omega}$, we find, using Proposition \ref{prop-C^2-wei-est}, that
\begin{eqnarray*}
\|\varphi\|_{C^{3,\alpha}(M)}&\leq& C(n,\alpha,\omega)\left(\|\Delta_{\omega}\varphi\|_{C^{1,\alpha}(M)}+\|\varphi\|_{C^{1,\alpha}(M)}\right)\\
&\leq& C\left(n,\alpha,\omega,\|F\|_{C^{3,\alpha}_{f,X}(M)}\right).
\end{eqnarray*}
\end{proof}
\end{claim}
With Claim \ref{claim-C^3-alpha-est} in hand, one can now apply the Schauder estimates to (\ref{MA-equ-hol-coord}) in the following way:
\begin{eqnarray*}
\|\varphi\|_{C^{3,\alpha}(B_g(x,\delta/2))}\leq C\left(\left\|\frac{X}{2}\cdot\varphi-\varphi\right\|_{C^{1,\alpha}(B_g(x,\delta))}+\|\varphi\|_{C^{1,\alpha}(B_g(x,\delta))}+\|F\|_{C^{1,\alpha}(B_g(x,\delta))}\right),
\end{eqnarray*}
where $C:=C\left(n,\alpha,w,\|F\|_{\mathcal{C}^{3,\alpha}_{f,X}(M)}\right).$ Thanks to Proposition \ref{prop-C^2-wei-est}, this estimate implies a rough decay on the third derivatives of $\varphi$; more precisely, it implies that
\begin{eqnarray*}
\sup_{x\in M}f(x)\|\varphi\|_{C^{3,\alpha}(B_g(x,\delta/2))}\leq C\left(n,\alpha,w,\|F\|_{\mathcal{C}^{3,\alpha}_{f,X}(M)}\right).
\end{eqnarray*}
To complete the proof of Theorem \ref{theo-C^3-wei-est}, it suffices to invoke Theorem \ref{iso-sch-Laplacian}. Indeed, since $\varphi$ satisfies (\ref{MA-lin-form}), and since $$F+\int_0^1\int_0^{\tau}\arrowvert \partial\bar{\partial}\varphi\arrowvert^2_{g_{\sigma\varphi}}d\sigma d\tau$$ is uniformly bounded in the $C^{1,\alpha}_{con,f}(M)$ sense by the previous rough estimate combined with Proposition \ref{prop-C^2-wei-est}, one can apply Theorem \ref{iso-sch-Laplacian} to assert that
\begin{eqnarray*}
\|\varphi\|_{\mathcal{D}^{2+1,\alpha}_{f,X}(M)}\leq C(n,\alpha,\omega)\left\|\Delta_{\omega}\varphi+\frac{X}{2}\cdot\varphi-\varphi\right\|_{\mathcal{C}^{1,\alpha}_{f,X}(M)}\leq C\left(n,\alpha,\omega,\|F\|_{\mathcal{C}^{3,\alpha}_{f,X}(M)}\right).
\end{eqnarray*}
\end{proof}

\newpage
\section{Bootstrapping}\label{section-bootstrapping}
In this section, we prove that if $\varphi$ is a solution to (\ref{MA}) with some finite regularity at infinity, say $\varphi\in \mathcal{D}^{k,\alpha}_{f,X}(M)$ for some $k\in\mathbb{N}$ and $\alpha\in(0,1)$, and if $F\in \mathcal{C}^{\infty}_{f,X}(M)$, then $\varphi\in \mathcal{D}^{k+1,\alpha}_{f,X}(M)$ with corresponding estimates on these norms. Again, we follow \cite{Siepmann}, but, since the rate of convergence to the asymptotic cone is only polynomial in our case, we must pay careful attention to the rough a priori estimates which are crucial intermediate steps in obtaining the full a priori estimate.
\begin{theorem}\label{theo-bootstrap}
Let $\varphi\in \mathcal{D}^{k+2,\alpha}_{f,X}(M)$ for some integer $k\geq 1$ and $\alpha\in(0,1)$ be a K\"ahler potential that is a solution to the complex Monge-Amp\`ere equation (\ref{MA}) with $F\in \mathcal{C}^{\infty}_{f,X}(M)$. Then $\varphi\in \mathcal{D}^{\infty}_{f,X}(M)$. Moreover, one has the following estimate:
\begin{eqnarray*}
\|\varphi\|_{\mathcal{D}^{(k+1)+2,\alpha}_{f,X}(M)}\leq C\left(n,k,\alpha,\omega,\|F\|_{\mathcal{C}^{\max\{k+1,3\},\alpha}_{f,X}(M)}\right).
\end{eqnarray*}
\end{theorem}

\begin{remark}
The a priori estimate we obtain in Theorem \ref{theo-bootstrap} is called a ``(rough) tame'' estimate in the terminology of \cite{Ham-Nas-Mos}. In \cite{Der-Smo-Pos-Cur-Con}, because the nonlinearities of the expanding gradient Ricci soliton equation are more tedious in the generic Riemannian case, we used a more refined method to derive such tame estimates, one that allowed us to make use of the Nash-Moser implicit function theorem as presented in \cite{Ham-Nas-Mos}.
\end{remark}

\begin{proof}
We divide the proof of Theorem \ref{theo-bootstrap} into three steps that follow closely the steps of the proof of Theorem \ref{theo-C^3-wei-est}. In the course of the proof, we denote by $C$ a positive constant that only depends on the data $n,k,\alpha,\omega,\|F\|_{\mathcal{C}^{\max\{k+1,3\},\alpha}_{f,X}(M)}$. Moreover, in order to keep clarity in our notation, we will denote the Levi-Civita connection of $g$ by $\nabla$. We start with a (non-weighted) a priori $C^{k+3,\alpha}(M)$-bound.\\
\begin{claim}\label{claim-C-k+3}
Under the assumptions of Theorem \ref{theo-bootstrap}, $\|\varphi\|_{C^{(k+1)+2,\alpha}(M)}\leq C$.
\end{claim}

\begin{proof}[Proof of Claim \ref{claim-C-k+3}]
Recall from the proof of Proposition \ref{prop-C^2-est} that $\Delta_{\omega}\varphi+\frac{X}{2}\cdot\varphi-\varphi$ satisfies
\begin{equation}\label{salut}
\begin{split}
\Delta_{\omega_{\varphi}}\left(\Delta_{\omega}\varphi+\frac{X}{2}\cdot\varphi-\varphi\right)=&\Delta_{\omega_{\varphi}}F+g_{\varphi}^{-1}\ast g^{-1}\ast\Rm(g)\\
&+g^{-1}\ast g^{-1}\ast\Rm(g)+g^{-1}\ast g_{\varphi}^{-1}\ast g_{\varphi}^{-1}\ast \bar{\nabla}\nabla\bar{\nabla}\varphi\ast\nabla\bar{\nabla}\nabla \varphi.
\end{split}
\end{equation}
Now, the coefficients of $\Delta_{\omega_{\varphi}}$ are uniformly bounded in $C^{k,\alpha}(M)$, hence in $C^{k-1,\alpha}(M)$. Moreover, the right-hand side of \eqref{salut} is also bounded in $C^{k-1,\alpha}(M)$. It then follows from the (local) Schauder estimates that $\Delta_{\omega}\varphi+\frac{X}{2}\cdot\varphi-\varphi$ is uniformly bounded in $C^{(k-1)+2,\alpha}(M)$. Since $\frac{X}{2}\cdot\varphi-\varphi$ is uniformly bounded in $C^{(k-1)+2,\alpha}(M)$ by assumption on $\varphi$, an application of the Schauder estimates once more to $\Delta_{\omega}\varphi$ yields the desired result.
\end{proof}

\begin{claim}\label{claim-rough-est}
The following rough estimate holds:
\begin{eqnarray*}
\sup_{x\in M}f^{\frac{k+1}{2}}(x)\|\varphi\|_{C^{(k+1)+2,\alpha}(B(x,\delta/2))}\leq C,
\end{eqnarray*}
where $\delta>0$ is uniform in $x\in M$ (cf.~Lemma \ref{hol-coor}).
\end{claim}

\begin{proof}[Proof of Claim \ref{claim-rough-est}]
Notice that this a priori estimate only applies to the $\alpha$-H\"older norm of the $((k+1)+2)$-th covariant derivative of $\varphi$. Let $x\in M$ and let $B_g(x,\delta)$ be a ball endowed with holomorphic normal coordinates as in Lemma \ref{hol-coor}. We will prove this claim by induction with the following induction hypotheses:
\begin{eqnarray*}
\sup_{x\in M}f^{\frac{l+1}{2}}(x)\|\varphi\|_{C^{(k+1)+2,\alpha}(B(x,\delta/2))}\leq C\quad\textrm{for $1\leq l\leq k$}.
\end{eqnarray*}
Recall that in normal holomorphic coordinates, $\varphi$ satisfies (\ref{MA-equ-hol-coord}), that is, the equation
\begin{equation}\label{6.8again}
a^{i\bar{\jmath}}\partial_i\partial_{\bar{\jmath}}\varphi=F-\frac{X}{2}\cdot\varphi+\varphi.
\end{equation}
By Claim \ref{claim-rough-est}, the coefficients $(a^{i\bar{\jmath}})_{i\bar{\jmath}}$ are uniformly bounded in $C^{k+1,\alpha}(M)$. The same also holds true for the right-hand side of \eqref{6.8again}. The Schauder estimates therefore tell us that
\begin{eqnarray*}
f(x)\|\varphi\|_{C^{(k+1)+2,\alpha}(B_g(x,\delta/2))}&\leq& f(x)C\left(\left\|F-\frac{X}{2}\cdot\varphi+\varphi\right\|_{C^{(k+1),\alpha}(B_g(x,\delta))}+\|\varphi\|_{C^{(k+1),\alpha}(B_g(x,\delta))}\right)\\
&\leq& C,
\end{eqnarray*}
which establishes the case $l=1$.

Next, let us derive the equation satisfied by $\nabla^{k+1}\varphi$. We compute:
\begin{eqnarray*}
a^{i\bar{\jmath}}\nabla_i\nabla_{\bar{\jmath}}\nabla^{k+1}\varphi&=&a^{i\bar{\jmath}}[\nabla_i\nabla_{\bar{\jmath}},\nabla^{k+1}]\varphi+\nabla^{k+1}(a^{i\bar{\jmath}}\partial_i\partial_{\bar{\jmath}}\varphi)+\sum_{p=0}^k\nabla^{k+1-p}a^{i\bar{\jmath}}\ast_{g}\nabla^p\partial_i\partial_{\bar{\jmath}}\varphi\\
&=&a^{i\bar{\jmath}}[\nabla_i\nabla_{\bar{\jmath}},\nabla^{k+1}]\varphi+\nabla^{k+1}\left(F-\frac{X}{2}\cdot\varphi+\varphi\right)+\sum_{p=0}^k\nabla^{k+1-p}a^{i\bar{\jmath}}\ast_{g}\nabla^p\partial_i\partial_{\bar{\jmath}}\varphi\\
&=&\sum_{p=1}^{k+1}a^{-1}\ast\nabla^{k+1-p}\Rm(g)\ast_g\nabla^p\varphi+\sum_{p=0}^k\nabla^{k+1-p}a^{i\bar{\jmath}}\ast_{g}\nabla^p\partial_i\partial_{\bar{\jmath}}\varphi\\
&&+\nabla^{k+1}\left(F-\frac{X}{2}\cdot\varphi+\varphi\right),
\end{eqnarray*}
where $\ast_{g}$ denotes contraction with respect to $g$. Now, by using the quadratic decay of the curvature at infinity, together with the assumption $\varphi\in\mathcal{D}^{k+2,\alpha}_{f,X}(M)$, one has that
\begin{eqnarray*}
&&\left\|f^{\frac{k+1}{2}}\nabla^{k+1}\left(F-\frac{X}{2}\cdot\varphi+\varphi\right)\right\|_{C^{0,\alpha}(M)}\leq C,\\
&&\|f^{\frac{k+1-p}{2}+1}\nabla^{k+1-p}\Rm(g)\|_{C^{0,\alpha}(M)}\leq C,\quad\|f^{\frac{p}{2}+1}\nabla^p\varphi\|_{C^{0,\alpha}(M)}\leq C,\quad p=0,...,k,\\
&&\| f^{\frac{k}{2}+1}\nabla^{k+1}\varphi\|_{C^{0,\alpha}(M)}\leq C,\\
&&\left\|f^{\frac{k}{2}+2}\sum_{p=1}^{k+1}a^{-1}\ast\nabla^{k+1-p}\Rm(g)\ast_g\nabla^p\varphi\right\|_{C^{0,\alpha}(M)}\leq C,
\end{eqnarray*}
where we put the best a priori possible power of $f$ in front of each term. It remains to estimate the sum $\sum_{p=0}^k\nabla^{k+1-p}a^{i\bar{\jmath}}\ast_{g}\nabla^p\partial_i\partial_{\bar{\jmath}}\varphi$. We know that
\begin{eqnarray*}
&&\|f^{\frac{p+2}{2}+1}\nabla^p\partial_i\partial_{\bar{\jmath}}\varphi\|_{C^{0,\alpha}(M)}\leq C,\quad p=0,...,k-2,\\
&&\|f^{\frac{k}{2}+1}\nabla^p\partial_i\partial_{\bar{\jmath}}\varphi\|_{C^{0,\alpha}(M)}\leq C,\quad p\in\{k-1,k\}.
\end{eqnarray*}
In order to estimate the covariant derivatives of $(a^{i\bar{\jmath}})_{i\bar{\jmath}}$, it suffices to understand the decay of the covariant derivatives of $g_{\varphi}^{-1}$. We have the following formulas:
\begin{eqnarray*}
&&\nabla\left(g_{\varphi}^{-1}\right)=g_{\varphi}^{-1}\ast g_{\varphi}^{-1}\ast\nabla g_{\varphi},\\
&&\nabla^m\left(g_{\varphi}^{-1}\right)=g_{\varphi}^{-1}\ast\sum_{p=0}^{m-1}\nabla^{m-p}g_{\varphi}\ast\nabla^p\left(g_{\varphi}^{-1}\right),\quad m\geq 1,
\end{eqnarray*}
which imply that
\begin{eqnarray*}
&&\|f^{\frac{m+2}{2}+1}\nabla^m\left(g_{\varphi}^{-1}\right)\|_{C^{0,\alpha}(M)}\leq C,\quad m\leq k-2,\\
&&\|f^{\frac{k}{2}+1}\nabla^m\left(g_{\varphi}^{-1}\right)\|_{C^{0,\alpha}(M)}\leq C,\quad m\in\{k-1,k\},\\
&&\|f^{\frac{l+1}{2}}\nabla^{k+1}\left(g_{\varphi}^{-1}\right)\|_{C^{0,\alpha}(M)}\leq C.
\end{eqnarray*}
Therefore
\begin{eqnarray*}
\left\|f^{\frac{l+1}{2}+2}\left(\sum_{p=0}^k\nabla^{k+1-p}a^{i\bar{\jmath}}\ast_{g}\nabla^p\partial_i\partial_{\bar{\jmath}}\varphi\right)\right\|_{C^{0,\alpha}(M)}\leq C,
\end{eqnarray*}
which, using the Schauder estimates, implies in turn that
\begin{eqnarray*}
\sup_{x\in M}\|f^{\max\{\frac{l+1}{2}+2,\frac{k+1}{2}\}}\nabla^{k+1}\varphi\|_{C^{0,\alpha}(B_g(x,\delta/2))}\leq C.
\end{eqnarray*}
This completes the proof of the claim.
\end{proof}

\begin{claim}\label{claim-gap-decay-inf}
We have that $\varphi\in\mathcal{D}^{(k+1)+2,\alpha}_{f,X}(M)$ and $\|\varphi\|_{\mathcal{D}^{(k+1)+2,\alpha}_{f,X}(M)}\leq C$.
\end{claim}

\begin{proof}[Proof of Claim \ref{claim-gap-decay-inf}]
By Theorem \ref{iso-sch-Laplacian}, we need only to prove that
\begin{eqnarray*}
&&\varphi\in\mathcal{D}^{(k+1)+2,\alpha}_{f,X}(M),\\
&&\left\|\Delta_{\omega}\varphi+\frac{X}{2}\cdot\varphi-\varphi\right\|_{\mathcal{C}^{(k+1),\alpha}_{f,X}(M)}\leq C.
\end{eqnarray*}
Now, since $\varphi$ satisfies (\ref{MA-lin-form}), it suffices to prove that
\begin{eqnarray*}
&&\varphi\in\mathcal{D}^{(k+1)+2,\alpha}_{f,X}(M),\\
&&\left\|F+\int_0^1\int_0^{\tau}\arrowvert \partial\bar{\partial}\varphi\arrowvert^2_{g_{\sigma\varphi}}d\sigma d\tau\right\|_{\mathcal{C}^{(k+1),\alpha}_{f,X}(M)}\leq C.
\end{eqnarray*}
The second estimate here is implied by the assumptions on $F$, $\varphi$, and by Claim \ref{claim-rough-est}; indeed,
\begin{eqnarray*}
&&\nabla^{k+1}(\partial\bar{\partial}\varphi\ast\partial\bar{\partial}\varphi)=\sum_{i=0}^{k+1}\nabla^{k+1-i}\partial\bar{\partial}\varphi\ast\nabla^i\partial\bar{\partial}\varphi,\\
&&\|f^{\frac{k+1}{2}+1}\nabla^{k+1}(\partial\bar{\partial}\varphi\ast\partial\bar{\partial}\varphi)\|_{C^{0,\alpha}(M)}\leq C.
\end{eqnarray*}
Hence it only remains to prove that $\varphi\in\mathcal{D}^{(k+1)+2,\alpha}_{f,X}(M)$; that is, $f^{\frac{k+1}{2}+1}\nabla^{k+1}\varphi\in C^{0,\alpha}(M)$. Notice that, since $\varphi\in\mathcal{D}^{k+2,\alpha}_{f,X}(M)$, we only know that $f^{\frac{k}{2}+1}\nabla^{k+1}\varphi\in C^{0,\alpha}(M)$. On the other hand, Theorem \ref{iso-sch-Laplacian} yields a solution $\tilde{\varphi}\in \mathcal{D}^{(k+1)+2,\alpha}_{f,X}(M)$; indeed,
\begin{eqnarray*}
\Delta_{\omega}\tilde{\varphi}+\frac{X}{2}\cdot\tilde{\varphi}-\tilde{\varphi}=F+\int_0^1\int_0^{\tau}\arrowvert i\partial\bar{\partial}\varphi\arrowvert^2_{\omega_{\sigma\varphi}}d\sigma d\tau.
\end{eqnarray*}
Consequently, $\psi:=\varphi-\tilde{\varphi}$ satisfies $\Delta_{\omega}\psi+\frac{X}{2}\cdot\psi-\psi=0$. Since $\psi$ tends to zero at infinity, we can now use the maximum principle to see that $\psi=0$, i.e., $\varphi=\tilde{\varphi}\in \mathcal{D}^{(k+1)+2,\alpha}_{f,X}(M)$.
\end{proof}
\end{proof}

\section{Proof of Theorem \ref{t:existence}}\label{section-theorem-existence}
In this section, we complete the proof of Theorem \ref{t:existence}. We recall the statement before completing the proof.
\begin{customthm}{D}\label{Existence-Uniqueness-Pde}
For any $F\in \mathcal{C}^{\infty}_{f,X}(M)$, there exists a unique K\"ahler potential $\varphi\in\mathcal{D}^{\infty}_{f,X}(M)$ such that
\begin{eqnarray*}
\omega_{\varphi}^n=e^{\varphi-\frac{X}{2}\cdot\varphi+F}\omega^n.
\end{eqnarray*}
\end{customthm}

\begin{proof} We split the proof up into two parts.
\begin{description}
\item[Existence]
Given $F\in \mathcal{C}^{\infty}_{f,X}(M)$, define $F_t:=tF\in \mathcal{C}^{\infty}_{f,X}(M)$ for $t\in[0,1]$. For $\alpha\in(0,1)$ fixed, denote by
$$S:=\{t\in[0,1]:\textrm{there exists $\varphi_t\in \mathcal{D}^{3,\alpha}_{f,X}(M)$ satisfying (\ref{MA-bis}) with data $F_t\in \mathcal{C}^{\infty}_{f,X}(M)\subset \mathcal{C}^{3,\alpha}_{f,X}(M)$}\}.$$
First note that $S\neq\emptyset$ since $0\in S$ (take $\varphi_{0}=0$).

We next claim that $S$ is open. Indeed, this follows from Theorem \ref{Imp-Def-Kah-Exp}; if $t_0\in S$, then, by Theorem \ref{Imp-Def-Kah-Exp}, there exists $\epsilon_0 >0$ such that there exists a solution $\varphi_{t_0+\epsilon}\in \mathcal{D}^{3,\alpha}_{f,X}(M)$ to $(\ref{MA-bis})$ with data $(t_0+\epsilon)F$ for $\epsilon\in[0,\epsilon_0)$, i.e., $(t_{0}-\epsilon_{0},\,t_{0}+\epsilon_{0})\subset[0,\,1]$.

We also claim that $S$ is closed. Indeed, take a sequence $(t_k)_{k\geq 0}$ in $S$ converging to $t_{\infty}\in S$. Then, for $F_k:=t_kF$, $k\geq 0$, the associated solutions $\varphi_{t_k}=:\varphi_k$, $k\geq 0$, of (\ref{MA-bis}) satisfy
 \begin{eqnarray}
\omega_{\varphi_k}^n=e^{F_k-\frac{X}{2}\cdot\varphi_k+\varphi_k},\quad k\geq 0.\label{MA-seq}
\end{eqnarray}
Now, it is straightforward to see that the sequence $(F_k)_{k\geq 0}$ is uniformly bounded in $\mathcal{C}^{1,\alpha}_{f,X}(M)$. Consequently, by Theorem \ref{theo-C^3-wei-est}, the sequence $(\varphi_k)_{k\geq 0}$ is bounded in $\mathcal{D}^{3,\alpha}_{f,X}(M)$, so that, by the Arzel\`a-Ascoli theorem, a subsequence of $(\varphi_k)_{k\geq 0}$ converges to some $\varphi_{\infty}\in C^{3,\beta}_{loc}(M)$, $\beta\in(0,\alpha)$. Since $(\varphi_k)_{k\geq 0}$ is uniformly bounded in $\mathcal{D}^{3,\alpha}_{f,X}(M)$, $\varphi_{\infty}$ will lie in $\mathcal{D}^{3,\alpha}_{f,X}(M)$ as well. We wish to show that $\varphi_{\infty}$ is a K\"ahler potential, i.e., that $\omega_{\varphi_{\infty}}(x)>0$ for every $x\in M$. To do this, it suffices to show that $\omega_{\varphi_{\infty}}^n(x)>0$ for every $x\in M$. But this last statement follows by letting $k$ tend to $+\infty$ (up to a subsequence) in \eqref{MA-seq}.

Since $S$ is an open and closed non-empty subset of $[0,1]$, connectedness of $[0,\,1]$ implies that in fact $S=[0,1]$. Moreover, Theorem \ref{theo-bootstrap} implies that any solution $\varphi\in \mathcal{D}^{2+1,\alpha}_{f,X}(M)$ of (\ref{MA}) with data $F\in\mathcal{C}^{\infty}_{f,X}(M)$ also lies in $\mathcal{D}^{\infty}_{f,X}(M)$. This concludes the proof of existence.\\

\item[Uniqueness]
Let $(\varphi_j)_{j=1,2}$ be two K\"ahler potentials in $\mathcal{C}^{\infty}_{f,X}(M)$ satisfying
\begin{eqnarray*}
\omega_{\varphi_j}^n=e^{\varphi_j-\frac{X}{2}\cdot\varphi_j+F}\omega^n,\quad j=1,2.
\end{eqnarray*}
Then
\begin{eqnarray}
\log\left(\frac{\omega_{\varphi_2}^n}{\omega_{\varphi_1}^n}\right)=(\varphi_2-\varphi_1)-\frac{X}{2}\cdot(\varphi_2-\varphi_1).\label{MA-diff-sol}
\end{eqnarray}
On one hand, since $\varphi_2-\varphi_1$ tends to zero at infinity, the maximum principle applied to (\ref{MA-diff-sol}) implies that $\sup_M(\varphi_2-\varphi_1)\leq 0$. On the other hand, the minimum principle implies that $\inf_M(\varphi_2-\varphi_1)\geq 0$. Hence $\varphi_2=\varphi_1$.
\end{description}
\end{proof}

\section{Proof of Theorem \ref{Uniqueness-Theorem}}\label{section-proof-uniqueness}
By assumption, we have that $$\rho_{\omega_{i}}+\omega_{i}=\frac{1}{2}\mathcal{L}_{X}\omega_{i}$$
for $i=1,\,2$. Subtracting these two equations yields
\begin{equation*}
\begin{split}
\omega_{1}-\omega_{2}&=-(\rho_{\omega_{1}}-\rho_{\omega_{2}})+\frac{1}{2}\mathcal{L}_{X}(\omega_{1}-\omega_{2})\\
&=i\partial\bar{\partial}\log\left(\frac{\omega_{1}^{n}}{\omega_{2}^{n}}\right)+\frac{1}{2}d((\omega_{1}-\omega_{2})\lrcorner X).\\
\end{split}
\end{equation*}
It is clear from our hypotheses that $$\left|\log\left(\frac{\omega_{1}^{n}}{\omega_{2}^{n}}\right)\right|(x)=O(d_{g_{1}}(x,\,x_{0})^{\lambda+2}).$$
We next claim that $d((\omega_{1}-\omega_{2})\lrcorner X)=i\partial\bar{\partial}f$ for a function $f$ satisfying
\begin{equation}
f(x)=O(d_{g_{1}}(x,\,x_{0})^{\lambda+2})\qquad\textrm{and}\qquad|df|_{g_1}(x)=O(d_{g_{1}}(x,\,x_{0})^{\lambda+1}).\label{est-diff-pot-fct}
\end{equation}

Indeed, let $f:=f_2-f_1$, where $f_i$ is the potential function of the expanding gradient K\"ahler-Ricci soliton $(M,\omega_i,X)$ for $i=1,2$. In particular, $X=\nabla^{g_i}f_i$ for $i=1,2$. By Lemma \ref{id-EGS}, one has the following estimates on $f$:
\begin{eqnarray*}
f&=&f_2-f_1=\arrowvert\nabla^{g_2}f_2\arrowvert^2_{g_2}-|\nabla^{g_1}f_1|^2_{g_1}+\frac{1}{2}(s_{\omega_2}-s_{\omega_1})\\
&=&(g_2-g_1)(X,X)+\frac{1}{2}(s_{\omega_2}-s_{\omega_1}),\\
df(\cdot)&=&g_2(X,\cdot)-g_1(X,\cdot).
\end{eqnarray*}
In case (i), the Ricci forms $\rho_{\omega_{i}}$, $i=1,2$ are non-negative by assumption, hence, by Proposition \ref{pot-fct-est}, the scalar curvatures $s_{\omega_i}$, $i=1,2,$ are bounded on $M$. Since $X$ grows linearly (with respect to the associated K\"ahler metrics $g_1$ and $g_2$ of $\omega_{1}$ and $\omega_{2}$ respectively), and since the difference of the metrics $g_2-g_1$ tends to zero at infinity, again by assumption, one obtains the claimed asymptotic behaviour (\ref{est-diff-pot-fct}) on $f$. A similar argument in case (ii) yields the same result.

Next let $\varphi:=f+\log\left(\frac{\omega_{1}^{n}}{\omega_{2}^{n}}\right)$. Then we have that $\varphi=O(r^{\lambda+2})$ and $\omega_{1}-\omega_{2}=i\partial\bar{\partial}\varphi$. Computing as in \eqref{computation},
one sees that $\varphi$ satisfies
\begin{equation*}
i\partial\bar{\partial}\left(\log\left(\frac{(\omega_2+i\partial\bar{\partial}\varphi)^n}{\omega_2^n}\right)+\frac{X}{2}\cdot\varphi-\varphi\right)=0.
\end{equation*}
We now trace this equation to obtain a function $u:M\to\mathbb{R}$ satisfying
\begin{eqnarray*}
&&u:=\log\left(\frac{(\omega_2+i\partial\bar{\partial}\varphi)^n}{\omega_2^n}\right)+\frac{X}{2}\cdot\varphi-\varphi,\\
&&\Delta_{g_1}u=0.\\
\end{eqnarray*}
Since $|X|_{g_{1}}$ grows at most linearly in both cases (i) and (ii) by Lemma \ref{id-EGS}, we have that
$$|X\cdot f|=O(r^{\lambda+2}).$$
To estimate the term $X\cdot\log\left(\frac{\omega_{1}^{n}}{\omega_{2}^{n}}\right)$, note that
\begin{equation*}
\begin{split}
X\cdot\log\left(\frac{\omega_{1}^{n}}{\omega_{2}^{n}}\right)&
=\tr_{\omega_{1}}\mathcal{L}_{X}\omega_{1}-\tr_{\omega_{2}}\mathcal{L}_{X}\omega_{2}\\
&=2\tr_{\omega_{1}}(\rho_{\omega_{1}}+\omega_{1})-2\tr_{\omega_{2}}(\rho_{\omega_{2}}+\omega_{2})\\
&=2(s_{\omega_{1}}-s_{\omega_{2}}).
\end{split}
\end{equation*}
In case (i), this difference is bounded by Lemma \ref{id-EGS}, so that $|X\cdot\varphi|=O(r^{\lambda+2})$. Consequently, in this case, $|u|=O(r^{\lambda+2})$. In case (ii), the difference $|s_{\omega_{1}}-s_{\omega_{2}}|=o(1)$ by assumption, and so we deduce that $|X\cdot\varphi|=o(1)$ so that $|u|=o(1)$ in this case also.

We now apply Cheng-Yau's result on harmonic functions with sublinear growth on a Riemannian manifold with non-negative Ricci curvature \cite{Che-Yau} in case (i) to deduce that $u$ is constant. In case (ii), this fact follows from the maximum principle.
Thus, in both cases, by subtracting a constant from $\varphi$ if necessary, the situation reduces to $\varphi$ satisfying the following complex Monge-Amp\`ere equation:
\begin{eqnarray}
&&\log\left(\frac{(\omega_2+i\partial\bar{\partial}\varphi)^n}{\omega_2^n}\right)+\frac{X}{2}\cdot\varphi-\varphi=0,\label{MA-Equ-Uni}\\
&&|\varphi|=O(r^{1-\epsilon})\qquad\textrm{for $\epsilon=-\lambda-1>0$}.\nonumber
\end{eqnarray}

Next, observe that
\begin{equation*}
\left|\frac{X}{2}\cdot\varphi-\varphi\right|=\left|\log\left(\frac{(\omega_2+i\partial\bar{\partial}\varphi)^n}{\omega_2^n}\right)\right|=O(r^{-1-\epsilon}).
\end{equation*}
By integrating this differential equation along the Morse flow $(\psi_t)_{t}$ associated to $X/2$, that is, along the flow generated by $2X/|X|_{g_1}^2$, where $|X|_{g_1}$ denotes the Riemannian norm of $X$ with respect to $g_{1}$, one finds that
\begin{eqnarray*}
\partial_t\varphi_{t}&=&\frac{2}{|X|_{g_1}^2}X\cdot \varphi_t\\
&=&\frac{4}{|X|_{g_1}^2}\varphi_t+\textit{O}(t^{-\frac{3+\epsilon}{2}})\\
&=&\frac{2}{|\nabla^{g_1}f|^2}\varphi_t+\textit{O}(t^{-\frac{3+\epsilon}{2}})\\
&=&\frac{1}{t-\frac{s_{\omega_1}}{2}}\varphi_t+\textit{O}(t^{-\frac{3+\epsilon}{2}})\\
&=&\frac{1}{t}\left(1+\textit{O}(t^{-1})\right)\varphi_t+\textit{O}(t^{-\frac{3+\epsilon}{2}}),\\
\varphi_{t}(x)&:=&\varphi(\psi_{t}(x))\quad\textrm{for all $x\in M$ and $t\geq t_0>\min_M f$},
\end{eqnarray*}
where we have used Proposition \ref{pot-fct-est} and the soliton identities from Lemma \ref{id-EGS}. (Note that although we do not assume any a priori decay on the curvature at infinity, Proposition \ref{pot-fct-est} ensures the boundedness of the scalar curvature under the hypothesis of non-negative Ricci curvature in case (i).) Thus,
\begin{eqnarray}
\varphi_t&=&\frac{t}{t_0}e^{\int_{t_0}^ta(s)ds}\left(\varphi_{t_0}+\int_{t_0}^te^{-a(\tau)d\tau}\frac{t_0}{s}\textit{O}(s^{-\frac{3+\epsilon}{2}})ds\right)\label{est-pot-kah-ode}
\end{eqnarray}
for some integrable function $a\in L^1([t_0,+\infty)).$ Since $f$ is quadratic in the distance from a fixed point, a priori we have that $\varphi_t=\textit{O}(t^{\frac{1-\epsilon}{2}})$. By letting $t$ tend to $+\infty$ in (\ref{est-pot-kah-ode}), it therefore follows that
\begin{eqnarray*}
\varphi_{t_0}&=&\textit{O}\left(t_0\int_{t_0}^{+\infty}s^{-\frac{3+\epsilon}{2}-1}ds\right)\\
&=&\textit{O}(t_0^{-\frac{1+\epsilon}{2}})
\end{eqnarray*}
so that $\sup_{\partial B_{g_1}(p,R)}\varphi=\textit{O}(R^{-\epsilon-1}).$ We can now apply the maximum principle to the complex Monge-Amp\`ere equation (\ref{MA-Equ-Uni}) as in the proof of the uniqueness statement of Theorem \ref{Existence-Uniqueness-Pde} to conclude that $\varphi\equiv 0$, i.e., $\omega_1=\omega_2$.

\section{Proof of Theorem A}\label{Proof of Theorem A}

From what we have already done, Theorem D immediately yields the implication ``$\eqref{condition}\implies\eqref{e:main_rate_0}$''. Indeed, for any equivariant resolution $\pi:M\to C_{0}$ satisfying the hypotheses of Theorem \ref{theorem-A}, we apply Proposition \ref{background-metric} to construct a suitable background K\"ahler metric on $M$. Hypothesis (b) of Theorem A then allows us to apply Proposition \ref{equationsetup}, and the result then follows from Theorem \ref{t:existence}. After analysing the background metric of Proposition \ref{background-metric}, one sees that the soliton metric takes the form $c\omega_{0}+\rho_{\omega_0}+O(r^{-4})$, which explains the appearance of the Ricci curvature in \eqref{e:main_rate_0}. The converse implication ``$\eqref{e:main_rate_0}\implies\eqref{condition}$'' follows from the defining equation of an expanding K\"ahler-Ricci soliton.

The uniqueness part of Theorem A is implied by Theorem C using the fact that by \cite{Der-Uni-Con-Ric-Exp}, any two expanding gradient Ricci solitons $g_{1},\,g_{2},$ that are asymptotically conical with the same tangent cone with respect to the same diffeomorphism satisfy $|g_{1}-g_{2}|_{g_{1}}=O(d_{g_{1}}(x,\,x_{0})^{-k})$ for any $k\geq 0$.
\newpage

\section{Proof of Theorem \ref{coro-F}}\label{Proof of coroF}

In this section, let $\pi:M\to C_0$ be an equivariant resolution of a K\"ahler cone $(C_0,\,g_{0})$ of complex dimension $n$ with K\"ahler cone metric $g_{0}$ and associated K\"ahler form $\omega_{0}=\frac{i}{2}\partial\bar{\partial}r^{2}$, where $r$ is the radius function of $g_{0}$, let $X$ denote the lift of the vector field $r\partial_{r}$ from $C_0$ to $M$, and let $E$ denote the exceptional set of the resolution. Our assumption throughout this section is that:

\vspace{0.3cm}
\noindent {\bf Assumption.} $K_{M}$ admits a hermitian metric $h$ whose curvature form $\Theta_{h}$ satisfies
\begin{equation*}
\int_{V}(i\Theta_{h})^{k}\wedge\omega^{\dim_{\mathbb{C}}V-k}>0
\end{equation*}
for all positive-dimensional irreducible analytic subvarieties $V\subset E$ and for all $1\leq k\leq \dim_{\C}V$ for some K\"ahler form $\omega$ on $M$.
\vspace{0.3cm}

\noindent Denote by $J$ and $J_{0}$ the complex structures on $M$ and $C_0$ respectively and by $\widetilde{\Theta_{h}}$ the average of $\Theta_{h}$ over the torus action on $M$ induced by the flow of the vector field $J_{0}r\partial_{r}$ on $C_0$. We first prove a family version of Theorem \ref{theorem-A}, namely the following theorem.
\begin{theorem}[Existence, family version]\label{theo-family-version}
Let $\{\varphi_{t}\}_{t\in[0,\,+\infty]}$ be a one-parameter family of smooth functions on $C_0$ varying smoothly in $t$ such that $\mathcal{L}_{X}\varphi_{t}=\mathcal{L}_{JX}\varphi_{t}=0$ with each member of the family $\{\omega_{0}^{t}=\frac{i}{2}\partial\bar{\partial}(r^{2}e^{2\varphi_{t}})\}_{t\,\in\,[0,\,+\infty]}$ positive-definite on $C_0$ and such that $\norm{\varphi_{t}}_{C^{k}}$ is uniformly bounded in $t$ for all $k\geq 0$. Moreover, suppose that $M$ is simply connected or that $X|_{A}=0$ for $A\subset E$ for which $H_{1}(A)\to H_{1}(E)$ is surjective.

Then for all $c>0$, there exists a compact subset $K\subset M$ and a one-parameter family of expanding gradient K\"ahler-Ricci solitons $\{g^{t}_{c}\}_{t\in[0,\,+\infty]}$ varying smoothly in $t$ with $\mathcal{L}_{JX}g^{t}_{c}=0$ such that on $M\setminus K$,
\begin{equation*}
|(\nabla^{g^{t}_{0}})^k(\pi_{*}g^{t}_{c}-cg^{t}_{0}-\operatorname{Ric}(g^{t}_{0}))|_{g^{t}_{0}} \leq C(k)r_{t}^{-4-k}\qquad\textrm{for all $k\in\mathbb{N}_{0}$}.
\end{equation*}
Here, $g^{t}_{0}$ denotes the K\"ahler cone metric associated to $\omega^{t}_{0}$, $r_{t}:=re^{\varphi_{t}}$ denotes the corresponding radius function, and $\operatorname{Ric}(g^{t}_{0})$ denotes the Ricci curvature of $g^{t}_{0}$.
\end{theorem}

\noindent Notice that we have the condition $\pi_{1}(M)=0$ here in place of the cohomological vanishing conditions of hypothesis (b) of Theorem A; see Remark \ref{fini} below for why this is the case.

The proof of Theorem \ref{theo-family-version} will involve several steps. We begin with a family version of the construction of a background metric. By inspecting the proof of Proposition \ref{background-metric}, one can deduce the following.
\begin{prop}\label{background-metric-family}
Let $\{\varphi_{t}\}_{t\in[0,\,+\infty]}$ be a one-parameter family of smooth functions on $C_0$ varying smoothly in $t$ such that $\mathcal{L}_{X}\varphi_{t}=\mathcal{L}_{JX}\varphi_{t}=0$ with each member of the family $\{\omega_{0}^{t}=\frac{i}{2}\partial\bar{\partial}(r^{2}e^{2\varphi_{t}})\}_{t\,\in\,[0,\,+\infty]}$ positive-definite and such that $\norm{\varphi_{t}}_{C^{k}}$ is uniformly bounded in $t$ for all $k\geq 0$.

Then for all $c>0$, there exists a compact subset $K\subset M$ and a one-parameter family of smooth functions $\{u^{t}_{c}\}_{t\in[0,\,+\infty]}\subset C^{\infty}(M)$ depending smoothly on $t$ with $\mathcal{L}_{JX}u^{t}_{c}=0$ such that \linebreak $\omega^{t}_{c}:=i\widetilde{\Theta_{h}}+i\partial\bar{\partial}u^{t}_{c}$ is a K\"ahler form satisfying
$\omega^{t}_{c}=\pi^{*}(c\omega^{t}_{0}-\rho_{\omega^{t}_{0}})$ on $M\setminus K$, where $\rho_{\omega^{t}_{0}}$ denotes the Ricci form of $\omega^{t}_{0}$. In particular, we have that $\mathcal{L}_{JX}\omega_{c}^{t}=0$.
\end{prop}

Next, we state and prove the family version of Proposition \ref{equationsetup}.
\begin{prop}\label{prop-set-up-equ-family}
In the setting of Proposition \ref{background-metric-family}, suppose that $M$ in addition is simply connected or that
$X|_{A}=0$ for $A\subset E$ for which $H_{1}(A)\to H_{1}(E)$ is surjective. Then for all $c>0$, there exists a one-parameter family of smooth functions $\{F_{c}^{t}:M\to\R\}_{t\,\in\,[0,\,+\infty]}\subset C_{-2}^{\infty}(M)$ varying smoothly in $t$ with
$\mathcal{L}_{JX}F_{c}^{t}=0$ and with $F_{c}^{t}=-\log\frac{(c\omega^{t}_{0}-\rho_{\omega^{t}_{0}})^{n}}{(\omega^{t}_{0})^{n}}$ on the complement of $K\subset M$ such that $$\rho_{\omega_{c}^{t}}+\omega_{c}^{t}-\frac{1}{2}\mathcal{L}_{X}\omega_{c}^{t}=i\partial\bar{\partial}F_{c}^{t}.$$
Here, $\rho_{\omega_{c}^{t}}$ denotes the Ricci form of $\omega_{c}^{t}$.
\end{prop}

\begin{remark}\label{fini}
The reason why we have to replace the vanishing of $H^1(M)$ or $H^{0,1}(M)$ as in Proposition \ref{equationsetup}
with the condition of simple connectedness of the resolution $M$ here is because we require an explicit primitive for every one-form on $M$ in order to see that the primitive depends smoothly on the parameter $t$ if the one-form itself depends smoothly on this parameter, something which is not clear from the aforementioned cohomological vanishing conditions.
\end{remark}

\begin{proof}
Fix $c>0$ and note that since $\mathcal{L}_{JX}\omega_{c}^{t}=0$, the dual one-form $\eta_{X}^{t}$  of $X$ with respect to the K\"ahler metric associated to $\omega_{c}^{t}$ is closed. Our hypotheses then imply that
\begin{equation}\label{miami}
 \mathcal{L}_{X}\omega_{c}^{t}=d(\omega_{c}^{t}\lrcorner X)=i\partial\bar{\partial}\theta_{X}^{t},
\end{equation}  
where $\theta^{t}_{X}(x)=\int^{x}_{x_{0}}\eta^{t}_{X}$, the integral of $\eta^{t}_{X}$ along any curve connecting a fixed point $x_{0}\in M$ to $x\in M$. When $\pi_{1}(M)=0$, this last assertion is clear. In the case that $X|_{A}=0$ for $A\subset E$ for which $H_{1}(A)\to H_{1}(E)$ is surjective, \eqref{miami} follows from Lemma \ref{simple} after noting that $M$ is homotopy equivalent to $E$ and $\eta_{X}^{t}|_{A}=0$. By averaging over the torus action on $M$ induced from that on $C_{0}$, we may assume that $\mathcal{L}_{JX}\theta_{X}^{t}=0$.

Next note that
$$\rho_{\omega_{c}^{t}}+i\Theta_{h}=i\partial\bar{\partial}v^{t}$$
for a smooth function $v^{t}$ that varies smoothly in $t$. Averaging again over the induced torus action on $M$ then yields
$$\rho_{\omega_{c}^{t}}+i\widetilde{\Theta_{h}}=i\partial\bar{\partial}\tilde{v}^{t}$$
for $\tilde{v}^{t}$ a family of smooth functions varying smoothly in $t$ with $\mathcal{L}_{JX}\tilde{v}^{t}=0$. We then have that
\begin{equation*}
\begin{split}
\rho_{\omega_{c}^{t}}+\omega_{c}^{t}-\frac{1}{2}\mathcal{L}_{X}\omega_{c}^{t}&
=\rho_{\omega_{c}^{t}}+i\widetilde{\Theta_{h}}+i\partial\bar{\partial}u^{t}_{c}-i\partial\bar{\partial}\theta^{t}_{X}\\
&=i\partial\bar{\partial}(\tilde{v}^{t}+u^{t}_{c}-\theta_{X}^{t})\\
&=i\partial\bar{\partial}F_{c}^{t}.
\end{split}
\end{equation*}
Clearly $F_{c}^{t}$ varies smoothly in $t$ and $\mathcal{L}_{JX}F^{t}_{c}=0$.

As in the proof of Proposition \ref{background-metric}, we see that outside $K$ we can write
$$F^{t}_{c}=f(t)-\log\frac{(c\omega_{t}-\rho_{\omega_{t}})^{n}}{(c\omega_{t})^{n}}$$
for some $f:\mathbb{R}\to\mathbb{R}$. Since $F^{t}_{c}$ varies smoothly in $t$, $f(t)$ must also be smooth as a function of $t$. We then redefine $F^{t}_{c}$ by $$F^{t}_{c}-f(t).$$ This function has the desired properties.
\end{proof}

We are now in a position to prove Theorem \ref{theo-family-version}.

\subsubsection*{Proof of Theorem \ref{theo-family-version}}
The proof is verbatim the proof of Theorem \ref{theorem-A}. Thanks to Proposition \ref{prop-set-up-equ-family}, it suffices to prove the family version of Theorem \ref{t:existence}. We adopt the assumptions and notations of Theorem \ref{theo-family-version} and Proposition \ref{prop-set-up-equ-family} and we fix $c>0$, although we shall omit the subscript $c$ for ease of notation. Theorem \ref{theo-family-version} will follow from the next claim.
\begin{claim}\label{claim-family-version-thm-D}
There exists a smooth one-parameter family of K\"ahler potentials $(\varphi_t)_{t\geq 0}$ satisfying
\begin{equation}
\left\{
\begin{array}{rl}
&\omega_{\varphi_t}:=\omega^t+i\partial\bar{\partial}\varphi_t>0,\quad\varphi_t\in \mathcal{D}^{\infty}_{\widetilde{r_t}^2,X}(M),\quad t\geq 0,\\
&\\
&-\varphi_t+\log\frac{(\omega^t+i\partial\bar{\partial}\varphi_t)^{n}}{(\omega^t)^{n}}+\frac{1}{2}X\cdot\varphi_t=F^t\label{MA-Family-Version},
\end{array} \right.
\end{equation}
where $\widetilde{r_t}$ denotes any positive extension of the radius function $r_t$ on $C_{0}$ to the resolution $M$.
\end{claim}

\begin{proof}[Proof of Claim \ref{claim-family-version-thm-D}]
First note that the function spaces $\mathcal{D}^{\infty}_{\widetilde{r_t}^2,X}(M)$ do not depend on the parameter $t$ since all of the background metrics furnished by Proposition \ref{background-metric-family} are equivalent.

The proof of this claim is based on a continuity method. Let
\begin{eqnarray*}
S_{Family}:=\{t\geq 0: \mbox{$(\ref{MA-Family-Version})_t$ has a solution}\}.
\end{eqnarray*}
Then $S_{Family}$ is non-empty; indeed, by Theorem \ref{t:existence}, $0\in S_{Family}$. Moreover, openness of $S_{Family}$ is proven by arguing as in the proof of Theorem \ref{t:existence}; if $t_0\in S_{Family}$, then there exists some $\epsilon>0$ such that $(\ref{MA-Family-Version})_{t_0+\epsilon}$ has a solution by the inverse function theorem established in Theorem \ref{iso-sch-Laplacian}. The smoothness in $t$ follows from Theorem \ref{iso-sch-Laplacian}. Finally, $S_{Family}$ is closed for the same reasons as in the proof of Theorem \ref{t:existence}; the adaptation of Theorem \ref{theo-C^3-wei-est} to the (compact) family of K\"ahler forms $(\omega^t)_{t\geq 0}$ provided by Proposition \ref{background-metric-family} is straightforward.
\end{proof}

Finally, we prove Theorem \ref{coro-F}.
\subsubsection*{Proof of Theorem \ref{coro-F}}
We view $\pi:\mathbb{C}^{n}\to\mathcal{O}_{\mathbb{P}^{n-1}}(-1)^{\times}$ trivially as an equivariant resolution with $\pi=\operatorname{Id}$. We have a K\"ahler cone metric $g_{0}$ on $\mathbb{C}^{n}$ with radial function $r$ whose K\"ahler form we shall denote by $\omega_{0}=\frac{i}{2}\partial\bar{\partial}r^{2}$. Denote by $\sigma$ the transverse K\"ahler form of $\omega_{0}$ on the complex base $D=\mathbb{P}^{n-1}$ of the cone and recall that, by assumption, $a\cdot r\partial_{r}$ is equal to the Euler vector field on $\mathbb{C}^{n}$ for some $0<a<1$. We invoke the result announced by Perelman and proved in \cite{Tia-Zhu-Con-Kah-Ric} that states that the normalised K\"ahler-Ricci flow $(\hat{\sigma}(t))_{t\geq 0}$ converges smoothly to a K\"ahler-Einstein metric with Einstein constant equal to $1$ starting from any K\"ahler form $\hat{\sigma}(0)=\hat{\sigma}$ with $[\hat{\sigma}]\in c_1(D)$. As in Example \ref{example-Kahler-Ricci-Type-II}, this yields a one-parameter family of potentials $\varphi(t)\in C^{\infty}(D)$ on $D$ with $\varphi(0)=0$
satisfying \eqref{krf} with initial metric $\hat{\sigma}=\frac{n}{a}\sigma\in c_{1}(D)$ such that $\varphi(t)$ converges smoothly to a smooth real-valued function $\varphi(\infty)$ on $D$, and such that the corresponding K\"ahler metrics $\sigma(t):=\sigma+i\partial\bar{\partial}(\frac{a}{n}\varphi(t))$ converge smoothly to $\sigma(\infty)=a\omega_{FS}=\sigma+i\partial\bar{\partial}(\frac{a}{n}\varphi(\infty))$ as $t\to\infty$, where $\omega_{FS}$ denotes the Fubini-Study form on $\mathbb{P}^{n-1}$. The induced evolution $\omega_{0}(t)$ of $\omega_{0}$ is then via $\omega_{0}(t)=\frac{i}{2}\partial\bar{\partial}(r^{2}e^{\frac{2a}{n}p^{*}\varphi(t)})$, where
$p:\mathcal{O}_{\mathbb{P}^{n-1}}(-1)\to\mathbb{P}^{n-1}$ denotes the projection. It is not hard to see that the limit cone is given by
$(\mathbb{C}^{n},\,\omega_{0}(\infty))=(\mathbb{C}^n,\Re(\partial\bar{\partial}|\cdot|^{2a}))$. (Note that $a<1$ here since we discard the flat case by assumption.) We next apply Theorem \ref{theo-family-version} to the one-parameter family of K\"ahler cones $(\mathbb{C}^{n},\,\omega_{0}(t))_{t\in[0,\,+\infty]}$ to obtain, for each $c>0$, a smooth one-parameter family of expanding gradient K\"ahler-Ricci solitons $(\mathbb{C}^{n},\omega_{c}^t,r\partial_{r})_{t\in[0,+\infty]}$, where $(\mathbb{C}^{n},\omega_{c}^{\infty})$ is asymptotic to $(\mathbb{C}^n,c\Re(\partial\bar{\partial}|\cdot|^{2a}))$. By uniqueness (see \cite{Cho-Fon} or Theorem \ref{Uniqueness-Theorem}(ii)), $(\mathbb{C}^{n},\,\omega_{c}^{\infty},\,r\partial_{r})$ is Cao's expanding gradient K\"ahler-Ricci soliton. Recall that this soliton is $U(n)$-invariant and has positive curvature operator on real $(1,1)$-forms (cf.~\cite{Che-Zhu-Pos-Cur} regarding this latter point). What must be checked is that the positivity of the curvature operator on real $(1,1)$-forms is preserved along the one-parameter path $(\mathbb{C}^{n},\omega_{c}^t,r\partial_{r})_{t\in[0,+\infty]}$. We use a continuity method to prove this. In what follows, we shall omit the subscript $c$.

First note that by construction, $\sigma(t)$ is the transverse K\"ahler form of $\omega_{0}(t)$ on $D$ and that $\frac{n}{a}\sigma(t)$
evolves along the normalised K\"ahler-Ricci flow on $D$. Since $\sigma$, and hence $\frac{n}{a}\sigma$, has positive curvature operator on real $(1,1)$-forms by assumption, every member of the one-parameter family $(\frac{n}{a}\sigma(t))_{t\geq 0}$ has in fact positive curvature operator on real $(1,1)$-forms as well; see \cite[Theorem 7.34]{Cho-Lu-Ni-I}. Consequently, each member of the family $(\sigma(t))_{t\geq 0}$ will have positive curvature operator on real $(1,1)$-forms.
Furthermore, since the Reeb field on the asymptotic cone is fixed along the path $(\omega_{0}(t))_{t\geq0}$, the sectional curvature $K$ of $\omega_{0}(t)$ for each $t$ in the directions $(J_0(r\partial_r),U)$, where $U$ is tangent to $D$, is given by $K(J_0(r\partial_r),U)=1$ when restricted to the slice $\{r=1\}$ of the cone. In particular, if $(\mathbb{C}^{n},\,\omega^{t_0},\,r\partial{r})$ has positive curvature operator on real $(1,1)$-forms for some $t_0\geq 0$, then, since the corresponding K\"ahler metrics $g^t$ are asymptotic to $g_0^t$ at rate $-2$ with $g_{0}^{t}$-derivatives with respect to $\pi$ for any $t$, there exists some neighbourhood $\mathcal{U}_{t_0}$ of $t_0$ such that the curvature operator is positive on real $(1,1)$-forms with no components involving the radial direction $X$ for all $t\in\mathcal{U}_{t_0}$. Hence, to complete the openness argument, it suffices to study the positivity of the curvature operator on $2$-planes containing the radial direction $X$. We argue precisely as in \cite[Proposition 3.13]{Der-Asy-Com-Egs}. For the convenience of the reader, we sketch a proof. We begin with the following claim.
\begin{claim}\label{claim-pos-cur-op}
In the above setting,
\begin{eqnarray*}
\liminf_{r\rightarrow +\infty}r^2\inf_{\partial B_{g^t}(p,r)}\min\{\Rm(g^t)(X,U,U,X): U\perp X\quad\textrm{and}\quad|U|_{g^t}=1\}>0
\end{eqnarray*}
for all $t\geq 0$, where $g^t$ is the metric associated to the K\"ahler form $\omega^t$.
\end{claim}

\begin{proof}[Proof of Claim \ref{claim-pos-cur-op}]
For the sake of clarity, we fix once and for all an expanding gradient K\"ahler-Ricci soliton $(M,\,\omega,\,X)$ with associated Riemannian metric $g$ that is asymptotically conical K\"ahler with tangent cone a regular K\"ahler cone with cone metric having positive curvature operator on real $(1,1)$-forms on the complex base $D$ of the cone.

On one hand, we see from (\ref{equ:4}) that
\begin{eqnarray*}
\Rm(g)(X,U,V,X)=-2\Div_{g}\Rm(g)(X, U, V)=-2\nabla^g_{X}\Ric(g)(U,V)+2\nabla^g_U\Ric(g)\left(X,V\right),
\end{eqnarray*}
for any vectors $U$ and $V$. On the other hand, by the soliton identities given by Lemma \ref{id-EGS}, we have that
\begin{eqnarray*}
\nabla^g_U\Ric(g)\left(X,V\right)&=&U\cdot\Ric(g)(X,V)-\Ric(g)(\nabla^g_UX,V)-\Ric(g)(X,\nabla^g_UV)\\
&=&\langle\nabla^g_U(\Ric(g)(X)),V\rangle-\Ric(g)\left(U+\Ric(g)(U),V\right)\\
&=&-\nabla^{g,2}s_{\omega}(U,V)-\Ric(g)\left(U+\Ric(g)(U),V\right).
\end{eqnarray*}
Therefore, since the norm of the curvature decays at infinity, we have that
\begin{eqnarray*}
\Rm(g)(X,U,U,X)&=&4\left(\Delta_{\omega}\Ric(g)+\Ric(g)+\Rm(g)\ast\Ric(g)\right)(U,U)\\
&&-2\left(\nabla^2s_{\omega}(U,V)+\Ric(g)+\Ric(g)\otimes\Ric(g)\right)(U,U),\\
&=&2\Ric(g)(U,U)+\textit{O}(r^{-4})|U|^2_g.
\end{eqnarray*}
Now, the positivity of the curvature operator of the metric on real $(1,1)$-forms on the complex base $D$ of the cone implies the positivity of the Ricci curvature of the metric of the cone when restricted to vectors orthogonal to the radial direction. In particular, this implies that
\begin{eqnarray*}
\liminf_{r\rightarrow +\infty}r^2\inf_{\partial B_{g}(p,r)}\min\{\Ric(g)(U,U): U\perp X\quad\textrm{and}\quad|U|_{g}=1\}>0,
\end{eqnarray*}
which in turn implies the claim.
\end{proof}
We now prove that the positivity assumption is a closed condition along the path\linebreak $(\mathbb{C}^{n},\omega^t,r\partial_{r})_{t\in[0,+\infty]}$, i.e., that any limit of a sequence $(\mathbb{C}^{n},\omega^{t_i},X)_{i\in\mathbb{N}}$ will have non-negative curvature operator on real $(1,1)$-forms. Since the corresponding curve on the complex base $D$ of the cone has constant volume, and since we begin with a non-flat expanding gradient K\"ahler-Ricci soliton, the Bishop-Gromov theorem implies that the whole curve cannot be flat for every $t\in[0,+\infty]$. In particular, since any limit is a self-similar solution of the Ricci flow, Hamilton's splitting theorem \cite[Theorem 7.34]{Cho-Lu-Ni-I} implies the positivity of the curvature operator on real $(1,1)$-forms of any limit of a sequence $(\mathbb{C}^{n},\omega^{t_i},X)_{i\in\mathbb{N}}$.

\newpage

\appendix

\section{Some technical results}

\subsection{Normal holomorphic coordinates}
A proof of the existence of normal holomorphic coordinates can be found for example in \cite{Bou-Eys-Gue}.
\begin{lemma}\label{hol-coor}
Let $(M^n,\,g)$, $n=\dim_{\mathbb{C}}M$, be a complete K\"ahler manifold with bounded geometry, that is, with positive injectivity radius and bounded curvature. Then there exists $\delta>0$ and $\Lambda>0$ such that for all $x\in M$, there exist holomorphic coordinates $z:=(z_1,...,z_n):B_g(x,\delta)\rightarrow B_{\eucl}(0,\delta)\subset \mathbb{C}^n$ such that with respect to these coordinates, $$\textrm{$g_{i\bar{\jmath}}(x)=\delta_{i\bar{\jmath}}$, $\Lambda^{-1}\delta\leq g\leq \Lambda\delta$ on $B_g(x,\delta)$, and $\Gamma(g)_{ij}^k(x)=0$.}$$ Moreover, if $\varphi$ is a K\"ahler potential, then the holomorphic coordinates can be chosen such that $(g_{\varphi})_{i\bar{\jmath}}(x)=(1+\varphi_{i\bar{\imath}}(x))\delta_{i\bar{\jmath}}.$
\end{lemma}

\subsection{Regularity}
We next establish the following local regularity result for solutions to (\ref{MA}).
\begin{prop}\label{prop-loc-reg}
Let $F\in C^{k,\alpha}_{loc}(M)$ for some $k\geq1$ and $\alpha\in(0,1)$ and let $\varphi\in C^{3,\alpha}_{loc}(M)$ be a solution to (\ref{MA}) with data $F$. Then $\varphi\in C^{k+2,\alpha}_{loc}(M)$.
\end{prop}
\begin{proof}
We prove this proposition by induction on $k\geq 1$. The case $k=1$ is true by assumption, so let $F\in C^{k+1,\alpha}_{loc}(M)$ and let $\varphi\in C^{3,\alpha}_{loc}(M)$ be a solution to (\ref{MA}). Then, by induction, $\varphi\in C^{k+2,\alpha}_{loc}(M)$.
Let $x\in M$ and choose holomorphic coordinates defined on $B_{g}(x,\delta)$ for some $\delta>0$ uniform in $x\in M$ (see Lemma \ref{hol-coor}). Then, since $\varphi$ is a solution to (\ref{MA}) for $F$, i.e.,
 \begin{eqnarray*}\nonumber
F&=&\log\left(\frac{\omega_{\varphi}^n}{\omega^n}\right)+\frac{X}{2}\cdot\varphi-\varphi,\\
\end{eqnarray*}
we know that the derivative $\partial_j\varphi$ for $j=1,...,n$, satisfies
\begin{eqnarray*}
\Delta_{\omega_{\varphi}}\partial_j\varphi=F_j(\varphi)\in C^{k,\alpha}_{loc}(M).
\end{eqnarray*}
Since the coefficients of $\Delta_{\omega_{\varphi}}$ are in $C^{k,\alpha}_{loc}(M)$, an application of the classical interior Schauder estimates now yields the desired local regularity result, that is, $\partial_j\varphi\in C^{k+2,\alpha}_{loc}(M)$ for any $j=1,...,n$, or equivalently, $\varphi\in C^{k+3,\alpha}_{loc}(M)$.
\end{proof}

\subsection{Properties of Ricci solitons}
The next lemma collects together some well-known Ricci soliton identities and (static) evolution equations satisfied by their curvature tensor.
\begin{lemma}\label{id-EGS}
Let $(M^{2n},\,\omega,\,X=\nabla^gf)$, $n=\dim_{\mathbb{C}}M$, be an expanding gradient K\"ahler-Ricci soliton. Then the trace and first order soliton identities are:
\begin{eqnarray}
&&\Delta_{\omega} f = \frac{s_{\omega}}{2}+n,\nonumber\\
&&\nabla^g s_{\omega}+ \Ric(g)(X)=0, \nonumber \\
&&\arrowvert \nabla^g f \arrowvert^2+\frac{1}{2}s_{\omega}-f=\operatorname{const.},\nonumber\\
&&2\Div_g\Rm(g)(Y,Z,T)=\Rm(g)(Y,Z, X,T),\label{equ:4}
\end{eqnarray}
for any vector fields $Y$, $Z$, and $T$, where $\arrowvert \nabla^g f \arrowvert^2:=g^{i\bar{\jmath}}\partial_if\partial_{\bar{\jmath}}f$. In particular, if $s_{\omega}$ is bounded, then $f$ grows at most quadratically and $\nabla^g f$ grows at most linearly.

The evolution equations for the curvature operator, the Ricci curvature, and the scalar curvature are via
\begin{eqnarray}
&& \Delta_{\omega}\Rm(g)+\nabla^g_{\frac{X}{2}}\Rm(g)+\Rm(g)+\Rm(g)\ast\Rm(g)=0,\nonumber\\
&&\Delta_{\omega}\Ric(g)+\nabla^g_{\frac{X}{2}}\Ric(g)+\Ric(g)+\Rm(g)\ast\Ric(g)=0,\nonumber\\
&&\Delta_{\omega}s_{\omega}+\nabla^g_{\frac{X}{2}}s_{\omega}+s_{\omega}+2\arrowvert\Ric(g)\arrowvert_{\omega}^2=0,\nonumber
\end{eqnarray}
where, for any two tensors $A$ and $B$, $A\ast B$ denotes any linear combination of contractions of the tensorial product of $A$ and $B$, and where $\arrowvert\Ric(g)\arrowvert_{\omega}:=g^{i\bar{l}}g^{k\bar{\jmath}}\Ric(g)_{k\bar{l}}\Ric(g)_{i\bar{\jmath}}$.
\end{lemma}

\begin{proof}
See \cite[Chapter 1]{Cho-Lu-Ni-I}.
\end{proof}

We now make a remark on our conventions.
\begin{convention}
Our convention is to normalize expanding gradient Ricci solitons $(M,\,g,\,X)$ such that their potential function $f:M\to\mathbb{R}$ satisfies $\arrowvert \nabla^g f \arrowvert^2+\frac{s_{g}}{2}=f.$ Here, $s_{g}$ denotes the scalar curvature of $g$.
\end{convention}

One can say quite a lot about expanding gradient Ricci solitons with constraints on the Ricci curvature. For example, we have:
\begin{prop}\label{pot-fct-est}
Let $(M^n,g,\nabla^g f)$ be a normalised expanding gradient Ricci soliton of real dimension $n$.
\begin{enumerate}
\item If $\Ric(g)\geq 0$, then $M^n$ is diffeomorphic to $\mathbb{R}^n$. Moreover, the following estimates hold:
\begin{eqnarray}
&&\frac{1}{4}r_p(x)^2+\min_{M}f\leq f(x)\leq\left(\frac{1}{2}r_p(x)+\sqrt{\min_{M}f}\right)^2,\quad \forall x\in M,\nonumber\\
&&\AVR(g):=\lim_{r\rightarrow+\infty}\frac{{{\vol}}_g B_g(q,r)}{r^n}>0,\quad\forall q\in M,\nonumber\\
&&0\leq s_g\leq S_0<+\infty,\nonumber
\end{eqnarray}
where $p\in M$ is the unique critical point of $f$.\\

\item If $\Ric(g)=\textit{O}(r_p^{-2})$, where $r_p$ denotes the distance function to a fixed point $p\in M$, then the potential function is equivalent to $r_p^2/4$ (up to order $2$).
\end{enumerate}
\end{prop}

\begin{proof}
See \cite{Der-Asy-Com-Egs} and the references therein.
\end{proof}

\subsection{Sufficient conditions for a vector field $X$ to be gradient}
We begin with the following general fact.
\begin{lemma}\label{lemmma}
Let $M$ be a K\"ahler manifold with K\"ahler metric $g$ and complex structure $J$ and let $X$ be a real holomorphic vector field on $M$. Then the following are equivalent.
\begin{enumerate}[label=\textnormal{(\roman{*})}, ref=(\roman{*})]
\item $JX$ is Killing.
\item $g(\nabla^{g}_{Y}X,\,Z)=g(Y,\,\nabla^{g}_{Z}X)$ for all real vector fields $Y,\,Z$ on $M$.
\item The $g$-dual one-form of $X$ is closed.
\end{enumerate}
\end{lemma}

\begin{proof}
Since $X$ is holomorphic, we have that $$\nabla^{g}_{JY}X=J\nabla^{g}_{Y}X$$ for every real vector field $Y$ on $M$. Hence, for every real vector field $Y$ and $Z$ on $M$, we see that
\begin{equation*}
\begin{split}
(\mathcal{L}_{JX}g)(Y,\,Z)=g(\nabla^{g}_{Y}(JX),\,Z)+g(Y,\,\nabla^{g}_{Z}(JX))&=g(J\nabla^{g}_{Y}X,\,Z)+g(Y,\,J\nabla^{g}_{Z}X)\\
&=g(\nabla^{g}_{JY}X,\,Z)+g(Y,\,J\nabla^{g}_{Z}X)\\
&=g(\nabla^{g}_{JY}X,\,Z)-g(JY,\,\nabla^{g}_{Z}X).
\end{split}
\end{equation*}
The equivalence of (i) and (ii) now follows.

The equivalence of (ii) and (iii) can be seen from the identity
\begin{equation*}
\begin{split}
d\eta_{X}(Y,\,Z)&=2g(\nabla^{g}_{Y}X,\,Z)-(\mathcal{L}_{X}g)(Y,\,Z)\\
&=2g(\nabla^{g}_{Y}X,\,Z)-(g(\nabla^{g}_{Y}X,\,Z)+g(Y,\,\nabla^{g}_{Z}X))\\
&=g(\nabla^{g}_{Y}X,\,Z)-g(Y,\,\nabla^{g}_{Z}X)
\end{split}
\end{equation*}
for every real vector field $Y$ and $Z$ on $M$, where $\eta_{X}$ denotes the $g$-dual one-form of $X$.
\end{proof}

Using this, we can prove:
\begin{corollary}\label{howdy}
Let $M$ be a K\"ahler manifold with K\"ahler metric $g$ and complex structure $J$ and let $X$ be a real holomorphic vector field on $M$.
\begin{enumerate}[label=\textnormal{(\roman{*})}, ref=(\roman{*})]
\item If $H^{1}(M)=0$ or $H^{0,\,1}(M)=0$ and $JX$ is Killing, then there exists a smooth real-valued function $f\in C^{\infty}(M)$ such that $X=\nabla^{g}f$.
\item Conversely, if $X=\nabla^{g}f$ for a smooth real-valued function $f\in C^{\infty}(M)$, then $JX$ is Killing.
\end{enumerate}
\end{corollary}

\begin{proof}
\begin{enumerate}[label=\textnormal{(\roman{*})}, ref=(\roman{*})]
\item Since $JX$ is Killing, the $g$-dual one-form $\eta_{X}$ of $X$ is closed by Lemma \ref{lemmma}. In the case that $H^{1}(N)=0$, we can write $\eta_{X}=d\theta_{X}$ for some real-valued smooth function $\theta_{X}$ on $N$, from which the result follows. A similar argument applies when $H^{0,\,1}(N)=0$.
 \item If $X=\nabla^{g}f$, then $\nabla^{g}X$, being the Hessian of $f$, is symmetric. The result then follows from Lemma \ref{lemmma}.
\end{enumerate}
\end{proof}

We next make the following simple observation.
\begin{lemma}\label{simple}
Let $M$ be a manifold and let $\alpha$ be a closed one-form on $M$. If $M$ is homotopy equivalent to a compact subset $E\subset M$ and $\alpha|_{A}=0$ for a subset $A\subset E$ for which $H_{1}(A)\to H_{1}(E)$ is surjective, then $\alpha=df$ for $f(x)=\int_{x_{0}}^{x}\alpha$ the integral of $\alpha$ along any curve connecting a fixed point $x_{0}\in M$ to $x\in M$.
\end{lemma}

\begin{proof}
Since $M$ is homotopy equivalent to $E$, every closed loop in $M$ is homotopy equivalent to a loop in $E$. Moreover, $\alpha$ being closed implies that its integral over any closed loop will only depend on the real homology class of the loop. The surjectivity of $H_{1}(A)\to H_{1}(E)$, together with the fact that  $\alpha|_{A}=0$, therefore implies that the integral of $\alpha$ over every closed loop is equal to zero. The lemma now follows.
\end{proof}

From this, we deduce:
\begin{corollary}\label{kingpin}
Let $M$ be a K\"ahler manifold with K\"ahler metric $g$ and complex structure $J$ and let $X$ be a real holomorphic vector field on $M$ for which $JX$ is Killing. If $M$ is homotopy equivalent to a compact subset $E\subset M$ and $X|_{A}=0$ for a subset $A\subset E$ for which $H_{1}(A)\to H_{1}(E)$ is surjective, then there exists a smooth real-valued function $f\in C^{\infty}(M)$ such that $X=\nabla^{g}f$.
\end{corollary}

\begin{proof}
Simply apply Lemma \ref{simple} using the fact that the $g$-dual one-form $\eta_{X}$ of $X$ is closed by Lemma \ref{lemmma} and $\eta_{X}|_{A}=0$ because $X|_{A}=0$ by assumption.
\end{proof}

\subsection{A vanishing theorem}
Finally, we note the following vanishing result.
\begin{prop}\label{vanishinging}
Let $\pi:M\to C$ be a resolution of a complex cone with a rational singularity. Then $H^{0,\,1}(M)=0$.
\end{prop}

\begin{proof}
By Oka's coherence theorem and Grauert's direct image theorem, the sheaves $R^{q}\pi_{*}\mathcal{O}_{M}$ are coherent analytic sheaves on $C_{0}$ for $q\geq 0$. Thus, since $C_{0}$ is a Stein space (cf.~\cite[Theorem 1.8]{Conlon}), by Cartan's Theorem B, we deduce that $$H^{p}(X_{0},\,R^{q}\pi_{*}\mathcal{O}_{X})=0 \quad\textrm{for all $p\geq 1$ and $q\geq 0$}.$$
Consider next the Leray spectral sequence \cite[Theorem 4.17.1, p.201]{Godement}
\begin{equation*}
E_{2}^{p,q}:=H^{p}(X_{0},\,R^{q}\pi_{*}\mathcal{O}_{X})\Rightarrow H^{p+q}(M,\,\mathcal{O}_{M})
\end{equation*}
and form its exact sequence of terms of low degree \cite[Theorem 4.5.1, p.82]{Godement}
\begin{equation*}
0\longrightarrow H^{1}(X_{0},\,\pi_{*}\mathcal{O}_{X})\longrightarrow H^{1}(X,\,\mathcal{O}_{X})
\longrightarrow H^{0}(X_{0},\,R^{1}\pi_{*}\mathcal{O}_{X})\longrightarrow H^{2}(X_{0},\,\pi_{*}\mathcal{O}_{X})\longrightarrow H^{2}(X,\,\mathcal{O}_{X}).
\end{equation*}
Since $R^{1}\pi_{*}\mathcal{O}_{M}=0$ and $H^{1}(C,\,\pi_{*}\mathcal{O}_{M})=0$, we find from this sequence that $H^{1}(M,\,\mathcal{O}_{M})=0$, as claimed.
\end{proof}

\newpage

\bibliographystyle{amsalpha}

\bibliography{ref2}

\end{document}